\documentclass[10pt]{article}

\usepackage{fullpage, amsmath}
\usepackage{amsthm,amsfonts,url,amssymb}
\usepackage{mathrsfs}
\usepackage{amscd}
\usepackage{color}
\usepackage{chemarrow}
\usepackage{pictexwd,dcpic}
\usepackage[textwidth=18cm,centering]{geometry}
\usepackage{extarrows}
\usepackage{enumitem}
\usepackage{lipsum}
\usepackage[all,cmtip]{xy}
\usepackage{leftidx}
\usepackage[version=3]{mhchem}
\usepackage{makeidx}
\usepackage{bm} 
\usepackage{stmaryrd}
\usepackage[linkcolor=red,anchorcolor=green,citecolor=blue]{hyperref}

\makeindex
\newtheorem{thm}{Theorem}[section]
\newtheorem{cor}[thm]{Corollary}
\newtheorem{lem}[thm]{Lemma}
\newtheorem{pro}[thm]{Proposition}
\newtheorem{dfn}[thm]{Definition}
\newtheorem{hypothesis}[thm]{Hypothesis}
\newtheorem{rmk}[thm]{Remark}

\newtheorem{exmp}[thm]{Example}
\newtheorem{conjecture}[thm]{Conjecture}

\newcommand{\hooklongrightarrow}{\lhook\joinrel\longrightarrow}

\DeclareMathOperator{\Gal}{Gal}
\DeclareMathOperator{\GL}{GL}

\DeclareMathOperator{\Hom}{Hom}

\DeclareMathOperator{\Spec}{Spec}
\DeclareMathOperator{\Spf}{Spf}

\newcommand{\pdr}{{\rm pdR}}

\newcommand{\ccyc}{{\epsilon}}

\DeclareMathOperator{\coker}{coker}

\DeclareMathOperator{\diag}{diag}

\renewcommand{\mod}{{\rm\,mod\,}}

\DeclareMathOperator{\rank}{rank}

\newcommand{\rig}{{\rm rig}}
\DeclareMathOperator{\rk}{rk}

\newcommand{\gal}{{\rm Gal}}
\newcommand{\pcr}{{\rm pcr}}

\newcommand{\fil}{{\rm Fil}}

\newcommand{\ind}{{\rm Ind}}
\newcommand{\homo}{{\rm Hom}}
\newcommand{\EndO}{{\rm End}}
\newcommand{\ext}{{\rm Ext}}
\newcommand{\GLN}{{\rm GL}}

\newcommand{\ana}{{\rm an}}

\newcommand{\op}{{\overline{{\mathbf{P}}}}}

\newcommand{\ob}{{\overline{{\mathbf{B}}}}}

\newcommand{\unr}{{\rm unr}}

\newcommand{\hH}{{\mathrm{H}}}

\newcommand{\ra}{{\longrightarrow}}
\DeclareMathOperator{\val}{val}

\newcommand{\soc}{{\rm soc}}

\newcommand{\sm}{{\rm sm}}

\newcommand{\spf}{{\rm Spf}}
\newcommand{\df}{{\mathrm{DF}}}
\newcommand{\wdre}{{\textbf{r}}}

\newcommand{\rec}{{\rm rec}_L}
\newcommand{\lalg}{{\rm lalg}}

\newcommand{\Art}{{\rm Art}}

\newcommand{\wt}{{\rm wt}}

\newcommand{\dr}{{\rm dR}}

\newcommand{\res}{{\rm Res}}
\newcommand{\tee}{{{\otimes}_{\cR_{E,L}}}}

\newcommand{\undelram}{{\delta^0}}

\newcommand{\gr}{{\rm gr}}
\newcommand{\hpi}{{{\mathbf{h}}}}
\newcommand{\Dpik}{{\mathbf{D}}}

\newcommand{\sbanpik}{{\big(\mathrm{Spec}\hspace{2pt}\FZ_{\Omega_{S_0}}\big)^{\mathrm{rig}}}}

\newcommand{\sbanpiku}{{\big(\mathrm{Spec}\hspace{2pt}\FZ_{\Omega_{S^u_0}}\big)^{\mathrm{rig}}}}

\newcommand{\rigchu}{{\mathcal{Z}_{\bL_{S^u_0},\mathcal{O}_L}}}
\newcommand{\rigchlu}{{\mathcal{Z}_{\bL_{S^u_0},L}}}

\newcommand{\univ}{{\rm univ}}
\newcommand{\loc}{{\rm loc}}

\newcommand{\gen}{{\rm gen}}

\newcommand{\ul}{\underline}
\newcommand{\omepik}{\Omega_{S^u_0}}

\newcommand{\bersteineigenvarpik}{{\mathcal{E}_{\omepik,\fp,\bm{\lambda}_\bh}^\infty(\overline{\rho})}}

\newcommand{\defvarring}{{R_{\overline{r}}^\Box}}

\newcommand{\defvaru}{{X^\Box_{\omepik,\mathbf{{h}}}(\overline{r})}}
\newcommand{\defvarrho}{{X^\Box_{\omepik,\mathbf{{h}}}(\overline{r})}}

\usepackage{amsfonts}
\newcommand{\BOne} {{\mathchoice{\hbox{\rm1\kern-2.7pt l\kern.9pt}}
		{\hbox{\rm1\kern-2.7pt l\kern.9pt}}
		{\hbox{\scriptsize\rm1\kern-2.3pt l\kern.4pt}}
		{\hbox{\scriptsize\rm1\kern-2.4pt l\kern.5pt}}}}

\newcommand{\BA}{{\mathbb{A}}}

\newcommand{\BG}{{\mathbb{G}}}

\newcommand{\BQ}{{\mathbb{Q}}}

\newcommand{\BU}{{\mathbb{U}}}

\newcommand{\BZ}{{\mathbb{Z}}}

\newcommand{\bB}{{\mathbf{B}}}

\newcommand{\be}{{\mathbf{e}}}
\newcommand{\bF}{{\mathbf{F}}}

\newcommand{\bI}{{\mathbf{I}}}

\newcommand{\bL}{{\mathbf{L}}}

\newcommand{\bN}{{\mathbf{N}}}

\newcommand{\bP}{{\mathbf{P}}}
\newcommand{\bQ}{{\mathbf{Q}}}

\newcommand{\bT}{{\mathbf{T}}}

\newcommand{\bW}{{\mathbf{W}}}

\newcommand{\bZ}{{\mathbf{Z}}}

\newcommand{\bh}{{\mathbf{h}}}
\newcommand{\bk}{{\mathbf{k}}}
\newcommand{\bd}{{\mathbf{d}}}

\newcommand{\cL}{\mathcal L}
\newcommand{\co}{\mathcal O}

\newcommand{\cB}{\mathcal B}
\newcommand{\cR}{\mathcal R}

\newcommand{\cS}{\mathcal S}

\newcommand{\cI}{\mathcal I}

\newcommand{\cM}{\mathcal M}
\newcommand{\cF}{\mathcal F}

\newcommand{\cE}{\mathcal E}
\newcommand{\cG}{\mathcal G}
\newcommand{\cJ}{\mathcal J}
\newcommand{\cO}{\mathcal O}

\newcommand{\FX}{{\mathfrak{X}}}

\newcommand{\FZ}{{\mathfrak{Z}}}

\newcommand{\fa}{{\mathfrak{a}}}
\newcommand{\fb}{{\mathfrak{b}}}

\newcommand{\fg}{{\mathfrak{g}}}

\newcommand{\fl}{{\mathfrak{l}}}
\newcommand{\fm}{{\mathfrak{m}}}
\newcommand{\fn}{{\mathfrak{n}}}

\newcommand{\fp}{{\mathfrak{p}}}

\newcommand{\ft}{{\mathfrak{t}}}

\newcommand{\fz}{{\mathfrak{z}}}

\newcommand{\sW}{\mathscr W}
\newcommand{\sL}{\mathscr L}

\begin{document}	
	
\title{\textbf{\textsc{Toward $p$-adic Hodge parameters for potentially crystalline representations of $\GLN_n$}}}

\author{Yiqin He
\thanks{Morningside Center of Mathematics, Chinese Academy of Sciences,\;No. 55, Zhongguancun East Road, Haidian District, Beijing 100190, P.R. China; e-mail address: \texttt{heyiqin@amss.ac.cn}
}}

\date{}
\maketitle

\begin{abstract}
Let $p$ be a prime number,\;$n\geq 2$,\;and let $L$ be a finite extension of $\bQ_p$.\;Let $\rho_L$ be an $n$-dimensional non-critical generic potentially crystalline $p$-adic representation of the absolute Galois group of $L$ with regular Hodge--Tate weights.\;Building on Ding's results and strategy in the crystabelline case \cite{ParaDing2024},\;and on the recent work of Breuil--Ding \cite{BDcritical25},\;we construct an explicit locally analytic representation $\pi_{1}(\rho_L)$ and describe explicitly the Hodge-filtration data of $\rho_L$ that it determines.\;When $\rho_L$ arises from a patched $p$-adic automorphic representation,\;we show, under mild hypotheses, that $\pi_{1}(\rho_L)$ is a subrepresentation of the $\GLN_n(L)$-representation globally associated with $\rho_L$ by using the framework of Bernstein eigenvarieties that were developed by Breuil-Ding \cite{Ding2021}. \;
\end{abstract}

{\hypersetup{linkcolor=black}
\tableofcontents}

\numberwithin{equation}{section}
	
\numberwithin{thm}{section}

\setlength{\baselineskip}{12pt}

\section{Introduction}

This paper aims to extend the discussion of crystabelline (resp.,\;crystalline) $p$-adic Galois representations in \cite{ParaDing2024} (resp.,\;the recent work \cite{BDcritical25}) to potentially crystalline case.\;For simplicity,\;we assume in the introduction that $L=\bQ_p$.\;Let $\rho_p:\gal_{\bQ_p}\rightarrow \GLN_n(E)$ be a de Rham $p$-adic Galois representation,\;where $\gal_{\bQ_p}$ is the absolute Galois group of $\bQ_p$ and $E$ is a finite extension of $\bQ_p$.\;By Fontaine's theory,\;we can attach an $n$-dimensional Weil--Deligne representation $\wdre(\rho_p)$ and thus an irreducible smooth representation $\pi_{\mathrm{sm}}(\rho_p)$ of $\GLN_n(\bQ_p)$ over $E$ (via the classical local Langlands correspondence).\;Assume that $\rho_p$ has regular Hodge--Tate weights $\bh=(\bh_1>\cdots>\bh_n)$.\;Then the locally algebraic representation $\pi_{\lalg}(\rho_p):=\pi_{\mathrm{sm}}(\rho_p)\otimes_EL(\bh-\theta)$ is expected to be the locally algebraic subrepresentation of the conjectural locally analytic representation $\pi_{\ana}(\rho_p)$ via the $p$-adic local Langlands correspondence,\;where $\theta=(0,-1,\cdots,1-n)$ and $L(\bh-\theta)$ is the algebraic representation of $\GLN_n(\bQ_p)$ with highest weight $\bh-\theta$.\;The representation $\pi_{\lalg}(\rho_p)$ only encodes the information in the $F$-semisimplification of $\wdre(\rho_p)$ and $\bh$,\;and therefore misses the Hodge-filtration information of $\rho_p$.\;A basic problem (and starting point) in the locally analytic $p$-adic local Langlands program is to recover the Hodge filtration from locally ${\BQ}_p$-analytic representations of $\GLN_n({\bQ}_p)$.\;

After the pioneering work of Breuil, Colmez, and others, this question was completely settled for the $\GLN_2(\bQ_p)$ case.\;For general $\GLN_n(\bQ_p)$, however, it remained mysterious for a long time.\;Recently, Yiwen Ding made a breakthrough in understanding the $p$-adic Hodge parameters of \textit{non-critical generic} crystabelline representations of $G$, namely representations that become crystalline after restriction to the absolute Galois group of a suitable abelian extension of $\bQ_p$ \cite{ParaDing2024}.\;More precisely, Ding constructed an explicit locally $\bQ_p$-analytic representation $\pi_{1}(\rho_p)$ which determines $\rho_p$ uniquely and has the following form (i.e.,\;$\pi_1(\rho_p)$ is a non-split extension of $\pi_{\lalg}(\rho_p)^{\oplus(2^n-\frac{n(n+1)}{2}-1)}$ by $\pi_1(\wdre(\rho_p),\bh)$):
\[\pi_1(\rho_p)=\left[\pi_1(\wdre(\rho_p),\bh)-\pi_{\lalg}(\rho_p)^{\oplus(2^n-\frac{n(n+1)}{2}-1)}\right],\]
where $\pi_1(\wdre(\rho_p),\bh)$ is obtained from the amalgamated sum of some locally analytic principal series and $\soc_{G}\pi_1(\rho_p)=\pi_{\lalg}(\rho_p)$.\;Moreover, when $\rho_p$ comes from a $p$-adic automorphic representation, $\pi_{1}(\rho_p)$ is a subrepresentation of the globally associated $\GLN_n(\bQ_p)$-representation $\widehat{\pi}(\rho_p)^{\ana}$.\;

The appearance, with large multiplicity, of extra locally algebraic constituents in the cosocle of $\pi_{1}(\rho_p)$ is a new phenomenon.\;Such constituents seem to occur widely for general de Rham $p$-adic Galois representations and may encode information about the Hodge filtration of $\rho_p$.\;In this article, we study this question for \textit{non-critical generic potentially crystalline} $p$-adic Galois representations.\;This extends the crystabelline picture of \cite{ParaDing2024} to the potentially crystalline setting and connects it with the local--global compatibility framework of the patched Bernstein eigenvariety developed in \cite{Ding2021}.\;

Although we impose the genericity assumption throughout this article in order to keep the exposition and the statements as clean as possible, the methods developed here are not expected to rely on it in an essential way.\;After replacing the generic smooth representation by the corresponding finite-length object in its Bernstein block and keeping track of possible coincidences among Jordan--H\"older factors and extension classes,  \textit{the same constructions should apply without essential difficulty to non-critical potentially crystalline representations without imposing the genericity assumption on $\wdre(\rho_p)$, or equivalently on $\pi_{\mathrm{sm}}(\rho_p)$}.\;We do not pursue this greater generality here only to avoid additional notation.\;In our work \cite{HEparaforsemitable}, we discuss the corresponding question in the \textit{semistable} case, which provides further evidence for the role of $p$-adic Hodge parameters.\;That work reduces the general semistable case to the crystalline case and to the so-called Steinberg case, the opposite extreme in which the Weil--Deligne representation $\wdre(\rho_p)$  has maximal monodromy rank.\;This reduction is, in spirit, parallel to the Bernstein--Zelevinsky classification theorem for smooth representations.\;


In the sequel, let $\Delta:=\{1,\cdots,n-1\}$ be the set of simple roots of $\GLN_n$.\;Fix $S_0\subseteq \Delta$, and let $\bL_{S_0}:=\prod_{i=1}^s\GLN_{r_i}\subseteq \GLN_n$, for integers $r_1,\cdots,r_s$ satisfying $r_1+\cdots+r_s=n$, be the standard Levi subgroup associated with $S_0$.\;For $u\in \sW_s$, let $S_0^u$ be the subset of $\Delta$ such that the associated Levi subgroup satisfies $\bL_{S_0^u}=\prod_{i=1}^s\GLN_{r_{u^{-1}(i)}}$.\;For $i=1,\cdots,s$ and $u\in\sW_s$, put $t_0^u=0$,\;$t_{s}^u=n$ and
\[t_i^u=\sum_{j=1}^i r_{u^{-1}(j)}.\]
Thus $\Delta\backslash S_0^u=\{t_1^u,\cdots,t_{s-1}^u\}$.\;Let $\bB$ be the Borel subgroup of $\GLN_n$ consisting of upper triangular matrices, and let $\bP_{S_0^u}\supseteq \bL_{S_0^u}\bB$ be the standard parabolic subgroup associated with $S_0^u$.\;For $1\leq i\leq s$, fix a cuspidal Bernstein component $\Omega_{i}$ of $\GLN_{r_i}(\bQ_p)$.\;Then, for $u\in \sW_s$, $\Omega_{S^u_0}:=\prod_{i=1}^s\Omega_{u^{-1}(i)}$ is a cuspidal Bernstein component of the Levi subgroup $\bL_{S^u_0}$.\;Put $G:=\GLN_n(\bQ_p)$.\;

Let $\cR_{E}$ be the Robba ring over $\bQ_p$ with coefficients in $E$.\;We write $\cR_{E}(\delta)$ for the $(\varphi,\Gamma)$-module over $\cR_{E}$ of rank one associated to a continuous character $\delta:\bQ_p^\times\rightarrow E$.\;For $\alpha\in E$,\;let $\unr(p)$ be the unramified character of $\bQ_p^\times$ sending uniformizers $p$ to $\alpha$.\;For integer $\bh$,\;let $z^{h}$ be the character sending $z\in \bQ^\times_p$ to $z^{h}$.\;

For a closed point $x_i\in \Spec\FZ_{\Omega_{i}}(E)$ ($1\leq i\leq r$), let $\pi_{x_i}$ be the associated irreducible cuspidal sm ooth representation over $E$ of type $\Omega_i$ and let $\wdre_{x_i}:=\mathrm{rec}(\pi_{x_i})$ be the associated  $r$-dimensional absolutely irreducible Weil--Deligne representation via the normalized classical local Langlands correspondence (see \cite{scholze2013local}).\;$\wdre_{x_i}$ corresponds to a Deligne--Fontaine module $\df_{x_i}$, by Fontaine's equivalence of categories as in \cite[Proposition 4.1]{breuil2007first}, and a $p$-adic differential equation ${\Delta_{x_i}}$ over $\cR_{E}$.\;Here ${\Delta_{x_i}}$ is a de Rham $(\varphi,\Gamma)$-module of rank $r$ over $\cR_{k(x),L}$ with constant Hodge--Tate weights $0$ such that, after forgetting the Hodge filtration, $D_{\mathrm{pst}}(\Delta_{x_i})$ is isomorphic to $\df_{x_i}$, by Berger's theory \cite[Theorem A]{berger2008equations}.\;

Fix  a potentially crystalline $p$-adic Galois representation $\rho_p$, and assume that the associated potentially crystalline $(\varphi,\Gamma)$-module $\Dpik$ over $\cR_{E}$ of rank $n$ admits, for each $u\in \sW_s$, a \textit{non-critical} $\Omega_{S^u_0}$-filtration $\cF_u$ in the sense of \cite{Ding2021}; this is a parabolic filtration by saturated $(\varphi,\Gamma)$-submodules of $\Dpik$ over $\cR_{E}$:
\[\cF_u=\fil_{\bullet}^{\cF_u}\Dpik: \ 0 =\fil_0^{\cF_u}\Dpik \subsetneq \fil_1^{\cF_u}\Dpik \subsetneq \cdots \subsetneq \fil_{s}^{\cF_u} \Dpik=\Dpik\]
such that there is an injection of $(\varphi,\Gamma)$-modules over $\cR_{E,L}$ of rank $r_{u^{-1}(i)}$:
\begin{equation}
	\mathbf{I}^u_{i}:E_{u,i}:=\gr_{i}^{\cF_u} \Dpik \hookrightarrow \Delta_{x_{u^{-1}(i)}}\otimes_{\cR_E} \cR_{E}(z^{t^u_i}),\quad(\text{note that}\;E_{u,i}[1/t]=\Delta_{x_{u^{-1}(i)}}[1/t])
\end{equation}
for some  $\underline{x}:=(x_{i})_{1\leq i\leq s}\in \Spec\FZ_{\Omega_{S_0}}(E)$ (so that $\underline{x}^u=(x_{u^{-1}(i)})_{1\leq i\leq s}\in \Spec\FZ_{\Omega_{S^u_0}}(E)$)
and the non-critical assumption means that the Hodge--Tate weights of $\fil_{i}^{\cF_u}\Dpik$ (resp.,\;$\gr_{i}^{\cF_u} \Dpik$) are
\[(\hpi_{1},\hpi_{2},\cdots,\hpi_{t_i^u}),\quad (\text{resp.,\;}(\hpi_{t^u_{i-1}+1},\hpi_{t^u_{i-1}+2},\cdots,\hpi_{t_i^u})),\]
In this paper,\;we usually use $\bullet-\bullet$ to denote an extension of two objects, such as $(\varphi,\Gamma)$-modules or $G$-representations, where the first (resp., second) object is the subobject (resp., quotient)).\;Then $\cF_u$ corresponds to the following successive  extension of $\{E_{u,i}\}$ :
\[\Dpik=[{E}_{u,1}-{E}_{u,2}-\cdots-{E}_{u,s}]=[{E}_{1}-{E}_{2}-\cdots-{E}_{s}],\;\]
where we write $E_i:=E_{1,i}$ for simplicity.\;In particular, if $S_0=\emptyset$, then $\Dpik$ is crystabelline and $\cF_u$ becomes the classical triangulation (or the refinement associated to $u$).\;

Let $L'$ be a finite Galois extension of $\bQ_p$ such that $\Dpik|_{L'}$ is a crystalline $(\varphi,\Gamma)$-module over $\cR_{E,L'}$ of rank $n$, where $\cR_{E,L'}$ is the Robba ring over $L'$ with coefficients in $E$.\;By Fontaine's theory, we consider the Deligne--Fontaine module attached to $\Dpik$:
$$\df_{\Dpik}=(\varphi,\Gal(L'/\bQ_p),D_{\pcr}(\Dpik))$$
where $D_{\pcr}(\Dpik)=D_{\mathrm{cr}}^{L'}(\Dpik\otimes_{\cR_{E}}\cR_{E,L'})$ is a finite free $L'_0\otimes_{\bQ_p}E$-module of rank $n$, and $L'_0$ is the maximal unramified subextension of $L'$ over $\bQ_p$.\;Since $\Dpik$ is potentially crystalline, hence de Rham, $D:=D_{\dr}(\Dpik)\cong (D_{\pcr}(\Dpik)\otimes_{L_0'}L')^{\Gal(L'/\bQ_p)}$ is a free $E$-module of rank $n$.\;The module $D$ carries a natural Hodge filtration, expressed by the following complete flag :
\[\fil^{\bullet}_{H}(D): \ 0 \subsetneq \fil^{-\hpi_{n}}_{H}(D) \subsetneq \fil^{-\hpi_{n-1}}_{H}(D) \subsetneq \cdots \subsetneq \fil^{{-\hpi_{1}}}_{H}(D)=D.\]
For $1\leq i\leq s$, fix an $E$-basis $\{e_{i,j}\}_{1\leq j\leq r_i}$ of $D_{\dr}(E_{i})$.\;Then $\{e_{1,j}\}_{1\leq j\leq r_1},\cdots,\{e_{s,j}\}_{1\leq j\leq r_s}$ form an $E$-basis of $D$; we relabel them as $e_{1},e_{2},\cdots,e_{n}$ using lexicographic order.\;Under such basis,\;$\fil^{\bullet}_{H}(D)$ corresponds to a point of $\bZ_{S_0}\backslash \GLN_n/\bB$,\;which we call the \textit{$p$-adic Hodge parameters} of $\rho_p$.\;Let $\homo_{\fil}(D,D)$ be the subspace of endomorphisms of $D$ that preserve the filtration $\fil^{H}_{\bullet}(D)$.\;

Throughout this paper,\;we assume that $\underline{x}$ is \textit{regular} and \textit{generic}, namely $\Delta_{x_i}\neq \Delta_{x_j}$ and $\Delta_{x_i}\neq \Delta_{x_j}\otimes_{\cR_{E}}\cR_{E}(\unr(p))$ for $i\neq j$.\;Equivalently, $\pi_{\mathrm{sm}}(\rho_p)$ is a generic smooth representation of $G$, and $\wdre(\rho_p)$ is a generic Weil--Deligne representation.\;Recall $\lambda=\bh-\theta$.\;Consider the following locally $\bQ_p$-analytic principal series of $G$:
\[\mathrm{PS}_{u}(\underline{x},\bh):=\Big(\ind^G_{\op_{S_0^u}(\bQ_p)}\pi_{\sm}(\underline{x}^u)\eta_{S_0^u}\otimes_E L_{S_0^u}(\lambda)\Big)^{\ana},\pi_{\sm}(\underline{x}^u):=\boxtimes_{j=1}^s\pi_{x_{u^{-1}(j)}},\]
where $L_{S_0^u}(\lambda)$ is the unique irreducible algebraic representation of $\bL_{S_0^u}(\bQ_p)$ with highest weight $\lambda$, $\op_{S_0^u}$ is the parabolic subgroup opposite to $\bP_{S_0^u}$, and $\eta_{S_0^u}$ is the square root of the modulus character of $\op_{S_0^u}$.\;The explicit structure of $\mathrm{PS}_{u}(\underline{x},\bh)$ is well-understand by Orlik-Strauch \cite[The main theorem]{orlik2015jordan}.\;By the generic assumption,\;$\soc_G\mathrm{PS}_{u}(\underline{x},\bh)\cong \pi_{\lalg}(\rho_p)\cong \mathrm{PS}^{\lalg}_{u}(\underline{x},\bh)$,\;where  $\mathrm{PS}^{\lalg}_{u}(\underline{x},\bh)$ is the space of locally $\bQ_p$-algebraic vectors in $\mathrm{PS}_{u}(\underline{x},\bh)$.\;For simplicity, put $\pi_{\lalg}(\underline{x},\bh):=\mathrm{PS}^{\lalg}_{u}(\underline{x},\bh)\cong \pi_{\lalg}(\rho_p)$.\;

Let $\mathrm{PS}_{u,1}(\underline{x},\bh)$ be the first two layers of the socle filtration of $\mathrm{PS}_u(\underline{x},\bh)$.\;The amalgamated sum 
\[\bigoplus^{u\in \sW_{s}}_{\pi_{\lalg}(\underline{x},\bh)}\mathrm{PS}_{u,1}(\underline{x},\bh)\]
 has the unique quotient $\pi^{\flat}_{1}(\underline{x},\bh)$ of socle $\pi_{\lalg}(\underline{x},\bh)$,\;which lies in the following short exact sequence (here $\sW_{s}^{\widehat{j},\emptyset}=\sW_{\widehat{j}}\backslash\sW_s$, $\sW_{\widehat{j}}$ is the subgroup of $\sW_s$ generated by the simple reflections $\{s_i\}_{i\in \widehat{j}}$, and $\widehat{j}=\Delta_s\backslash \{j\}$):
\begin{equation}\label{firstwholeintro}
	0\rightarrow \pi_{\lalg}(\underline{x},\bh) \rightarrow \pi^{\flat}_{1}(\underline{x},\bh)\rightarrow \oplus_{\substack{1\leq j\leq s-1\\u\in \sW_{s}^{\widehat{j},\emptyset}}}C_{t_j^u,u}\rightarrow 0,
\end{equation}
where $C_{t_j^u,u}$ are explicit pure irreducible locally analytic representations given by the Orlik--Strauch functor (see \cite[The main theorem]{orlik2015jordan}).\;We state our three main theorems and several remarks as follows.\;

\begin{thm}\label{thmintor1}(Proposition \ref{construext1}) There exists a map, depending only on the choice of $\log_p(p)\in E$,
	\begin{equation}\label{reconsfortDpikcirc}
{t}^{\flat}_{D}:\ext^1_{G}\big(\pi_{\lalg},\pi^{\flat}_{1}(\underline{x},\bh)\big)\twoheadrightarrow
		\Big(\prod_{i=1}^s\ext^{1}(E_i[1/t],E_i[1/t])\Big)\oplus\homo^{\flat}_{\fil}(D,D),
	\end{equation}
and $\homo^{\flat}_{\fil}(D,D)\subseteq \homo_{\fil}(D,D)$ is described in Proposition \ref{introforhomoflat}.\;
\end{thm}

Define $\Delta'=\bigcup_{u\in \sW_s}\Delta\backslash S_0^u=\bigcup_{u\in \sW_s}\{t_1^u,\cdots,t_{s-1}^u\}\subseteq\Delta$, and write $\Delta':=\{i_{1}<i_{2}<\cdots<i_{|\Delta'|}\}$.\;For $i\in\Delta'$, put $\cI_{i}^{\flat}:=\{u\in \sW_s:i\notin S^u_0 \}$.\;For any $u\in \cI_{i}^{\flat}$,\;suppose that $i=t^u_{l_i(u)}$ for some $1\leq l_i(u)\leq s-1$.\;Define
\[D^{\flat}_{(i)}:=E\langle e_{u^{-1}(q),j}:l_i(u)+1\leq q\leq s,1\leq j\leq r_{u^{-1}(q)},u\in \cI_{i}^{\flat}\rangle\subseteq D,\]
and set $\bd_i=n-\dim_ED^{\flat}_{(i)}$.\;Note that $\bd_i=0$ if and only if $D^{\flat}_{(i)}=D$.\;

\begin{thm}\label{Hodgeparadeterintro}(Theorem \ref{Hodgeparadeter} \& Corollary \ref{Hodgeparadetercor}) For $1\leq t\leq |\Delta'|$, $\ker(t^{\flat}_{D})$ determines
	\begin{equation}
		\begin{aligned}
			\Big\{\fil_H^{-\bh_{{\max\{i_{{j}},\bd_{i_{t}}\}}+1}}(D_{\sigma})\oplus D_{\sigma}/\fil_H^{-\bh_{\bd_{i_{t}}}}(D_{\sigma})\Big\}_{1\leq j\leq t}
		\end{aligned}
	\end{equation}	
	In particular, if $\bd_{i_t}=0$, then $\ker(t^{\flat}_{D_{\sigma}})$ determines $\{\fil_H^{-\bh_{i_{{j}}+1}}(D_{\sigma})\}_{1\leq j\leq t}$.\;
\end{thm}
The tautological extension of $\pi_{\lalg}(\underline{x},\bh)\otimes_E\ker(t^{\flat}_{D})$ by $\pi^{\flat}_{1}(\underline{x},\bh)$ gives a locally analytic representation $\pi^{\flat}_{\min}(\Dpik)$.\;
\begin{thm}\label{thmintor11} (Theorem \ref{mainthmglobal}:\;local--global compatibility in the patched setting) Let $R_{\infty}$ be the patched Galois deformation ring and $\Pi_{\infty}$ be the patched $R_{\infty}$-admissible unitary representation.\;If $\rho_p$ comes from a maximal ideal $\fm_{\rho}$ of $R_{\infty}[1/p]$, then $\pi^{\flat}_{\min}(\Dpik)\hookrightarrow \Pi_{\infty}[\fm_{\rho}]$.\;
\end{thm}

\begin{rmk}In general,\;for a de Rham $\Dpik$,\;when a suitable locally analytic representation $\pi_{1}(D)$ is available,\;\cite[(6)]{BDcritical25} predicts a surjection
 \begin{equation}
	\begin{aligned}
t_{D}:\ext^1_{G}\big(\pi_{\lalg}(D),\pi_{1}(D)\big)\twoheadrightarrow\ext^1_{\varphi}(D,D)\oplus\homo_{\fil}(D,D),\;
	\end{aligned}
\end{equation}
where $\ext^1_{\varphi}(D,D)$ means extensions as $\varphi$-modules.\;In our case, we expect $\pi^{\flat}_{1}(\underline{x},\bh)$ to be a subrepresentation of $\pi_{1}(D)$, and (\ref{reconsfortDpikcirc}) gives $t_{D}|_{\ext^1_{G}\left(\pi_{\lalg},\pi^{\flat}_{1}(\underline{x},\bh)\right)}$ and its image.\;For $u\in \sW_s$, recall that $\Dpik=[{E}_{u,1}-{E}_{u,2}-\cdots-{E}_{u,s}]$.\;Assume that for each $1\leq i\leq s$, one can associate a locally analytic representation $\pi_{\ana}({E}_{u,i})$ that determines ${E}_{u,i}$.\;Consider the locally analytic parabolic induction
\begin{equation*}
	\mathrm{PS}_{\cF_{u}
		}(\underline{x},\bh):=\big(\ind^G_{\op_{S^u_0}(L)}\pi_{\ana}(\underline{x}^u)\eta_{S_0^u}\big)^{\ana},\;\pi_{\ana}(\underline{x}^u):=\boxtimes_{i=1}^s\pi_{\ana}({E}_{u,i}).\;
\end{equation*}
We have a injection of locally analytic $G$-representations $\mathrm{PS}_{S_0,u}(\underline{x},\bh)\hookrightarrow\mathrm{PS}_{\cF_{u}}(\underline{x},\bh)$.\;Then the  amalgamated sum (if possible) of $\{\mathrm{PS}_{\cF_{u}}(\underline{x},\bh)\}_{u\in \sW_{s}}$ has a unique quotient $\pi_{1}^+(\underline{x},\bh)$ of socle $\pi_{\lalg}(\underline{x},\bh)$  such that $\pi^{\flat}_{1}(\underline{x},\bh)\hookrightarrow\pi_{1}^+(\underline{x},\bh)$; we expect $\pi_{1}^+(\underline{x},\bh)\hookrightarrow \pi_{1}(D)$.\;Then (\ref{reconsfortDpikcirc}) can be extended to a morphism
\begin{equation}\label{introt+D}
	t^+_{D}:\ext^1_{G}\big(\pi_{\lalg}(\underline{x},\bh),\pi^+_{1}(\underline{x},\bh)\big)\rightarrow\ext^1_{\varphi}(D,D)\oplus\homo_{\fil}(D,D).
\end{equation}
See  Expectation \ref{expectintro} for a further comment.\;
\end{rmk}

Let $\fg$ (resp.,\;$\fb$) be the $E$-Lie algebra of $\GLN_n$ (resp.,\;$\bB$).\;Moreover,\;for $u\in \sW_s$,\;let $\fz_{S^u_0}$ (resp.,\;$\fn_{S^u_0}$) be the $E$-Lie algebra of the center $\bZ_{S^u_0}$ of $\bL_{S_0^u}$ (resp.,\;of the unipotent radical of $\bP_{S^u_0}$).\;Put $\tau_{S^u_0}=\fz_{S^u_0}\ltimes\fn_{S^u_0}$ the full radical.\;Under the basis $\{e_{i}\}_{1\leq i\leq n}$, we identify $\homo_{E}(D,D)$ with $
\fg$.\;If we choose $g\in \GLN_n(E)$ such that $g\bB\in \GLN_n/\bB$ gives the coordinates of the complete Hodge flag $\fil^{\bullet}_{H}(D)$ in this basis, then $\homo_{\fil}(D,D)$ is identified with $\mathrm{Ad}_g(\fb)$.\;Let $u^{\sharp}\in \sW_n$ be the unique element sending $e_{i,j}$ for $1\leq j\leq r_i$ to $e_{u^{-1}(i),j}$ for $1\leq j\leq r_i$,\;which thus induces a permutation on $e_{1},\cdots,e_{n}$.\;

\begin{pro}\label{introforhomoflat}(Propositions \ref{dfnisoforflatfilhomo} \& \ref{comparesharpflat})
	$\homo^{\flat}_{\fil}(D,D)\cong \mathrm{Ad}_{g}(\fb)_{S_0}^{\flat}:=\sum_{u\in \sW_s}\mathrm{Ad}_{(u^{\sharp})^{-1}}(\tau_{S^u_0})\cap \mathrm{Ad}_{g}(\fb)$ and 
$\dim_E\ker(t^{\flat}_{D})=2^s-1-\dim_E\mathrm{Ad}_{g}(\fb)_{S_0}^{\flat}$.\;
\end{pro}
\begin{rmk}In particular,\;if $S_0=\emptyset$,\;then $\mathrm{Ad}_{g}(\fb)_{S_0}^{\flat}=\mathrm{Ad}_{g}(\fb)$ and $\dim_E\ker(t^{\flat}_{D})=2^n-\frac{n(n+1)}{2}-1$.\;We come back to \cite{ParaDing2024}.\;
\end{rmk}

\begin{rmk}\label{rmkintroII} (An example for Theorem \ref{Hodgeparadeterintro}) A typical example is $\bL_{S_0}\cong \GL_1^{\oplus (n-m)}\times \GLN_m$.\;As $E$-vector spaces, $D=D^1\oplus\cdots\oplus D^{n-m}\oplus D^{n-m+1}$, with $\dim_ED^j=1$ for $1\leq j\leq n-m$ and $\dim_ED
		_{\sigma}^{n-m+1}=m$.\;For $i\in \Delta'$,\;we have
		\begin{equation}
			D_{(i)}=\left\{
			\begin{array}{ll}
				D/D^{n-m+1},\;&n-i<m\\
				D ,\;&n-i\geq m
			\end{array}
			\right.,\;\bd_{i}=\left\{
			\begin{array}{ll}
				m,\;&n-i<m\\
				0,\;&n-i\geq m
			\end{array}
			\right.,\;
		\end{equation}
		If $m\leq n-m$ (resp.,\;$m>n-m$), then $\Delta'=\Delta$ (resp.,\;$\Delta'=\Delta\backslash \{n-m+1,\cdots,m-1\}$).\;For $n-i\geq m$, $\ker(t^{\flat}_{D})$ determines $\{\fil_H^{-\bh_{j+1}}(D)\}_{1\leq j\leq i}$.\;If $n-i<m$ (resp.,\;$i>m-1$), then $\ker(t^{\flat}_{D})$ determines
		\begin{equation}
			\begin{aligned}
				\Big\{\fil_H^{-\bh_{j+1}}(D)\oplus D/\fil_H^{-\bh_{m}}(D)\Big\}_{m \leq j\leq i}\cup \Big\{\fil_H^{-\bh_{m+1}}(D)\oplus D/\fil_H^{-\bh_{m}}(D)\Big\}.
			\end{aligned}
		\end{equation}	
\end{rmk}

We finally  relate Theorem \ref{thmintor1} to certain parabolic deformations.\;For $u\in \sW_{s}$, note that $\bZ_{S_0^u}(\bQ_p)\cong \bQ_p^s$ is the center of $\bL_{S^u_0}(\bQ_p)$.\;Let $\homo(\bZ_{S_0^u}(\bQ_p),E)$ (resp.,\;$\homo(\bQ_p^{\times},E)$) be the $E$-vector space of locally analytic $E$-valued characters on $\bZ_{S_0^u}(\bQ_p)$ (resp.,\;$\bQ_p^{\times}$).\;Let $\ext^1_u(\Dpik,\Dpik)\subseteq \ext^1(\Dpik,\Dpik)$ be the space of $\bP_{S_0^u}$-parabolic deformations of $\Dpik$ that preserve $\cF_u$.\;There are several deformation subspaces $\ext^{1,\flat}_{u}(\Dpik,\Dpik)\subseteq \ext^{1,0}_{u}(\Dpik,\Dpik)$ of $\ext^{1}_{{u}}(\Dpik,\Dpik)$ which appear on the Bernstein eigenvariety. After identifying elements of $\ext^1(\Dpik,\Dpik)$ with deformations of $\Dpik$ over $\cR_{E[\epsilon]/\epsilon^2}$ (where $\cR_{E[\epsilon]/\epsilon^2}$ is the Robba ring over $\bQ_p$ with coefficients in $E[\epsilon]/\epsilon^2$), an element $\widetilde{D}\in \ext^1(\Dpik,\Dpik)$ belongs to $\ext^{1,\flat}_u(\Dpik,\Dpik)$ (resp.,\;$\ext^{1,0}_u(\Dpik,\Dpik)$) if
\[\widetilde{D}=[\widetilde{E}_{u,1}-\widetilde{E}_{u,2}-\cdots-\widetilde{E}_{u,s}]\]
and $\widetilde{E}_{u,i}\cong {E}_{u,i}\otimes_{\cR_{E}}\cR_{E[\epsilon]/\epsilon^2}(1+\psi_i\epsilon)$ (resp.,\;$\widetilde{E}_{u,i}\hookrightarrow  \Delta_{x_{u^{-1}(i)}}\otimes_{\cR_{E}}\cR_{E[\epsilon]/\epsilon^2}(z^{t^u_i}(1+\psi_i\epsilon))$) for some $\psi_i\in \homo(\bQ_p^{\times},E)$ and $1\leq i\leq s$.\;For $?\in\{\flat,0\}$,\;we have a  natural surjection (by the generic assumption)
\begin{equation}
	\kappa_{u}^{?}:\ext^{1,?}_{{u}}(\Dpik,\Dpik)\rightarrow \homo(\bZ_{S_0^u}(\bQ_p),E),
\end{equation}
which sends $\widetilde{D}$ to its deformation parameters $(\psi_i)_{1\leq i\leq s}$.\;Put $\ext^{1,\flat}_0(\Dpik,\Dpik)=\ker \kappa^{\flat}_1$.\;For any subspace $V\subseteq \ext^{1}(\Dpik,\Dpik)$, put $\overline{V}:=V/(V\cap \ext^{1,\flat}_0(\Dpik,\Dpik))$.\;In particular, we obtain an isomorphism
$\overline{\ext}^{1,\flat}_{u}(\Dpik,\Dpik)\xrightarrow{\sim }\homo(\bZ_{S_0^u}(\bQ_p),E)$.\;Moreover,\;for $?\in\{\flat,0\}$, we have a natural map
\begin{equation}\label{dfnforgDpikintro}
	\overline{g}^{?}_{\Dpik}:\bigoplus_{u\in\sW_{s}}\overline{\ext}^{1,?}_u(\Dpik,\Dpik)\rightarrow\overline{\ext}^{1}(\Dpik,\Dpik).
\end{equation}
We have a natural morphism:
\begin{equation}
	\begin{aligned}
		\homo(\bZ_{S^u_0}(\bQ_p),E)&\rightarrow\ext^1_{G}\big(\mathrm{PS}_{u}(\underline{x},\bh),\mathrm{PS}_{u}(\underline{x},\bh)\big),\;\\
		\psi&\mapsto \Big(\ind^G_{\op_{S_0^u}(\bQ_p)}\pi_{\sm}(\underline{x}^u)\eta_{S_0^u}\otimes_E L_{S_0^u}(\lambda)\otimes_E(1+(\psi\circ\mathrm{det}_{\bL_{S_0^u}})\epsilon)\Big)^{\ana},\;
	\end{aligned}
\end{equation}
which leads to a map, obtained by pushing forward via $\mathrm{PS}_{u}(\underline{x},\bh)\hookrightarrow \pi^{\flat}_{1}(\underline{x},\bh)$ and pulling back via the natural map $\pi_{\lalg}(\underline{x},\bh)\cong \mathrm{PS}^{\lalg}_{u}(\underline{x},\bh)\hookrightarrow \mathrm{PS}_{u}(\underline{x},\bh)$,
\[\zeta_{u}:\homo(\bZ_{S^u_0}(\bQ_p),E)\rightarrow  \ext^1_{G}\left(\pi_{\lalg}(\underline{x},\bh),\pi^{\flat}_{1}(\underline{x},\bh)\right).\]
Then the composition $\oplus_{u\in\sW_{s}}(\zeta_u\circ \kappa^{\flat}_u)$ gives
\begin{equation}\label{gammDpikcirc}
	\begin{aligned}
		\gamma^{\flat}_{\Dpik}:\bigoplus_{u\in\sW_{s}}\overline{\ext}^{1,\flat}_{u}(\Dpik,\Dpik)\rightarrow \ext^1_{G}\left(\pi_{\lalg}(\underline{x},\bh),\pi^{\flat}_{1}(\underline{x},\bh)\right).\;
	\end{aligned}
\end{equation}
The theory of almost de Rham representations induces a natural morphism $\nu:{\ext}^{1}(\Dpik,\Dpik)\rightarrow \homo_{\fil}(D,D)$. Its kernel is ${\ext}^{1}_g(\Dpik,\Dpik)$, the subspace of de Rham deformations of $\Dpik$.\;We have the following short exact sequence:
\begin{equation}\label{exactseqforhomfilanext1}
	0\rightarrow \prod_{i=1}^s\ext^1_g(E_i,E_i)\rightarrow  \overline{\ext}^{1}(\Dpik,\Dpik) \rightarrow \homo_{\fil}(D,D)\rightarrow 0.
\end{equation}

\begin{pro}\label{thmintor2}
\begin{itemize}
	\item[(1)] (Proposition \ref{splitingforext1gps})	Put ${\ext}^{1,\flat}(\Dpik,\Dpik)=\mathrm{Im}(g^{\flat}_{\Dpik})$ the image (\ref{dfnforgDpikintro}).\;The restriction of (\ref{exactseqforhomfilanext1}) to $\overline{\ext}^{1,\flat}(\Dpik,\Dpik)$ admits a splitting, depending only on the choice of $\log_p(p)\in E$:
	\[f_{\Dpik}^{\flat}:\overline{\ext}^{1,\flat}(\Dpik,\Dpik)\xrightarrow{\sim} \ext^{1,\flat}_{\varphi^f}(\Dpik[1/t],\Dpik[1/t]) \oplus\mathrm{Im}(\nu|_{\overline{\ext}^{1,\flat}(\Dpik,\Dpik)}),\]
	where $\ext^{1,\flat}_{\varphi^f}(\Dpik[1/t],\Dpik[1/t]):=\prod_{i=1}^s\ext^{1,\flat}_g(E_i,E_i)\cong \homo(\bZ_{S_0}(\bQ_p),E)$.\;
	\item[(2)]	(Proposition \ref{mainthmforfullrefine}) $f_{\Dpik}^{\flat}\circ \overline{g}^{\flat}_{\Dpik}=t^{\flat}_{D}\circ\gamma^{\flat}_{\Dpik}$.\;
\end{itemize}	
\end{pro}

\begin{rmk}\label{rmkintroIII}(Infinite fern for potentially crystalline cases;\;Theorem \ref{infinitePoten},\;Theorem \ref{Inifinitefernpoten}) The natural map $g_{\Dpik}:\bigoplus_{u\in \sW_{s}}{\ext}^{1}_{u}(\Dpik,\Dpik)\rightarrow {\ext}^1(\Dpik,\Dpik)$  is surjective.\;
\end{rmk}

\begin{rmk}\label{expectintro}
  (\ref{gammDpikcirc}) can be extended  to $\gamma^+_{\Dpik}:\bigoplus\nolimits_{u}\overline{\ext}^{1}_{u}(\Dpik,\Dpik)\rightarrow\ext^1_{G}\big(\pi_{\lalg}(\underline{x},\bh),\pi^+_{1}(\underline{x},\bh)\big)$.\;We  suspect that there are $p$-adic Hodge parameters that are contributed by the conjectural supersingular constituents in $\pi_1(D)/\pi^+_{1}(\underline{x},\bh)$.\;Thus  $\ker({t}^{+}_{D})$ still not determines $\{\fil_H^{-\bh_{i+1}}(D_{\sigma})\}_{i\in \Delta}$.\;
\end{rmk}

\section*{Acknowledgment}
The author thank Yiwen Ding and Zicheng Qian for discussions or answers to questions.\;

\section{Preliminaries}

\subsection{General notation}

\noindent Let $L$ (resp.\;$E$) be a finite extension of $\bQ_p$ with $\co_L$ (resp.\;$\co_E$) as its ring of integers and $\varpi_L$ (resp.\;$\varpi_E$) a uniformizer.\;Suppose $E$ is sufficiently large containing all the embeddings of $L$ in $\overline{\BQ}_p$.\;Put $\Sigma_L:=\{\sigma: L\hookrightarrow \overline{\BQ}_p\} =\{\sigma: L\hookrightarrow E\}$.\;Let $\val_L$ (resp. $\val_p$) be the $p$-adic valuation on $\overline{\bQ}_p$ normalized by sending uniformizers of $\co_L$ (resp.,\;the ring $\bZ_p$ of integers in $\bQ_p$) to $1$.\;Let $d_L:=[L:\bQ_p]=|\Sigma_L|$ and $q_L:=p^{f_L}=|\co_L/\varpi_L|$.\;Let $\gal_{L}:=\gal(\overline{L}/L)$ be the absolute Galois group of $L$.\;

Let $\cR_{L}:=\bB_{\rig,L}^{\dagger}$ be the Robba ring associated to $L$.\;For $\bQ_p$-affinoid algebra (for example $A\in\{E,E[\epsilon]/\epsilon^2\}$),\;let $\cR_{A,L}:=\cR_{L}\widehat{\otimes}_{\bQ_p} A$ for the Robba ring associated to $L$ with $A$-coefficient.\;We write $\cR_{A,L}(\delta_A)$ for the $(\varphi,\Gamma)$-module over $\cR_{A,L}$ of rank one associated to a continuous character $\delta_A:L^\times\rightarrow A^\times$.\;Let $D$ be a $(\varphi,\Gamma)$-module over $\cR_{A,L}$,\;we write $D(\delta_A):=D\otimes_{\cR_{A,L}}\cR_{A,L}(\delta_A)$ for simplicity.\;

For a Lie algebra $\fg$ over $L$,\;and $\sigma\in \Sigma_L$,\;let $\fg_{\sigma}:=\fg\otimes_{L,\sigma} E$ (which is  a Lie algebra over $E$).\;

Let $\GL_n$ be the general linear group over $L$.\;Let $\Delta:=\Delta_n$ be the set of simple roots of $\GL_n$ (with respect to the Borel subgroup $\bB$ of upper triangular matrices), and we identify the set $\Delta$ with $\{1,\cdots, n-1\}$ such that $i\in \{1,\cdots, n-1\}$ corresponds to the simple root $\alpha_{i}:\;(x_1,\cdots, x_n)\in \ft \mapsto x_{i}-x_{i+1}$, where $\ft$ denotes the $L$-Lie algebra of the torus $\bT$ of diagonal matrices.\;Each $I\subseteq \Delta$ gives rise to a (standard) Levi subgroup $\bL_I\subseteq \GL_n$ such that $\bT\subseteq \bL_I$ and $I$ is equal to the set of simple roots of $\mathbf{L}_{I}$.\;Let $\mathbf{P}_{I}:=\mathbf{L}_{I}\bB$ be the (standard) parabolic subgroup of
$\GL_n$ containing $\bB$.\;In particular,\;we have $\mathbf{P}_{\Delta}=\GL_n$, $\mathbf{P}_{\emptyset}=\mathbf{B}$.\;Let $\overline{\mathbf{P}}_{I}$ be the parabolic subgroup opposite to $\mathbf{P}_{I}$.\;Let $\bN_{I}$ (resp.\;$\overline{\bN}_{I}$) be the unipotent radical of $\bP_{I}$ (resp.\;$\overline{\bP}_{I}$).\;We have Levi decompositions $\bP_I=\bL_I\bN_I$,\;$\overline{\bP}_I=\bL_I\overline{\bN}_I$ and $\Delta \backslash  I$ are precisely the simple roots of the unipotent radical $\bN_{I}$.\;Let $\bZ$ (resp.,\;$\bZ_I$) be the center of $\GL_n$ (resp.,\;$\mathbf{L}_{I}$).\;Let $\fg$,\;$\fb$,\;$\fp_{I}$,\;$\fl_{I}$,\;$\fn_I$ and $\fz_I$ be the $L$-Lie algebras of $\GL_n$,\;$\bB$,\;$\mathbf{P}_{I}$, $\mathbf{L}_{I}$,\;$\bN_I$ and $\bZ_I$ respectively.\;We put $G:=\GL_n(L)$.\;

Let $m\in\BZ_{\geq 1}$,\;and $\pi$ be an irreducible smooth admissible representation of $\GLN_m(L)$,\;let $\rec(\pi)$ be the $m$-dimensional absolutely irreducible $F$-semi-simple Weil-Deligne representation of the Weil group $W_L$ via the normalized classical local Langlands correspondence (normalized in \cite{scholze2013local}).\;We normalize the reciprocity isomorphism $\rec:L^\times\rightarrow W_L^{\mathrm{ab}}$ of local class theory such that the uniformizer $\varpi_{L}$ is mapped to a geometric Frobenius morphism,\;where $W_L^{\mathrm{ab}}$ is the abelization of the Weil group $W_L\subset \gal_L$.\;

For a group $A$ and $a\in A$,\;denote by $\unr(\alpha)$ the unramified character of $L^\times$ sending uniformizers to $\alpha$.\;If $\bk:=(\bk_{\tau})_{\tau\in \Sigma_L}\in \BZ^{\Sigma_L}:=\prod_{\sigma\in \Sigma_L}\BZ$,\;we denote $z^{\bk}:=\prod_{\tau\in  \Sigma_L}\tau(z)^{\bk_{\tau}}$.\;For a character $\chi$ of $\cO_L^{\times}$,\;denoted by $\chi_{\varpi_L}$ the character of $L^{\times}$ such that $\chi_{\varpi_L}|_{\cO_L^{\times}}=\chi$ and $\chi_{\varpi_L}(\varpi_L)=1$.\;Let ${\ccyc}:\gal_L\rightarrow \bZ_p^\times$ be the $p$-adic cyclotomic character.\;Then we have ${\ccyc}\circ\rec=\unr(q_L^{-1})\prod_{\tau\in \Sigma_L}\tau$ by local class theory.\;Let $A$ be an affinoid $E$-algebra.\;A locally $\bQ$-analytic character $\delta:L^\times\rightarrow A^\times$ induces a $\bQ_p$-linear map $L\rightarrow A$,\;$x\mapsto \frac{d}{dt}\delta(\exp(tx))|_{t=0}$ and hence it induces an $E$-linear map $L\otimes_{\bQ_p} E=\prod_{\tau\in \Sigma_L}E\rightarrow A$.\;There exist $\wt(\delta):=(\wt_{\tau}(\delta))_{\tau\in \Sigma_L}$ such that the latter map is given by $(a_{\tau})_{\tau\in \Sigma_L}\mapsto\sum_{\tau\in \Sigma_L}a_{\tau}\wt_{\tau}(\delta)$.\;We call $\wt(\delta)$ the weight of $\delta$.\;

Let $\ul{\lambda}:=(\lambda_{\sigma,1}, \cdots, \lambda_{\sigma,n})_{\sigma\in \Sigma_L}$ be a weight of $\ft_{L}$.\;For $I\subseteq \Delta$,\;we call that $\ul{\lambda}$ is $I$-dominant (resp.,\;strictly $I$-dominant)  if $\lambda_{\sigma,i}\geq \lambda_{\sigma,i+1}$ (resp.,\;$\lambda_{\sigma,i}> \lambda_{\sigma,i+1}$)   for all $i\in I$ and $\sigma\in \Sigma_L$.\;We denote by $X_{I}^+$ the set of $I$-dominant integral weights  of $\ft_{\Sigma_L}$.\;For $\ul{\lambda}\in X_{I}^+$, there exists a unique irreducible algebraic representation, denoted by $L_{I}(\ul{\lambda})$, of $\res_{L/
\bQ_p}\bL_{I}$ with highest weight $\ul{\lambda}$ with respect to $\res_{L/
\bQ_p}\bL_{I}\cap \bB$,\;so that $\overline{L}_{I}(-\ul{\lambda}):=(L_{I}(\ul{\lambda}))^\vee$ is the irreducible algebraic representation of $(\bL_{I})_{L}$ with highest weight $-\ul{\lambda}$ with respect to $\res_{L/
\bQ_p}\bL_{I}\cap \ob$.\;If $\ul{\lambda}\in X_{\Delta}^+$, let $L(\ul{\lambda}):=L(\ul{\lambda})_{\Delta}$.\;A $\bQ_p$-algebraic representation of $G$ over $E$ is the induced action of $G\subset \GL_{n,L}(E)$ on an algebraic representation of $\res_{L/
\bQ_p}\GL_{n}$.\;Let $\ul{\lambda}\in X_I^+$ be an integral weight,\;denote by $M_I(\ul{\lambda}):=\text{U}(\fg_{L})\otimes_{\text{U}(\fp_{I,L})} L_{I}(\ul{\lambda})$ (resp.\;$\overline{M}_I(\ul{\lambda}):=\text{U}(\fg_{L})\otimes_{\text{U}(\overline{\fp}_{I,L})}L_{I}(-\ul{\lambda})$),\;the corresponding Verma module with respect to $\fp_{I,L}$ (resp.\;$\overline{\fp}_{I,L}$).\;Let $L(\ul{\lambda})$ (resp. $\overline{L}(-\ul{\lambda})$) be the unique simple quotient of $M(\ul{\lambda}):=M_{\emptyset}(\ul{\lambda})$ (resp. of $\overline{M}(-\ul{\lambda}):=\overline{M}_{\emptyset}(-\ul{\lambda})$).\;



Denote by $\sW_n$ ($\cong S_n$) the Weyl group of $\GL_n$, and denote by $s_{i}$ the simple reflection corresponding to $i\in \Delta$.\;For any $I\subset \Delta$,\;define $\sW_{I}$ to be the subgroup of $\sW_{n}$ generated by simple reflections $s_{i}$ with $i\in I$ (so that $\sW_I$ is the Weyl group of $\bL_I$).\;Let $I,J\subseteq\Delta$,\;recall that $\sW_{I}\backslash \sW_n/\sW_{J}$ has a canonical set of representatives,\;which we will denote by $\sW^{I,J}_n$ (resp.,\;$\sW^{I,J}_{n,\max}$),\;given by taking in each double coset the elements of minimal (resp.,\;maximal) length.\;Let $w_0:=w_{0,n}$  be the longest elements in $\sW_{n}$.\;

If $V$ is a continuous representation of $G$ over $E$, we denote by $V^{\bQ_p-\ana}$ its locally $\bQ_p$-analytic vectors.\;If $V$ is a locally $\bQ_p$-analytic representation of $G$, we denote by $V^{\mathrm{sm}}$ (resp.\;$V^{\mathrm{lalg}}$) the smooth (resp., locally $\bQ_p$-algebraic) subrepresentation consisting of its smooth (resp., locally $\bQ_p$-algebraic) vectors; see \cite{Emerton2007summary} for details.\;Let $\pi_P$ be either a locally $\bQ_p$-analytic representation of $P$ on a locally convex $E$-vector space of compact type or a smooth representation of $P$ over $E$.\;We denote by
\begin{equation}\label{smoothadj2}
	\begin{aligned}
		&(\mathrm{Ind}_{P}^{G}\pi_P)^{\bQ_p-\ana}:=\{f:G\rightarrow \pi_P \text{\;locally $\bQ_p$-analytic},\;f(pg)=pf(g)\},\\
		&\text{resp.,\;}(\mathrm{Ind}_{P}^{G}\pi_P)^{\infty}=\{f:G\rightarrow \pi_P \text{\;smooth},\;f(pg)=pf(g)\}
	\end{aligned}
\end{equation}
the locally $\bQ_p$-analytic parabolic induction (resp., the unnormalized smooth parabolic induction) to $G$.\;It becomes a locally $\bQ_p$-analytic representation (resp., a smooth representation) of $G$ over $E$ by equipping it with the left action of $G$ by right translation on functions: $(gf)(g')=f(g'g)$.\;


Let $r$ be an integer, and let $\Omega_r$ be a cuspidal Bernstein component of $\GLN_r(L)$.\;Let $\FZ_{\Omega_{r}}$ be the rational Bernstein centre over $E$ (see \cite[Section 3.2]{PATCHING2016}).\;For a closed point $x\in \Spec\FZ_{\Omega_{r}}$, let $\pi_x$ be the associated irreducible cuspidal smooth representation over $k(x)$.\;Then, via the normalized classical local Langlands correspondence (see \cite{scholze2013local}), $\pi_x$ gives rise to an $r$-dimensional absolutely irreducible Weil--Deligne representation $\wdre_x:=\rec(\pi_x)$ of $W_L$ over $k(x)$.\;$\wdre_x$ corresponds to an absolutely irreducible Deligne--Fontaine module $\df_x$, by Fontaine's equivalence of categories as in \cite[Proposition 4.1]{breuil2007first}, and a $p$-adic differential equation ${\Delta_x}$ over $\cR_{k(x),L}$.\;Here ${\Delta_x}$ is a de Rham $(\varphi,\Gamma)$-module of rank $r$ over $\cR_{k(x),L}$ with constant Hodge--Tate weights $0$ such that, after forgetting the Hodge filtration, $D_{\mathrm{pst}}(\Delta_x)$ is isomorphic to $\df_x$, by Berger's theory \cite[Theorem A]{berger2008equations}.\;Assume that $\wdre_x$ is unramified when restricted to $W_{L'}$ for some finite Galois extension $L'$ of $L$.\;Then $\df_x=(\varphi_x,N=0,\Gal(L'/L),\df_x)$, where
$\varphi_x$ is the Frobenius semilinear operator on $\df_x$.\;

Throughout the paper, we use $\bullet-\bullet$ to denote an extension of two objects, such as $(\varphi,\Gamma)$-modules or $G$-representations, where the first (resp., second) object is the subobject (resp., quotient).\;

\subsection{\texorpdfstring{Potentially crystalline $(\varphi,\Gamma)$-modules over $\cR_{E,L}$}{Potentially crystalline (phi,Gamma)-modules over R E,L}}\label{Omegafil}

We give a quick review, following \cite[Section 2.3]{Ding2021}, of the structure of \textit{non-critical and generic} potentially crystalline $(\varphi,\Gamma)$-modules over $\cR_{E,L}$.\;

Throughout this article, fix $S_0\subseteq\Delta$ and write $\bL_{S_0}:=\prod_{i=1}^s\GLN_{r_i}$ for integers $r_1,\cdots,r_s$ such that $n=r_1+\cdots+r_s$.\;For $1\leq i\leq s$, fix a cuspidal Bernstein component $\Omega_{i}$ of $\GLN_{r_i}$.\;Then $\Omega_{S_0}:=\prod_{i=1}^s\Omega_{i}$ is a cuspidal Bernstein component of $\bL_{S_0}$.\;For $u\in \sW_s$, let $S_0^u$ be the subset of $\Delta$ such that $\bL_{S_0^u}=\prod_{i=1}^s\GLN_{r_{u^{-1}(i)}}$.\;Put $\Omega_{S^u_0}:=\prod_{i=1}^s\Omega_{u^{-1}(i)}$, which is a cuspidal Bernstein component of $\bL_{S^u_0}$.\;For $1\leq i\leq s$, put $t_0^u=0$ and $t_i^u=\sum_{j=1}^ir_{u^{-1}(j)}$, so that $t^u_s=n$.\;Then $\Delta\backslash S_0^u=\{t_i^u:i\in \Delta_s\}$, where $\Delta_s=\{1,\cdots,s-1\}$.\;



Let $\Dpik$ be a potentially crystalline $(\varphi,\Gamma)$-module over $\cR_{E,L}$ of rank $n$.\;Let $L'$ be a finite Galois extension of $L$ such that $\Dpik|_{L'}:=\Dpik\otimes_{\cR_{E,L}}\cR_{E,L'}$ is a crystalline $(\varphi,\Gamma)$-module over $\cR_{E,L'}$ of rank $n$.\;We consider the filtered  Deligne--Fontaine module associated with $\Dpik$:
\[\df_{\Dpik}=(\varphi,\Gal(L'/L),D_{\pcr}(\Dpik),\fil_H^{\bullet}(D_{L'})),\]
where $D_{\pcr}(\Dpik)=D_{\mathrm{cr}}^{L'}(\Dpik|_{L'})$ is a finite free $L'_0\otimes_{\bQ_p}E$-module of rank $n$, and $L'_0$ is the maximal unramified subextension of $L'$ over $\bQ_p$.\;The $\varphi$-action on $D_{\pcr}(\Dpik)$ is induced by the $\varphi$-action on $B_{\mathrm{cris}}$, and the $\Gal(L'/L)$-action on $D_{\pcr}(\Dpik)$ is the residual action of $\gal_L$.\;Moreover, $D_{L'}:=D_{\pcr}(\Dpik)\otimes_{L_0'}L'$ and the  filtered structure $\fil_H^{\bullet}(D_{L'})$ is  induced by the filtration on $B_{\mathrm{cris}}$.\;Note that $D_{L'}=\prod_{\sigma\in \Sigma_L}D_{L',\sigma}$ with $D_{L',\sigma}:=D_{L'}\otimes_{L,\sigma}E$ a $L'\otimes_{L,\sigma}E$-module  and $\fil_H^{\bullet}(D_{L'})=\prod_{\sigma\in \Sigma_L}\fil_H^{\bullet}(D_{L',\sigma})$.\;Since $\Dpik$ is potentially crystalline, it is de Rham.\;Hence $D_{\dr}(\Dpik)\cong D_{L'}^{\Gal(L'/L)}$ is a free $L\otimes_{\bQ_p}E$-module of rank $n$,\;which  is equipped with a natural Hodge filtration  $\fil^{\bullet}_{H}D_{\dr}(\Dpik):=(\fil^{\bullet}_{H}(D_{L'})^{\Gal(L'/L)}$.\;In this article, we define the Hodge--Tate weights of a de Rham representation as the opposites of the jumps of the filtration on the covariant de Rham functor, so that the Hodge--Tate weight of ${\ccyc}$ is $1$.\;Assume that $D_{\dr}(\Dpik)$ has distinct Hodge--Tate weights $\bh:=(\hpi_{1,\sigma}>\hpi_{2,\sigma}>\cdots>\hpi_{n,\sigma} )_{\sigma\in \Sigma_L}$.\;Denote $\hpi_{i}=(\hpi_{i,\sigma})_{\sigma\in \Sigma_L}$ for $1\leq i\leq n$.\;

Let $\wdre_{\Dpik}$ be the Weil--Deligne representation associated with $\df_{\Dpik}$ by Fontaine's equivalence of categories; see \cite[Proposition 4.1]{breuil2007first}.\;We say that $\wdre_{\Dpik}$ admits an $\Omega_{S_0}$-filtration $\cF$ if $\wdre_{\Dpik}$ admits an increasing filtration $\cF$ by Weil--Deligne subrepresentations:
\[\cF=\fil_{\bullet}^{\cF} \wdre_{\Dpik}: \ 0 =\fil_0^{\cF} \wdre_{\Dpik} \subsetneq \fil_1^{\cF} \wdre_{\Dpik} \subsetneq \cdots \subsetneq \fil_{s}^{\cF} \wdre_{\Dpik}=\wdre_{\Dpik},\]
such that the irreducible smooth cuspidal representation associated with $\gr^{\cF}_{i}\wdre_{\Dpik}$ lies in the cuspidal Bernstein component $\Omega_{i}$ for all $1\leqslant i\leqslant s$.\;By \cite[Proposition 4.1]{breuil2007first}, the $\Omega_{S_0}$-filtration $\cF$ on $\wdre_{\Dpik}$ corresponds to an increasing $\Omega_{S_0}$-filtration on $\df_{\Dpik}$, still denoted by $\cF$, by Deligne--Fontaine submodules
\[\cF=\fil_{\bullet}^{\cF}\df_{\Dpik}: \ 0 =\fil_0^{\cF} \df_{\Dpik} \subsetneq \fil_1^{\cF} \df_{\Dpik} \subsetneq \cdots \subsetneq \fil_{s}^{\cF} \df_{\Dpik}=\df_{\Dpik},\]
such that $\fil_{i}^{\cF} \df_{\Dpik}$ corresponds to $\fil_{i}^{\cF} \wdre_{\Dpik}$ via \cite[Proposition 4.1]{breuil2007first}.\;Thereofer,\;for $u\in\sW_s$, we also have an $\omepik$-filtration $\cF_u$ on $\wdre_{\Dpik}$ (resp.,\;$\df_{\Dpik}$) by Weil--Deligne subrepresentations (resp.,\;Deligne--Fontaine submodules):
\begin{equation}
	\begin{aligned}
		\cF_u=\fil_{\bullet}^{\cF_u} \wdre_{\Dpik}: &\ 0 =\fil_0^{\cF_u} \wdre_{\Dpik} \subsetneq \fil_1^{\cF_u}\wdre_{\Dpik} \subsetneq \cdots \subsetneq \fil_{s}^{\cF_u} \wdre_{\Dpik}=\wdre_{\Dpik},\\
		\cF_u=\fil_{\bullet}^{\cF_u}\df_{\Dpik}: &\ 0 =\fil_0^{\cF_u} \df_{\Dpik} \subsetneq \fil_1^{\cF_u} \df_{\Dpik} \subsetneq \cdots \subsetneq \fil_{s}^{\cF_u} \df_{\Dpik}=\df_{\Dpik}
	\end{aligned}
\end{equation}
such that $\gr_{i}^{\cF_u} \df_{\Dpik}$ is associated with $\gr_{i}^{\cF_u} \wdre_{\Dpik}=\gr_{u^{-1}(i)}^{\cF}\wdre_{\Dpik}$.\;

Let ${\Delta_\Dpik}$ be the $p$-adic differential equation over $\cR_{E,L}$ associated with $\df_{\Dpik}$.\;The $\Omega_{S^u_0}$-filtration on $\df_{\Dpik}$ induces an $\Omega_{S^u_0}$-filtration $\fil_{\bullet}^{\cF_u}{\Delta_\Dpik}=\{\fil_{i}^{\cF_u}{\Delta_\Dpik}\}$ on ${\Delta_\Dpik}$ by saturated $(\varphi,\Gamma)$-submodules over $\cR_{E,L}$, such that $\fil_{i}^{\cF_u}{\Delta_\Dpik}$ is the $p$-adic differential equation over $\cR_{E,L}$ associated with $\fil_{i}^{\cF_u} \df_{\Dpik}$.\;Consider the $(\varphi,\Gamma)$-module $\cM_\Dpik=\Dpik[1/t]\cong \Delta_\Dpik [1/t]$ over $\cR_{E,L}[1/t]$.\;After inverting $t$, the filtration $\cF_u$ on $\Delta_\Dpik$ induces an increasing filtration $\fil_{i}^{\cF_u}\cM_\Dpik:=\fil_{i}^{\cF_u}{\Delta_\Dpik}[1/t]$ on $\cM_\Dpik$ by $(\varphi,\Gamma)$-submodules over $\cR_{E,L}\big[\frac{1}{t}\big]$.\;Therefore, the filtration $\cF_u$ on $\cM_\Dpik$ induces an increasing $\bP_{S^u_0}$-filtration on $\Dpik$:
\[\cF_u=\fil_{\bullet}^{\cF_u}\Dpik: \ 0 =\fil_0^{\cF_u}\Dpik \subsetneq \fil_1^{\cF_u}\Dpik \subsetneq \cdots \subsetneq \fil_{s}^{\cF_u} \Dpik=\Dpik,\;\fil_{i}^{\cF_u}\Dpik=(\fil_{i}^{\cF_u}\cM_\Dpik)\cap \Dpik,\]
by saturated $(\varphi,\Gamma)$-submodules of $\Dpik$ over $\cR_{E,L}$.\;

For each $\sigma\in \Sigma_L$,\;the natural Hodge filtration on $D_{\dr}(\Dpik)_{\sigma}:=D_{\dr}(\Dpik)\otimes_{L,\sigma}E\cong D_{L',\sigma}^{\Gal(L'/L)}$ can be expressed by the following complete flag :
\[\fil^{\bullet}_{H}D_{\dr}(\Dpik)_{\sigma}: \ 0 \subsetneq \fil^{-\hpi_{n,\sigma}}_{H} D_{\dr}(\Dpik)_{\sigma} \subsetneq \fil^{-\hpi_{n-1,\sigma}}_{H}D_{\dr}(\Dpik)_{\sigma} \subsetneq \cdots \subsetneq \fil^{{-\hpi_{1,\sigma}}}_{H}D_{\dr}(\Dpik)_{\sigma}=D_{\dr}(\Dpik)_{\sigma}.\]
On the other hand, the $\Omega_{S^u_0}$-filtration $\cF_u$ on $\df_{\Dpik}$ induces an $\Omega_{S^u_0}$-filtration $\cF_u$ on $D_{\dr}(\Dpik)$ by free $L\otimes_{\bQ_p}E$-submodules $\fil_{\bullet}^{\cF_u}D_{\dr}(\Dpik):=(\fil_{\bullet}^{\cF_u}\df_{\Dpik}\otimes_{L_0'}L')^{\Gal(L'/L)}$.\;By Berger's equivalence of categories, $\fil_{i}^{\cF_u}\Dpik$ corresponds to the filtered Deligne--Fontaine module $\fil_{i}^{\cF_u} \df_{\Dpik}$ equipped with the filtration induced from the Hodge filtration on $D_{\pcr}(\Dpik)$.\;The \textit{non-critical} assumption means that the Hodge--Tate weights of $\fil_{i}^{\cF_u}\Dpik$ (resp.,\;$\gr_{i}^{\cF_u} \Dpik$) are given by
\begin{equation}
	\begin{aligned}
		&\{\hpi_{1,\sigma},\hpi_{2,\sigma},\cdots,\hpi_{t_i^u,\sigma}\}_{\sigma\in \Sigma_L},\;
		\text{resp.,\;}&\{\hpi_{t^u_{i-1}+1,\sigma},\hpi_{t^u_{i-1}+2,\sigma},\cdots,\hpi_{t^u_i,\sigma}\}_{\sigma\in \Sigma_L}.\;
	\end{aligned}
\end{equation}
Using Berger's equivalence of categories \cite[Theorem A]{berger2008equations} and comparing the weights (or see \cite[(2.4)]{Ding2021}), we obtain an injection of $(\varphi,\Gamma)$-modules over $\cR_{E,L}$ of rank $r_i$:
\begin{equation}\label{Dpikinjectionu}
	\mathbf{I}^u_{i}:\gr_{i}^{\cF_u} \Dpik \hookrightarrow \gr^{\cF_u}_{i}{\Delta_\Dpik}\tee \cR_{E,L}(z^{\bh_{t^u_i}}),
\end{equation}
for $i=1,\cdots,s$.\;Put $E_{u,i}:=\gr_{i}^{\cF_u}\Dpik$ and $E_{i}:=E_{1,i}$.\;Then $\cF_u$ (resp.,\;$\cF:=\cF_1$) corresponds to the following successive  extension of $E_{u,i}$ (resp.,\;$E_{i}:=E_{1,i}$) for $1\leq i\leq s$ :
\begin{equation}\label{omegafilforu}
\Dpik=[E_{u,1}-E_{u,2}-\cdots-E_{u,s}],\;\quad (\text{resp.,\;}\Dpik=[E_{1}-E_{2}-\cdots-E_{s}]).\;
\end{equation}
For $1\leq i\leq s$, let $\df_{E_i}:=(\varphi_i,\Gal(L'/L),D_{\pcr}(E_i))$ (resp.,\;$\Delta_{E_i}$) be the underlying Deligne--Fontaine module (resp.,\;$p$-adic differential equation) associated with $E_i$.\;Then $\gr_{i}^{\cF_u} \df_{\Dpik}\cong \df_{E_{u^{-1}(i)}}$ and $\gr^{\cF_u}_{i}{\Delta_\Dpik}\cong \Delta_{E_{u^{-1}(i)}}$ for $1\leq i\leq s$.\;

In the sequel, let $\mathcal{Z}_{\bL_{S^u_0},L}$ (resp.,\;$\mathcal{Z}_{\bL_{S^u_0},\cO_L}$) be the rigid space over $E$ parametrizing continuous characters of $\bL_{S^u_0}(L)$ (resp.,\;$\bL_{S^u_0}(\cO_L)$).\;

\begin{dfn}\label{noncriticalassumption}
\begin{itemize}
	\item[(1)] We say that $\Dpik$ is \textit{regular and generic} if :
	\[\df_{E_i}\neq \df_{E_j},(\varphi_i,\Gal(L'/L),D_{\pcr}(E_i))\neq (p\varphi_j,\Gal(L'/L),D_{\pcr}(E_j))\]
	for $i\neq j$; equivalently, $\Delta_{E_i}\neq \Delta_{E_j}(\unr(q_L))=\Delta_{E_j}\otimes_{\cR_{E}}\cR_{E}(\unr(q_L))$ for $i\neq j$.\;
	\item[(2)] Suppose that $\Delta_{E_i}\cong \Delta_{x_i}$ for some $x_i\in \Spec\FZ_{\Omega_{i}}(E)$.\;For $u\in \sW_s$, put
	\begin{equation}
		\begin{aligned}
			&\underline{x}^u:=(x_{u^{-1}(i)})\in \sbanpiku,\underline{\delta}^u=(z^{\bh_{t^u_i}})_{1\leq i\leq s}\in \mathcal{Z}_{\bL_{S^u_0},L},\\
			&\underline{x}^u_+:=(x'_{u^{-1}(i)})\in \sbanpiku,\underline{\delta}^{0,u}:=\underline{\delta}^u|_{\bL_{S^u_0}(\cO_L)}\in \mathcal{Z}_{\bL_{S^u_0},\cO_L},\;
		\end{aligned}
	\end{equation}
	where $x'_{u^{-1}(i)}\in \Spec\FZ_{\Omega_{u^{-1}(i)}}$ satisfies $\pi_{x'_{u^{-1}(i)}}\cong \unr(\varpi_L^{\bh_{t^u_i}})\pi_{x_{u^{-1}(i)}}$ for $1\leq i\leq s$.\;In the language of \cite[Definition 4.1.6]{Ding2021},\;The $\Omega_{S^u_0}$-filtration $\cF_u$ has parameters  $(\underline{x}^u,\underline{\delta}^u)\in \sbanpiku\times\rigchlu$ (resp.,\;$(\underline{x}^u_+,\underline{\delta}^{0,u})\in \sbanpiku\times \mathcal{Z}_{\bL_{S^u_0},\cO_L}$).\;	
\end{itemize}

\end{dfn}

%

\subsection{(Parabolic) deformations of $(\varphi,\Gamma)$-modules}\label{genedeformations}

For a locally $L$-analytic group $H$, let $\homo(H,E)$ (resp.,\;$\homo_{\mathrm{sm}}(H,E)$) be the $E$-vector space of locally $\bQ_p$-analytic (resp., locally constant) $E$-valued characters on $H$.\;For $\sigma\in \Sigma_L$, let $\homo_{\sigma}(H,E)\subseteq \homo(H,E)$ be the subspace of locally $\sigma$-analytic $E$-valued characters on $H$.\;In particular, if $H=L^{\times}$, then by \cite[Section 1.3.1]{2015Ding}, $\homo(L^{\times},E)$ is generated by $\val_L$ and $\{\log_{p,\sigma}\}_{\sigma\in \Sigma_L}$, where $\log_p:L^{\times}\rightarrow L$ is the $p$-adic logarithm restricted to $\cO_L^{\times}$.\;Then $\homo_{\mathrm{sm}}(L^{\times},E)\cong E\val_L$, $\homo_{\sigma}(L^{\times},E)=E\langle\val_L,\log_{p,\sigma}\rangle$, and $\homo(L^{\times},E)=E\langle\val_L,\{\log_{p,\sigma}\}_{\sigma\in \Sigma_L}\rangle$.\;In particular, $\dim_E\homo(L^{\times},E)=1+d_L$, $\dim_E\homo_{\mathrm{sm}}(L^{\times},E)=1$, and $\dim_E\homo_{\sigma}(L^{\times},E)=2$.\;Moreover, a choice of $\log_p(p)\in E$ (we may choose $\log_p(p)=0$) defines a section $\homo_{\sigma}(\cO_L^{\times},E)\hookrightarrow \homo_{\sigma}(L^{\times},E)$ of the restriction map $\homo_{\sigma}(L^{\times},E)\rightarrow \homo_{\sigma}(\cO_L^{\times},E)$ by sending $\log_{p,\sigma}$ to its unique extension to $L^{\times}$ sending $p$ to $\log_p(p)$.\;

Let $\cG$ be an increasing $\bP_{\cG}$-parabolic filtration on $\Dpik$ by saturated $(\varphi,\Gamma)$-submodules of $\Dpik$ over $\cR_{E,L}$:
\[\cG=\fil_{\bullet}^{\cG}\Dpik: \ 0 =\fil_0^{\cG}\Dpik \subsetneq \fil_1^{\cG}\Dpik \subsetneq \cdots \subsetneq \fil_{h}^{\cG} \Dpik=\Dpik.\]
Put $M_i:=\gr_i^{\cG}\Dpik$; then $\Dpik=[M_1-M_2-\cdots-M_h]$.\;Let $\bL_{\cG}$ be the standard Levi subgroup of $\bP_{\cG}$, and let $\bZ_{\cG}\cong \BG_m^h$ be the center of $\bL_{\cG}$.\;Denote by $F_{\Dpik,\cG}$ the deformation functor
\[F_{\Dpik,\cG} :\Art(E):=\{\text{Artinian local $E$-algebra with residue field $E$}\}\longrightarrow \{\text{sets}\}\]
which sends $A$ to the set of isomorphism classes $\{(D_A,\pi_A,\cG_A)\}/\sim$, where
\begin{itemize}
	\item[(1)] $D_A$ is a $(\varphi,\Gamma)$-module over $\cR_{A,L}$ of rank $n$ with $\pi_A:D_A\otimes_AE\cong \Dpik$;
	\item[(2)] $\cG_A=\fil^{\cG_A}_{\bullet}D_A$ is an increasing filtration of $D_A$ by $(\varphi,\Gamma)$-submodules over $\cR_{A,L}$, such that the $\fil^{\cG_A}_{i}D_A$ are direct summands of $D_A$ as $\cR_{A,L}$-modules and $\pi_A(\fil^{\cG_A}_{i}D_A)\cong \fil_{i}^{\cG} \Dpik=M_i$.
\end{itemize}
Moreover, we identify $F_{\Dpik,\cG}(E[\epsilon]/\epsilon^2)$ with a subspace $\ext^1_{\cG}(\Dpik,\Dpik)\subseteq \ext^1(\Dpik,\Dpik)$ of $\cG$-parabolic deformations of $\Dpik$; namely, $\widetilde{D}\in \ext^1(\Dpik,\Dpik)$ belongs to $\ext^1_{\cG}(\Dpik,\Dpik)$ if and only if 
\[\widetilde{D}=[\widetilde{M}_{1}-\widetilde{M}_{2}-\cdots-\widetilde{M}_{h}],\;\]
where $\widetilde{M}_{i}$ is a deformation of $M_{i}$ over $E[\epsilon]/\epsilon^2$ for $1\leq i\leq h$.\;

We next recall the content of \cite[Section 4.1.1]{Ding2021}.\;For $1\leq i\leq h$, let $\Delta_{M_i}$ be the differential equation associated with $M_i$; it is a de Rham $(\varphi,\Gamma)$-module of rank $h_i$ over $\cR_{E,L}$ with constant Hodge--Tate weight $0$.\;Suppose that we have an injection $M_i\hooklongrightarrow \Delta_{M_i}\otimes_{\cR_{E,L}}\cR_{E,L}(\delta_i)$ of $(\varphi,\Gamma)$-modules over $\cR_{E,L}$ for some character $\delta_i:L^\times\rightarrow E^\times$.\;For each $\sigma\in \Sigma_L$, we usually assume that $0$ is the minimal $\sigma$-Hodge--Tate weight of $M_i\otimes_{\cR_{E,L}}\cR_{E,L}(\delta^{-1}_i)$.\;In \cite[Sections 4.1.1 \& 4.1.2]{Ding2021}, the authors consider the following functor:
\[F_{M_{i}}^0 (\text{resp.},\;F_{M_{i}}^{\flat}):\Art(E)\longrightarrow \{\text{sets}\}\]
which sends $A$ to the set of isomorphism classes $\{(D_A,\pi_A,\delta_A)\}/\sim$, where
\begin{description}
	\item[(1)] $D_A$ is a $(\varphi,\Gamma)$-module of rank $r$ over $\cR_{A,L}$ with $\pi_A:D_A\otimes_AE\cong M_i$;
	\item[(2)] ${\delta}_{A}:L^\times\rightarrow A^\times$ is a continuous character such that ${\delta}_{A}\equiv {\delta}_i\mod \mathfrak{m}_A$;
	\item[(3)] there exists an injection (resp., an isomorphism) $D_A\hooklongrightarrow \Delta_{M_i}\otimes_{\cR_{E,L}}\cR_{A,L}({\delta}_{A})$ (resp.,\;$D_A\cong M_i\otimes_{\cR_{E,L}}\cR_{A,L}({\delta}_{A})$) of $(\varphi,\Gamma)$-modules over $\cR_{A,L}$.
\end{description}

The map sending the data $\{(D_A,\pi_A,\cG_A)\}$ to $\{(\gr^{\cG_A}_iD_A,\pi_A|_{\gr_i^{\cG_A}D_A})\}$ induces a natural morphism $F_{\Dpik,\cF}\longrightarrow \prod_{i=1}^hF_{M_i}$.\;For $?\in\{\flat,0\}$, put
\[F_{\Dpik,\cG}^{?}=F_{\Dpik,\cG}\times_{\prod_{i=1}^h F_{M_i}}\prod_{i=1}^h F_{M_i}^{?},\]
i.e.\;$F_{\Dpik,\cG}^0$ (resp.,\;$F_{\Dpik,\cG}^{\flat}$) sends $A$ to the set of isomorphism classes $\{(D_A,\pi_A,\cG_A,\delta_A)\}/\sim$ satisfying the above conditions $(1)$--$(3)$ and
\begin{description}
	\item[(4)] For $1\leq i\leq h$, there is an injection (resp., an isomorphism) of $(\varphi,\Gamma)$-modules over $\cR_{A,L}$:
	\begin{equation}\label{injectiondefF0deforma}
		\gr^{\cG_A}_iD_A\hooklongrightarrow \Delta_{M_i}\otimes_{\cR_{E,L}}\cR_{A,L}({\delta}_{A,i}),\;\text{resp.\;}\gr^{\cG_A}_iD_A\cong M_{i}\otimes_{\cR_{E,L}}\cR_{A,L}({\delta}_{A,i}).\;
	\end{equation}
\end{description}
We have inclusions of functors $F_{\Dpik,\cG}^{\flat}\subseteq F_{\Dpik,\cG}^{0}\subseteq F_{\Dpik,\cG}$, and their $E[\epsilon]/\epsilon^2$-points correspond to the inclusions of $\ext^1$-spaces
\[\ext^{1,\flat}_{\cG}(\Dpik,\Dpik)\subseteq \ext^{1,0}_{\cG}(\Dpik,\Dpik)\subseteq \ext^1_{\cG}(\Dpik,\Dpik)\] 
such that $\widetilde{D}=[\widetilde{M}_{1}-\widetilde{M}_{2}-\cdots-\widetilde{M}_{h}]$ belongs to $\ext^{1,0}_{\cG}(\Dpik,\Dpik)$ (resp.,\;$\ext^{1,\flat}_{\cG}(\Dpik,\Dpik)$) if $\widetilde{M}_{i}\hookrightarrow  \Delta_{M_i}\otimes_{\cR_{E,L}}\cR_{E[\epsilon]/\epsilon^2,L}(\delta_i(1+\psi_i\epsilon))$ (resp.,\;$\widetilde{M}_{i}\cong M_{i}\otimes_{\cR_{E,L}}\cR_{E[\epsilon]/\epsilon^2}(\delta_i(1+\psi_i\epsilon))$) for some $\psi_i\in \homo(L^{\times},E)$.\;We thus have natural maps
\begin{equation}
	\begin{aligned}
		\kappa^{?}_{u}:\ext^{1,?}_{\cG}(\Dpik,\Dpik)\rightarrow\homo(\bZ_{\cG}(L),E),\;\widetilde{D}\mapsto (\psi_i)_{1\leq i\leq h}.\;
	\end{aligned}
\end{equation}
for $?\in\{\flat,0\}$.\;In particular, if $\rank(M_i)=1$ for all $1\leq i\leq h$, the functors $\{F_{\Dpik,\cG}^{?}\}_{?\in\{\emptyset,\flat,0\}}$ and the extension groups $\{\ext^{1,?}_{\cG}(\Dpik,\Dpik)\}_{?\in\{\emptyset,\flat,0\}}$ coincide, and we drop the symbols $\{\flat,0\}$ from the subscripts.\;

\section{Deformations of potentially crystalline $(\varphi,\Gamma)$-modules}

Throughout this section, we fix a non-critical generic potentially crystalline $(\varphi,\Gamma)$-module $\Dpik$ with Hodge--Tate weights $\bh$ in Section \ref{Omegafil}  and keep the notation of Section \ref{Omegafil}.\;

\subsection{Computations on extension groups}

For $u\in\sW_s$, recall the $\omepik$-filtration $\cF_{u}$, written as $\Dpik:=[E_{u,1}-E_{u,2}-\cdots-E_{u,s}]$ in (\ref{omegafilforu}).\;We apply the discussion in Section \ref{genedeformations} to $\cG=\cF_{u}$.\;For $?\in\{\emptyset,\flat,0\}$, write $\ext^{1,?}_{u}(\Dpik,\Dpik):=\ext^{1,?}_{\cF_{u}}(\Dpik,\Dpik)$ for simplicity.\;We have the following inclusions of extension groups:
$\ext^{1,\flat}_{u}(\Dpik,\Dpik)\subseteq\ext^{1,0}_{u}(\Dpik,\Dpik)\subseteq \ext^{1}_{u}(\Dpik,\Dpik)$ and the natural maps
\begin{equation}
	\begin{aligned}
		\kappa^{?}_{u}:\ext^{1,?}_{u}(\Dpik,\Dpik)\rightarrow\homo(\bZ_{S^u_0}(L),E)
		\end{aligned}
\end{equation}
where $\bZ_{S^u_0}\cong \BG_m^s$ is the center of $\bL_{S^u_0}$ and $?\in\{\flat,0\}$.\;By the genericity assumption, $\kappa^{\flat}_{u}$ and $\kappa^{0}_{u}$ are surjective.\;By \cite[Proposition 4.1.17,\;Corollary 4.1.8]{Ding2021} and its proof, we have
\begin{pro}\label{dimesionforParafilgeneric}
\begin{itemize}
	\item[(1)] $\dim_E\ext^1(\Dpik,\Dpik)=1+d_Ln^2$ and $\dim_E\ext^1_u(\Dpik,\Dpik)=1+d_L\dim \bP_{S^u_0}$.\; 
	\item[(2)]  $\dim_E\ext^{1,\flat}_{u}(\Dpik,\Dpik)=1+d_L(\dim\bN_{S^u_0}+s) $ and $\dim_E\ext^{1,0}_{u}(\Dpik,\Dpik)=1+d_L(\frac{n(n-1)}{2}+s)$.\;
\end{itemize}
\end{pro}	


For $u\in \sW_{s}$ and $?\in\{\flat,0\}$,\;let  $\ext^{1,?}_{g,u}(\Dpik,\Dpik)\subseteq\ext^{1,?}_{u}(\Dpik,\Dpik)$  be the preimage of $\homo_{\sm}(\bZ_{S^u_0}(L),E)$ of via $\kappa^{?}_u$.\;Let $\ext^1_g(\Dpik,\Dpik)$ be the subspace of de Rham deformations of $\Dpik$.\;Note that $\dim_E{\ext}^{1}_g(\Dpik,\Dpik)=1+d_L\frac{n(n-1)}{2}$.\;
\begin{pro}\label{studyforkerkappau}
\begin{itemize}
	\item[(1)] $\ext^1_g(\Dpik,\Dpik)\subseteq \ext^{1}_{u}(\Dpik,\Dpik)$ and $\ext^{1,0}_{g,u}(\Dpik,\Dpik)=\ext^{1}_{g}(\Dpik,\Dpik)$.\;
	\item[(2)] For $u\in \sW_{s}$,\;we have $\ext^{1,\flat}_{g,u}(\Dpik,\Dpik)=\ext^{1,\flat}_{u}(\Dpik,\Dpik)\cap \ext^1_g(\Dpik,\Dpik)$.\;
\end{itemize}
	\end{pro}
\begin{proof}By \cite[Appendix,\;Proposition A.3]{Dingsocle},\;for any $i<j$,\;we deduce that $\dim_E\hH^1_g(\gal_L,W(E_{u^{-1}(i)}^{\vee}\otimes_{\cR_{E,L}}E_{u^{-1}(j)}))=0$,\;where $W(-)$ means the associated $E$-$B$ pair.\;The d\'{e}vissage argument shows the inclusion $\ext^1_g(\Dpik,\Dpik)\subseteq \ext^{1}_u(\Dpik,\Dpik)$.\;Since $\cR_{E[\epsilon]/\epsilon^2}(1+\psi_i\epsilon)$ is de Rham iff $\psi\in \homo_{\sm}(L^{\times},E)$,\;we deduce that $\ext^{1,\flat}_{u}(\Dpik,\Dpik)\cap \ext^1_g(\Dpik,\Dpik)\subseteq \ext^{1,\flat}_{g,u}(\Dpik,\Dpik)$.\;On the other hand,\;for $(\psi_i)\in \homo_{\sm}(\bZ_{S^u_0}(L),E)$,\;we show that \[\widetilde{D}=[	\widetilde{E}_{u,1}-\widetilde{E}_{u,2}-\cdots-	\widetilde{E}_{u,s}],\;\]
with $\widetilde{E}_{u,i}\cong \widetilde{E}_{u,i}\otimes_{\cR_{E,L}}\cR_{E[\epsilon]/\epsilon^2}(1+\psi_i\epsilon)$ and $\psi\in \homo_{\sm}(\bQ^{\times}_p,E)$ is de Rham.\;It suffices to show the de Rhamness of $\widetilde{D}_1^j:=[\widetilde{E}_{u,1}-\widetilde{E}_{u,2}-\cdots-\widetilde{E}_{u,j}]$ (with $j$ increasing) step by step.\;This is true for $j=1$ obviously.\;Suppose that the $\widetilde{D}_1^j$ is de Rham and we show the de Rhamness of $\widetilde{D}_1^{j+1}$.\;It suffices to show that each element in
\[\hH^1\left(\gal_L,W(\widetilde{D}')\right),\;\widetilde{D}':=\widetilde{D}_1^j\otimes_{\cR_{E,L}}\widetilde{E}_{u,j+1}^{\vee}\otimes_{\cR_{E,L}}\cR_{E[\epsilon]/\epsilon^2,L}(1-\psi_{j+1}\epsilon)\]
is de Rham.\;However,\;since all the Hodge-Tate weights of $\widetilde{D}'$ are positive,\;we deduce from \cite[Proposition A.3]{Dingsocle}  that $\hH_g^1(\gal_L,W(\widetilde{D}'))=\hH^1(\gal_L,W(\widetilde{D}'))$ (note that $\widetilde{H}^2_{\Sigma_L}(\gal_L,W(\widetilde{D}'))={H}^2(\gal_L,W(\widetilde{D}'))$).\;We thus obtain $\ext^{1,\flat}_{g,u}(\Dpik,\Dpik)\subseteq\ext^{1,\flat}_{u}(\Dpik,\Dpik)\cap \ext^1_g(\Dpik,\Dpik)$.\;Since $\ext^1_g(E_{u,i},E_{u,i})\subseteq \ext^{1,0}(E_{u,i},E_{u,i})$,\;we get that $\ext^{1,0}_{g,u}(\Dpik,\Dpik)=\ext^{1}_{g}(\Dpik,\Dpik)$. 
\end{proof}

The following lemma is an analogue of \cite[Lemma 2.11]{ParaDing2024}.\;For any $u_1,u_2\in \sW_s$,\;under the isomophism $\ext^{1,0}_{g,u_1}(\Dpik,\Dpik)\cong \ext^{1}_{g}(\Dpik,\Dpik)\cong \ext^{1,0}_{g,u_2}(\Dpik,\Dpik)$,\;we can identity $\ext^{1,\flat}_{g,u_1}(\Dpik,\Dpik)$ with $ \ext^{1,\flat}_{g,u_2}(\Dpik,\Dpik)$.\;Put
$\ext^{1,\flat}_{g,\{u_1,u_2\}}(\Dpik,\Dpik):=\ext^{1,\flat}_{g,u_1}(\Dpik,\Dpik)\cap \ext^{1,\flat}_{g,u_2}(\Dpik,\Dpik)\cong \ext^{1,\flat}_{g,u_1}(\Dpik,\Dpik)\cong \ext^{1,\flat}_{g,u_2}(\Dpik,\Dpik)$.\;

\begin{lem}\label{smhomoindepent}For any $u_1,u_2\in \sW_s$, the following diagram commutes:
\begin{equation}
	\xymatrix{ \ext^{1,\flat}_{g,\{u_1,u_2\}}(\Dpik,\Dpik)\ar@{=}[d] \ar[r]_{\kappa_{u_1}} & \homo_{\sm}(\bZ^{u_1}_{S_0}(L),E) \ar[d]^{\sim}_{u_2u_1^{-1}}\\
		\ext^{1,\flat}_{g,\{u_1,u_2\}}(\Dpik,\Dpik) \ar[r]_{\kappa_{u_2}} & \homo_{\sm}(\bZ^{u_2}_{S_0}(L),E) }.
\end{equation}
\end{lem}
\begin{proof}
It suffices to prove the statement when $u_2u_1^{-1}=s_k$ is a simple reflection.\;Let $\widetilde{D}\in  \ext^{1,\flat}_{g,\{u_1,u_2\}}(\Dpik,\Dpik)$, and suppose that
\[\widetilde{D}=\widetilde{E}_{u_1,1}-\widetilde{E}_{u_1,2}-\cdots-\widetilde{E}_{u_1,s},\]
with $\widetilde{E}_{u_1,i}\cong E_{u_1,i}\otimes_{\cR_{E,L}}\cR_{E[\epsilon]/\epsilon^2}(1+\psi_i\epsilon)$.\;By assumption, $u^{-1}_1(j)=u^{-1}_2(j)$ for $j\neq k,k+1$.\;Thus, for $j<k$ or $j>k+1$, we have $\fil^j_{\cF_{u_1}}\widetilde{D}=\fil^j_{\cF_{u_2}}\widetilde{D}$.\;Since $\homo(\widetilde{E}_{u_1,1},\widetilde{D})\cong \homo(\widetilde{E}_{u_1,1},\widetilde{E}_{u_1,1})\cong E[\epsilon]/\epsilon^2$, d\'{e}vissage for $\fil^j_{\cF_{u_2}}\widetilde{D}$ gives $\homo(\widetilde{E}_{u_1,1},\widetilde{E}_{u_2,1})\cong E[\epsilon]/\epsilon^2$, hence $\widetilde{E}_{u_1,1}=\widetilde{E}_{u_2,1}$.\;Now consider $\widetilde{D}/\widetilde{E}_{u_1,1}$ and repeat the same argument.\;We obtain $\widetilde{E}_{u_1,j}=\widetilde{E}_{u_2,j}$ for $j<k$.\;For $j=k$, we see that $\homo(\widetilde{E}_{u_1,k},\widetilde{E}_{u_2,k+1})\cong E[\epsilon]/\epsilon^2$, so $\widetilde{E}_{u_1,k}=\widetilde{E}_{u_2,k+1}$.\;Exchanging $\cF_{u_1}$ and $\cF_{u_2}$ gives $\widetilde{E}_{u_1,k+1}=\widetilde{E}_{u_2,k}$.\;For $j>k+1$, we consider $\widetilde{D}/\fil^{k+1}_{\cF_{u_1}}$ and see that $\widetilde{E}_{u_1,j}=\widetilde{E}_{u_2,j}$.\;This completes the proof.\;
\end{proof}

For any $u\in\sW_s$, let $\ext^{1,\flat}_{0,u}(\Dpik,\Dpik):=\ker\kappa^{\flat}_{u}\subseteq\ext^{1,\flat}_{g,u}(\Dpik,\Dpik)$.\;The above lemma shows that $\ker\kappa^{\flat}_{u}$ is independent of the choice of $u$; we denote it by $\ext^{1,\flat}_0(\Dpik,\Dpik)$.\;Note that 
\[ \dim_E{\ext}^{1}_0(\Dpik,\Dpik)=1+d_L\dim\bN_{S^u_0}-s=1+d_L\dim\bN_{S_0}-s.\]
For any subspace $V\subseteq \ext^1(\Dpik,\Dpik)$, let $\overline{V}:=V/(V\cap \ext^{1,\flat}_0(\Dpik,\Dpik))$.\;Therefore, $\kappa_{u}^{\flat}$ induces isomorphisms
\[\overline{\ext}^{1,\flat}_{u}(\Dpik,\Dpik)\xrightarrow{\sim }\homo(\bZ_{S^u_0}(L),E),\;\overline{\ext}^{1,\flat}_{g,u}(\Dpik,\Dpik)\xrightarrow{\sim }\homo_{\sm}(\bZ_{S^u_0}(L),E)\]
and (comparing with the short exact sequence in (\ref{commutativefortriandwhole}))
\[\dim_E\overline{\ext}^{1}(\Dpik,\Dpik)=d_L\dim_E\bP_{S_0}+s.\]

\subsection{Reinterpretation and supplements for deformations}\label{reinterfor33}

We recall briefly Fontaine's theory of almost de Rham representations.\;Let $B_{\pdr}^+:=B_{\dr}^+[\log t]$ and $B_{\pdr}:=B_{\pdr}^+[1/t]$.\;The $\gal_L$-action on $B_{\dr}$ extends uniquely to an action of $\gal_L$ on $B_{\pdr}$ with $g(\log t)=\log t+\log(\ccyc(g))$.\;Let $v_{\pdr}$ denote the unique $B_{\dr}$-linear derivation of $B_{\pdr}$ such that $v_{\pdr}(\log t)=-1$.\;Note that $v_{\pdr}$ and $\gal_L$ commute, and both preserve $B_{\pdr}^+$.\;

Let $\mathrm{Rep}_{B_{\dr}}(\gal_L)$ (resp.,\;$\mathrm{Rep}_{B^+_{\dr}}(\gal_L)$) be the category of finite free $B_{\dr}$-representations (resp., $B^+_{\dr}$-representations) of $\gal_L$.\;If $W\in \mathrm{Rep}_{B_{\dr}}(\gal_L)$, let $D_{\pdr}(W):=(B_{\pdr}\otimes_{B_{\dr}}W)^{\gal_L}$; this is a finite-dimensional $L$-vector space of dimension at most $\dim_{B_{\dr}}W$.\;The $B_{\dr}$-representation $W$ is called \textit{almost de Rham} if $\dim_LD_{\pdr}(W)=\dim_{B_{\dr}}W$.\;The $B^+_{\dr}$-representation $W^+$ is called \textit{almost de Rham} if $W^+[1/t]$ is almost de Rham.\;Let $\mathrm{Rep}_{\pdr}(\gal_L)$ be the category of almost de Rham $B_{\dr}$-representations of $\gal_L$.\;

Keep the notation of Section \ref{Omegafil}.\;In this section, we use the language of $E$-$B$-pairs (see \cite{nakamura2009classification}).\;Recall that the $(\varphi,\Gamma)$-module $\cM_\Dpik:=\Dpik[1/t]$ over $\cR_{E,L}[1/t]$ admits an $\Omega_{S_0^u}$-filtration $\cF_u$ with graded pieces $E_{u^{-1}(i)}[1/t]$ for $1\leq i\leq s$ and $\Dpik=[E_{1}-E_{2}-\cdots-E_{s}]$.\;Let $\bW_{\Dpik}=W_{\dr}(\cM_\Dpik)$ (resp.,\;$\bW^+_{\Dpik}:=W_{\dr}^+(\Dpik)$) be the $\bB_\dr\otimes_{\bQ_p}E$-representation (resp.,\;$\bB_\dr^+\otimes_{\bQ_p}E$-representation) of $\gal_L$ associated with $\cM_\Dpik$.\;Note that $D:=D_{\pdr}(\bW_\Dpik)\cong (L\otimes_{\bQ_p}E)^{\oplus n}$ for simplicity.\;For $\sigma\in \Sigma_L$, put $D_{\sigma}:=D_{\pdr,\sigma}(\bW_\Dpik):=D_{\pdr}(\bW_\Dpik)\otimes_{L\otimes_{\bQ_p}E}(L\otimes_{L,\sigma}E)$.\;Then $D=\prod_{\sigma\in \Sigma_L}D_{\sigma}$.\;For $1\leq i\leq s$, put $D^i:=D_{\pdr}(\bW_{E_i})$ and $D^i_{\sigma}:=D_{\pdr,\sigma}(\bW_{E_i})$.\;Then $D^i=\prod_{\sigma\in \Sigma_L}D^i_{\sigma}$.\;

For $u\in \sW_s$,\;the filtration $\cF_u$ on $\cM_{\Dpik}$ induces a $\bP_{S^u_0}$-filtration:
\[\bF_{u}=\fil_{\bullet}^{\bF_{u}}\bW_\Dpik: \ 0 =\fil_0^{\bF_{u}}\bW_\Dpik \subsetneq \fil_1^{\bF_{u}}\bW_\Dpik \subsetneq \cdots \subsetneq \fil_{s}^{\bF_{u}}\bW_\Dpik=\bW_\Dpik\]
on $\bW_\Dpik$ by $(\bB_\dr\otimes_{\bQ_p}E)$-subrepresentations of $\bW_{\Dpik}$.\;For $1\leq i\leq s$, $\bF_{u}$ has graded pieces $\gr_i^{\bF_u}\bW_{\Dpik}\cong \bW_{E_{u^{-1}(i)}}\cong (\bB_\dr\otimes_{\bQ_p}E)^{\oplus r_{u^{-1}(i)}}$.\;We thus get a $\bP_{S^u_0}$-filtration on $D$ and on each $D_{\sigma}$:
\begin{equation}
	\begin{aligned}
		\bF_{u}=\fil_{\bullet}^{\bF_{u}}(D)&: \ 0 =\fil_0^{\bF_{u}}(D) \subsetneq \fil_1^{\bF_{u}}(D) \subsetneq \cdots \subsetneq \fil_{s}^{\bF_{u}}(D)=D,\\
		\bF_{u,\sigma}=\fil_{\bullet}^{\bF_{u,\sigma}}(D_{\sigma})&: \ 0 =\fil_0^{\bF_{u,\sigma}}(D_{\sigma}) \subsetneq \fil_1^{\bF_{u,\sigma}}(D_{\sigma}) \subsetneq \cdots \subsetneq \fil_{s}^{\bF_{u,\sigma}}(D_{\sigma})=D_{\sigma}.
	\end{aligned}
\end{equation}
Note that $\gr_i^{\bF_u}D_{\pdr}(\bW_\Dpik)\cong D^i\cong (L\otimes_{\bQ_p}E)^{\oplus r_{u^{-1}(i)}}$.\;We have $\bF_{u}=\prod_{\sigma\in \Sigma_L}\bF_{u,\sigma}$.\;On the other hand, the $B_{\dr}^+$-lattice $\bW^+_\Dpik$ induces complete flags of $D_{\sigma}$ and $D$, respectively (see \cite[Section 6.3]{Ding2021}):
\begin{equation}
	\begin{aligned}
		\fil^{\bullet}_{H}(D_{\sigma})&: \ 0 \subsetneq \fil^{-\hpi_{n,\sigma}}_{H} (D_{\sigma}) \subsetneq \fil^{-\hpi_{n-1,\sigma}}_{H}(D_{\sigma}) \subsetneq \cdots \subsetneq \fil^{{-\hpi_{1,\sigma}}}_{H}(D_{\sigma})=D_{\sigma},\\
	\end{aligned}
\end{equation}
where $\fil^{-\bh_{n+1-i,\sigma}}_{H}(D_{\sigma}):=\big(t^{-\bh_{n+1-i,\sigma}}D^+_{\sigma}\big)^{\gal_L}$ and $D^+_{\sigma}:=D_{\pdr}(\bW^+_\Dpik)\otimes_{L\otimes_{\bQ_p}E}(L\otimes_{L,\sigma}E)$.\;

In the sequel,\;for $1\leq i\leq s$,\;let $\{e_{i,j,\sigma}\}_{1\leq j\leq r_i}$ be a fixed $E$-basis of $D^i_{\sigma}$.\;Under the $\bP_{S_0}$-filtration $\bF_1$ on $D_{\sigma}$,\;we use the following ordering of the basis $\{e_{1,j,\sigma}\}_{1\leq j\leq r_1},\cdots,\{e_{s,j,\sigma}\}_{1\leq j\leq r_s}$ of $D_{\sigma}$ (i.e.,\;the lexicographical order),\;and relabel  them with $e_{1,\sigma},\cdots,e_{n,\sigma}$.\;

Let $\homo_{\fil}(D_{\sigma},D_{\sigma})$ be the subspace of $E$-linear endomorphisms of $D_{\sigma}$ that preserve the filtration $\fil_{H}^{\bullet}(D_{\sigma})$.\;By the same argument as in the discussion before \cite[(112)]{BDcritical25}, we obtain canonical $E$-linear morphisms:
\begin{equation}\label{maptofilDsigma}
	\begin{aligned}
		\nu:	&\;\ext^{1}(\Dpik,\Dpik)\rightarrow \homo_{\fil}(D,D):=\bigoplus_{\sigma\in\Sigma_L}\homo_{\fil}(D_{\sigma},D_{\sigma}),\\
		\nu_{\sigma}:	&\;\ext^{1}_{\sigma}(\Dpik,\Dpik)\rightarrow \homo_{\fil}(D_{\sigma},D_{\sigma}).\;
	\end{aligned}
\end{equation}
where $\ext^{1}_{\sigma}(\Dpik,\Dpik)\subseteq \ext^{1}(\Dpik,\Dpik)$ is the subspace consisting of deformations that are $\Sigma_L\backslash \{\sigma\}$-de Rham.\;
By \cite[Lemma 2.4.1]{BDcritical25}, $\nu$ and $\nu_{\sigma}$ has kernel $\ext^{1}_g(\Dpik,\Dpik)$.\;

For $u\in \sW_s$, let $\homo_{\bF_u}(D_{\sigma},D_{\sigma})$ be the subspace of endomorphisms of $D_{\sigma}$ that preserve the filtration $\bF_{u,\sigma}$.\;For $?\in \{\flat,0\}$,\;put
\[\homo^{?}_{\bF_u}(D_{\sigma},D_{\sigma}):=\left\{f\in \homo_{\bF_u}(D_{\sigma},D_{\sigma}):\;f|_{\gr_{i}^{\bF_u}(D_{\sigma})} \text{\;scalar,\;}\;\forall \;1\leq i\leq s\right\}\subseteq \homo_{\bF_u}(D_{\sigma},D_{\sigma}).\]
For $?\in \{\emptyset,\flat,0\}$,\;put
\begin{equation}
	\begin{aligned}
		\homo^{?}_{\fil,\bF_u}(D_{\sigma},D_{\sigma}):=&\;\homo_{\fil}(D_{\sigma},D_{\sigma})\cap \homo^{?}_{\bF_u}(D_{\sigma},D_{\sigma}),\\
		\homo^{?}_{\fil,\bF_u}(D,D):=&\;\bigoplus_{\sigma\in\Sigma_L}\homo^{?}_{\fil,\bF_u}(D_{\sigma},D_{\sigma}).\;
	\end{aligned}
\end{equation}
For $?\in\{\flat,0\}$, put $\ext^{1,?}_{\sigma,u}(\Dpik,\Dpik):=\ext^{1,?}_{u}(\Dpik,\Dpik)\cap \ext^{1}_{\sigma}(\Dpik,\Dpik)$.\;Therefore,\;for $\star\in\{\emptyset,\sigma\}$, (\ref{maptofilDsigma}) induces
\begin{equation}\label{sendtofilD}
	\begin{aligned}
		\nu^?_{\star,u}:&\;\ext^{1,?}_{\star,u}(\Dpik,\Dpik)/\ext^{1,?}_{g,u}(\Dpik,\Dpik)\rightarrow\homo^{?}_{\fil,\bF_u}(D_{\star},D_{\star}),\\
		\nu_{\star,u}:&\;\ext^{1}_{\star,u}(\Dpik,\Dpik)/\ext^{1,?}_{g,u}(\Dpik,\Dpik)\rightarrow\homo_{\fil,\bF_u}(D_{\star},D_{\star}).\;
	\end{aligned}
\end{equation}

\begin{lem}
\begin{itemize}
	\item[(1)]
	For $1\leq i\leq s$,\;there is a  commutative diagram of short exact sequences of $E$-vector spaces: :
	\begin{equation}
		\xymatrix{
		0   \ar[r]  & \ext^{1,\flat}_g(E_i,E_i) \ar@{^(->}[d]\ar[r]  & \ext^{1,\flat}(E_i,E_i) \ar@{^(->}[d]\ar[r]  &  \bigoplus\limits_{\sigma\in\Sigma_L}E \ar@{=}[d] \ar[r] & 0  \\
		0    \ar[r] & \ext^{1,0}_g(E_i,E_i) \ar@{=}[d] \ar[r] & \ext^{1,0}(E_i,E_i)\ar[r] \ar@{^(->}[d]  & \bigoplus\limits_{\sigma\in\Sigma_L}E \ar[r]\ar@{^(->}[d] & 0 \\
		0    \ar[r] & \ext^1_g(E_i,E_i)  \ar[r] & \ext^{1}(E_i,E_i)\ar[r]   & \bigoplus\limits_{\sigma\in\Sigma_L}\homo_{\fil}(D^i_{\sigma},D^i_{\sigma}) \ar[r] & 0 }  
	\end{equation}	
		where $\ext^{1,\flat}_g(E_i,E_i)\cong \homo_{\sm}(L^{\times},E)$, and we have a natural inclusion $E\hookrightarrow \homo_{\fil}(D^i_{\sigma},D^i_{\sigma})$.\;There is also a $\sigma$-version obtained by dropping the symbol $\oplus_{\sigma\in\Sigma_L}$ and restricting to $\ext^{1,?}_{\sigma}(E_i,E_i)$.\;
	\item[(2)] For $1\leq i\leq s$,\;$\ext^{1}_g(E_i[1/t],E_i[1/t])\cong \homo_{\sm}(L^{\times},E)$.\;We have a natural map (by inverting $t$):
	\begin{equation}\label{mapblockinverset}
		\prod_{i=1}^s\ext^{1}_g(E_i,E_i)\rightarrow \prod_{i=1}^s\ext^{1}_g(E_i[1/t],E_i[1/t]).
	\end{equation}
		This map induces an isomorphism $\prod_{i=1}^s\ext^{1,\flat}_g(E_i,E_i)\xrightarrow{\sim} \prod_{i=1}^s\ext^{1}_g(E_i[1/t],E_i[1/t])$.\;
\end{itemize}
\end{lem}
In particular, (\ref{mapblockinverset}) is not injective if $S_0\neq\emptyset$ (resp., is an isomorphism if $S_0=\emptyset$).\;In the sequel,\;put 
\begin{equation}
	\begin{aligned}
		&\ext^{1,\flat}_{\gr}({\Dpik},{\Dpik}):=\prod_{i=1}^s\ext^{1,\flat}_g(E_i,E_i)\xrightarrow{\sim}\prod_{i=1}^s\ext^{1}_g(E_i[1/t],E_i[1/t]),\;\\
		&\ext^{1}_{\gr}({\Dpik},{\Dpik}):=\ext^{1,0}_{\gr}({\Dpik},{\Dpik}):=\prod_{i=1}^s\ext^{1}_g(E_i,E_i).
	\end{aligned}
\end{equation}
For any $1\leq i\leq s$,\;we have a decomposition $\ext^{1}(E_{i},E_{i})=\ext^{1,\flat}(E_{i},E_{i})\oplus \hH^1_{(\varphi,\Gamma)}(\EndO^0(E_{i}))$,\;where $\EndO^0(E_{i})=\EndO(E_{i})/\cR_{E,L}$.\;Therefore,\;$\ext^{1,\flat}_{\gr}({\Dpik},{\Dpik})$ is a direct summand of $\ext^{1}_{\gr}({\Dpik},{\Dpik})$ and thus we have a natural projection
\begin{equation}\label{projtocenterdefor}
	\ext^{1}_{\gr}({\Dpik},{\Dpik})\twoheadrightarrow \ext^{1,\flat}_{\gr}({\Dpik},{\Dpik}).\;
\end{equation}

For $?\in\{\emptyset,\flat,0\}$ and $\star\in\{\emptyset,\sigma\}$ for some $\sigma\in \Sigma_L$,\;consider the natural morphism
\begin{equation}\label{natmorphismforGDpikfern}
	g_{\Dpik,\star}^{?}:\bigoplus_{u\in \sW_{s}}{\ext}^{1,?}_{\star,u}(\Dpik,\Dpik)\rightarrow {\ext}_{\star}^1(\Dpik,\Dpik).\;
\end{equation}
For $?\in\{\flat,0\}$ and $\star\in\{\emptyset,\sigma\}$,\;let
${\ext}_{\star}^{1,?}(\Dpik,\Dpik)$ be the image of $g_{\Dpik,\star}^{?}$ and put $\nu_{\star}^{?}=\nu|_{{\ext}^{1,?}_{\star}(\Dpik,\Dpik)}$ and $\homo^{?}_{\fil}(D_{\star},D_{\star}):=\mathrm{Im}(\nu_{\star}^{?})$.\;

\begin{lem}\label{lemforexactforextgropus} 
\begin{itemize}
	\item[(1)]	There is an isomorphism
	\begin{equation}\label{expianextphi}
		{\ext}^{1,0}_g(\Dpik,\Dpik)/{\ext}^{1,\flat}_0(\Dpik,\Dpik)\xrightarrow{\sim}\ext^{1}_{\gr}({\Dpik},{\Dpik}).
	\end{equation}
	\item[(2)] For $u\in \sW_s$,\;$?\in\{\emptyset,\flat,0\}$ and $\star\in\{\emptyset,\sigma\}$,\;there exists a commutative diagrm of short exact sequences of $E$-vector spaces: 
	\begin{equation}\label{commutativefortriandwhole}
		\xymatrix{
			0   \ar[r]  & \ext^{1,?}_{\gr}({\Dpik},{\Dpik}) \ar@{=}[d]\ar[r]  & {\ext}^{1,?}_{\star,u}(\Dpik,\Dpik)/{\ext}^{1,\flat}_{0}(\Dpik,\Dpik) \ar@{^(->}[d]\ar[r]  &  \homo^{?}_{\fil,\bF_u}(D_{\star},D_{\star}) \ar@{^(->}[d] \ar[r] & 0  \\
			0    \ar[r] & \ext^{1,?}_{\gr}({\Dpik},{\Dpik}) \ar[r] & {\ext}_{\star}^{1,?}(\Dpik,\Dpik)/{\ext}^{1,\flat}_{0}(\Dpik,\Dpik)\ar[r]   &  \homo^{?}_{\fil}(D_{\star},D_{\star}) \ar[r] & 0  }.
	\end{equation}
\end{itemize}
\end{lem}
\begin{proof}Recall the dimensions of ${\ext}^{1,\flat}_0(\Dpik,\Dpik)$ and ${\ext}^{1,0}_g(\Dpik,\Dpik)={\ext}^{1}_g(\Dpik,\Dpik)$.\;These subspaces lead to a natural morphism ${\ext}^{1,0}_g(\Dpik,\Dpik)\rightarrow \ext^{1}_{\gr}({\Dpik},{\Dpik})$.\;Thus,\;we get $\dim_E{\ext}^{1,0}_g(\Dpik,\Dpik)/{\ext}^{1,\flat}_0(\Dpik,\Dpik)=s+d_L\dim_E\bL_{S_0}\cap\bN_{\emptyset}$,\;which is equal to the $E$-dimension of the right-hand side of (\ref{expianextphi});\;hence we get the isomorphism in $(1)$.\;
\end{proof}

Similar to \cite[(121)]{BDcritical25} and the commutative diagram below  \cite[(123)]{BDcritical25},\;we have maps
\[\ext^{1}(\Dpik,\Dpik)\rightarrow \ext^{1}(\cM_{\Dpik},\cM_{\Dpik})\rightarrow \bigoplus_{\sigma\in\Sigma_L}\homo_{E}(D_{\sigma},D_{\sigma}),\]
and thus obtain a commutative diagram of short exact sequences: 
\begin{equation}\label{diagDtoM}
	\begin{aligned}
		\xymatrix{
		0   \ar[r]  &  \ext^{1}_{\gr}({\Dpik},{\Dpik}) \ar[d]^{(\ref{projtocenterdefor})} \ar[r] & {\ext}^{1}(\Dpik,\Dpik)/{\ext}^{1,\flat}_{0}(\Dpik,\Dpik) \ar[d]\ar[r]  &  \bigoplus\limits_{\sigma\in\Sigma_L}\homo_{\fil}(D_{\sigma},D_{\sigma})\ar@{^(->}[d] \ar[r] & 0  \\
		0    \ar[r] &  \ext^{1,\flat}_{\gr}({\Dpik},{\Dpik}) \ar[r] & \ext^{1}(\cM_{\Dpik},\cM_{\Dpik})/{\ext}^{1,\flat}_{0}(\cM_{\Dpik},\cM_{\Dpik})\ar[r]   & \bigoplus\limits_{\sigma\in\Sigma_L}\homo_{E}(D_{\sigma},D_{\sigma})   \ar[r] & 0 }
	\end{aligned}
\end{equation}
where ${\ext}^{1,\flat}_{0}(\cM_{\Dpik},\cM_{\Dpik})$ is the image of ${\ext}^{1,\flat}_{0}(\Dpik,\Dpik)$ via the natural map $\ext^{1}(\Dpik,\Dpik)\rightarrow \ext^{1}(\cM_{\Dpik},\cM_{\Dpik})$.\;

Recall that $L'$ is a finite Galois extension of $L$ such that $\Dpik|_{L'}$ is a crystalline $(\varphi,\Gamma)$-module over $\cR_{E,L'}$ of rank $n$.\;Let $\cM_{\Dpik}|_{L'}:=\cM_{\Dpik}\otimes_{\cR_{E,L}[1/t]}\cR_{E,L'}[1/t]$ and $E_i[1/t]|_{L'}:=E_i[1/t]\otimes_{\cR_{E,L}[1/t]}\cR_{E,L'}[1/t]$ for $1\leq i\le s$.\;Note that $E_i[1/t]|_{L'}$ is crystalline for all $1\leq i\leq s$.\;For all $1\leq i\neq j\leq s$,\;we deduce from \cite[Proposition 2.7]{nakamura2009classification} that $\ext^1_e(E_i[1/t]|_{L'},E_j[1/t]|_{L'})=0$,\;and thus obtain
\[\cM_{\Dpik}|_{L'}\cong \oplus_{i=1}^sE_i[1/t]|_{L'}.\]
Similar to the argument in \cite[(122)-(123)]{BDcritical25},\;we have a commutative diagram of short exact sequences: 
\begin{equation}\label{exacttoL'}
	\begin{aligned}
		\xymatrix{
			0   \ar[r]  &  \ext^{1,\flat}_{\gr}({\Dpik},{\Dpik}) \ar[d]^{\sim}\ar[r]  & \ext^{1}(\cM_{\Dpik},\cM_{\Dpik})/{\ext}^{1,\flat}_{0}(\cM_{\Dpik},\cM_{\Dpik}) \ar[d]\ar[r]  &  \bigoplus\limits_{\sigma\in\Sigma_L}\homo_{E}(D_{\sigma},D_{\sigma}) \ar[d]^{\sim}    \\
			0    \ar[r] &  \ext^{1,\flat}_{\gr}({\Dpik}|_{L'},{\Dpik}|_{L'}) \ar[r] & \ext^{1}(\cM_{\Dpik}|_{L'},\cM_{\Dpik}|_{L'}) \ar[r]   & \bigoplus\limits_{\sigma\in\Sigma_L}\homo_{E}(D_{L',\sigma},D_{L',\sigma})    }
	\end{aligned}
\end{equation}
where $D_{\sigma}\cong D_{L',\sigma}^{\gal(L'/L)}$ and $\ext^{1,\flat}_{\gr}({\Dpik}|_{L'},{\Dpik}|_{L'})\cong \prod_{i=1}^s\ext^{1}_g(E_i[1/t]|_{L'},E_i[1/t]|_{L'})$.\;The first vertical map in (\ref{exacttoL'}) is an isomorphism since $\ext^{1}_g(E_i[1/t],E_i[1/t])\cong \ext^{1}_g(E_i[1/t]|_{L'},E_i[1/t]|_{L'})\cong \homo_{\mathrm{sm}}(L^{\times},E)$. By comparing dimensions,\;the last morphism in the bottom exact sequence of (\ref{exacttoL'}) is thus surjective,\;since $\dim_E\ext^{1}(\cM_{\Dpik}|_{L'},\cM_{\Dpik}|_{L'})=s+d_Ln^2$.\;Combining (\ref{diagDtoM}) with (\ref{exacttoL'}),\;we actually obtain the following commutative diagram (an analogue, in the potentially crystalline case, of the commutative diagram before \cite[Proposition 2.4.4]{BDcritical25}):
\begin{equation}\label{diagDtoML'}
	\begin{aligned}
		&\xymatrix{
			0   \ar[r]  &  \ext^{1}_{\gr}({\Dpik},{\Dpik}) \ar[d]^{\sim}\ar[r]  & {\ext}^{1}(\Dpik,\Dpik)/{\ext}^{1,\flat}_{0}(\Dpik,\Dpik) \ar[d]\ar[r]  &  \bigoplus\limits_{\sigma\in\Sigma_L}\homo_{\fil}(D_{\sigma},D_{\sigma})\ar@{^(->}[d] \ar[r] & 0  \\
			0    \ar[r] &   \ext^{1,\flat}_{\gr}({\Dpik}|_{L'},{\Dpik}|_{L'})\ar[r] & \ext^{1}(\cM_{\Dpik}|_{L'},\cM_{\Dpik}|_{L'})\ar[r]   & \bigoplus\limits_{\sigma\in\Sigma_L}\homo_{E}(D_{L',\sigma},D_{L',\sigma}) \ar[r]  & 0 }
	\end{aligned}
\end{equation}
Similar to \cite[Proposition 2.4.4]{BDcritical25},\;we see that
\begin{pro}\label{splitingforext1gps}
	There is a splitting of the bottom exact sequence in (\ref{exacttoL'}) which depends only on the choice of $\log_p(p)\in E$.\;Therefore,\;for $\ast\in \{\emptyset,u\}$ and $\star\in\{\emptyset,\sigma\}$,\;we have 
	\begin{equation}\label{splittingforextgps}
		\begin{aligned}
			&f^{\flat}_{\Dpik,\star}:{\ext}^{1,\flat}_{\star,\ast}(\Dpik,\Dpik)/{\ext}^{1,\flat}_{0,\ast}(\Dpik,\Dpik)\xrightarrow{\sim}\ext^{1,\flat}_{\gr}({\Dpik},{\Dpik})\oplus\mathrm{Im}(\nu^{\flat}_{\star,\ast}).\;
		\end{aligned}
	\end{equation}
\end{pro}
\begin{proof}By the same argument as in the proof of \cite[Proposition 2.4.4]{BDcritical25},\;the following map is also surjective:
\[\ext^{1}(\cM_{\Dpik}|_{L'},\cM_{\Dpik}|_{L'})\rightarrow \bigoplus_{i=1}^s\ext^{1}(E_i[1/t]|_{L'},E_i[1/t]|_{L'})\twoheadrightarrow \bigoplus_{i=1}^s\ext^{1}_g(E_i[1/t]|_{L'},E_i[1/t]|_{L'}).\]
	This gives the first assertion.\;The assertion in (\ref{splittingforextgps}) follows by replacing the top exact sequence in (\ref{diagDtoML'}) with the first row of  (\ref{commutativefortriandwhole}) (put $?=\flat$ and $\star\in\{\emptyset,\sigma\}$).\;
\end{proof}
\begin{rmk}The bottom exact sequence in (\ref{diagDtoM})
splits.\;For $1\leq i,j\leq s$,\;the natural injection $E_i[1/t]\hookrightarrow \cM_{\Dpik}$ and the surjection $\cM_{\Dpik}\twoheadrightarrow E_j[1/t]$ induce a natural map 
\[\ext^{1}(\cM_{\Dpik},\cM_{\Dpik})\rightarrow \oplus_{i,j=1}^s\ext^{1}(E_i[1/t],E_j[1/t]).\]	
This gives a splitting surjection $\ext^{1}(\cM_{\Dpik},\cM_{\Dpik})\rightarrow \oplus_{i=1}^s\ext^{1}(E_i[1/t],E_i[1/t])\twoheadrightarrow \oplus_{i=1}^s\ext^{1}_g(E_i[1/t],E_i[1/t])$. In general,\;the top exact sequence in (\ref{diagDtoML'}) may not admit a splitting.\;
\end{rmk}

Let $u^{\sharp}\in \sW_n$ be the unique element sending $e_{i,j,\sigma}$ for $1\leq j\leq r_i$ to $e_{u^{-1}(i),j,\sigma}$ for $1\leq j\leq r_i$,\;which thus induces a permutation on $e_{1,\sigma},\cdots,e_{n,\sigma}$.\;Therefore,\;we identify the space $\homo_{E}(D_{\sigma},D_{\sigma})$ with $\fg_{\sigma}$,\;so that
 $\homo_{\bF_u}(D_{\sigma},D_{\sigma})$ (resp.,\;$\homo^{\flat}_{\bF_u}(D_{\sigma},D_{\sigma})$) is identified with $\mathrm{Ad}_{(u^{\sharp})^{-1}}(\fp_{S^u_0,\sigma})$ (resp.,\;$\mathrm{Ad}_{(u^{\sharp})^{-1}}(\tau_{S^u_0,\sigma})$),\;where $\tau_{S^u_0,\sigma}=\fz_{S^u_0,\sigma}\ltimes\fn_{S^u_0,\sigma}$.\;Then we have 
\begin{equation}
	\begin{aligned}
		&\homo_{\fil}(D_{\sigma},D_{\sigma})=\mathrm{Ad}_{g_{\sigma}}(\fb_{\sigma}),\\
		&\homo_{\fil,\bF_u}(D_{\sigma},D_{\sigma})=\mathrm{Ad}_{(u^{\sharp})^{-1}}(\fp_{S^u_0,\sigma})\cap \mathrm{Ad}_{g_{\sigma}}(\fb_{\sigma}),\\
		&\homo^{\flat}_{\fil,\bF_u}(D_{\sigma},D_{\sigma})=\mathrm{Ad}_{(u^{\sharp})^{-1}}(\tau_{S^u_0,\sigma})\cap \mathrm{Ad}_{g_{\sigma}}(\fb_{\sigma}).\;
	\end{aligned}
\end{equation}
\begin{rmk}\label{explainHomasLiealg}For each $u\in \sW_s$,\;we deduce from the non-critical assumption that $g_{\sigma}\bB_{\sigma}=(u^{\sharp})^{-1}b_{\sigma}(u)w_{0,\sigma}\bB_{\sigma}$ for some $b_{\sigma}(u)\in \bB_{\sigma}$ (of course $b_{\sigma}(u)$ depends on $u$).\;Note that $\overline{\fb}_{\sigma}=\mathrm{Ad}_{w_0}(\fb_{\sigma})$ coincides with the Borel algebra of lower triangular matrices.\;We get that 
	\begin{equation}\label{noncriticalintersection}
		\begin{aligned}
			&\mathrm{Ad}_{(u^{\sharp})^{-1}}(\fp_{S^u_0,\sigma})\cap \mathrm{Ad}_{g_{\sigma}}(\fb_{\sigma})\cong \mathrm{Ad}_{(u^{\sharp})^{-1}b_{\sigma}(u)}(\fp_{S^u_0,\sigma}\cap \mathrm{Ad}_{w_0}(\fb_{\sigma}))=\mathrm{Ad}_{(u^{\sharp})^{-1}b_{\sigma}(u)}(\fl_{S^u_0,\sigma}\cap \overline{\fb}_{\sigma}),\\
			&\mathrm{Ad}_{(u^{\sharp})^{-1}}(\tau_{S^u_0,\sigma})\cap \mathrm{Ad}_{g_{\sigma}}(\fb_{\sigma})=\mathrm{Ad}_{(u^{\sharp})^{-1}b_{\sigma}(u)}(\fz_{S^u_0,\sigma}).
		\end{aligned}
	\end{equation}
\end{rmk}
By the proof of Theorem \ref{Inifinitefernpoten},\;we get that $\sum_{u\in \sW_s}\mathrm{Ad}_{(u^{\sharp})^{-1}}(\fp_{S^u_0,\sigma})\cap \mathrm{Ad}_{g_{\sigma}}(\fb_{\sigma})=\mathrm{Ad}_{g_{\sigma}}(\fb_{\sigma})$.\;Put
\begin{equation}
	\mathrm{Ad}_{g_{\sigma}}(\fb_{\sigma})_{S_0}^{\flat}:=\sum_{u\in \sW_s}\mathrm{Ad}_{(u^{\sharp})^{-1}}(\tau_{S^u_0,\sigma})\cap \mathrm{Ad}_{g_{\sigma}}(\fb_{\sigma})\subseteq \mathrm{Ad}_{g_{\sigma}}(\fb_{\sigma}).\;
\end{equation}
Note that $\dim_E\mathrm{Ad}_{g_{\sigma}}(\fb_{\sigma})_{S_0}^{\flat}$ depends on  the sharp of $\bP_{S_0}$.\;
\begin{thm} (Infinite fern in the non-critical potentially crystalline case)\label{infinitePoten} 
	\begin{itemize} 
		\item[(1)] For $\star\in\{\emptyset,\sigma\}$,\;the natural map  $g_{\Dpik,\star}:\bigoplus_{u\in \sW_{s}}{\ext}^{1}_{\star,u}(\Dpik,\Dpik)\rightarrow {\ext}^1_{\star}(\Dpik,\Dpik)$  is surjective.\;
		\item[(2)] For $?\in\{\flat,0\}$,\;${\ext}^{1,?}(\Dpik,\Dpik)/\ext^{1,?}_{g}({\Dpik},{\Dpik})\cong \bigoplus_{\sigma\in \Sigma_L}\mathrm{Ad}_{g_{\sigma}}(\fb_{\sigma})_{S_0}^{\flat}$.\;
	\end{itemize}
\end{thm}
\begin{proof}By commutative diagram (\ref{commutativefortriandwhole}),\;it suffices to show the following assertion for Lie-algebras :
		\[\sum_{u\in \sW_s}\mathrm{Ad}_{(u^{\sharp})^{-1}}(\fp_{S^u_0,\sigma})\cap \mathrm{Ad}_{g_{\sigma}}(\fb_{\sigma})=\mathrm{Ad}_{g_{\sigma}}(\fb_{\sigma}).\]
This is Lemma \ref{Inifinitefernpoten} .\;
\end{proof}


\section{Locally analytic representations in the potentially crystalline case}

Throughout this section,\;we fix a non-critical generic potentially crystalline $(\varphi,\Gamma)$-module $\Dpik$ with Hodge--Tate weights $\bh$ in Section \ref{Omegafil}.\;Let $\theta=(0,-1,\cdots,1-n)\in X_{\Delta}^+$ and put $\lambda=\bh-\theta=(\lambda_i=\bh_i+i-1)_{1\leq i\leq n}\in X_{\Delta}^+$.\;

For admissible locally $\bQ_p$-analytic representations $V_1,V_2$ (\cite{Coadmissible}),\;let $\ext^1_{G}(V_1,V_2):=\ext^1_{D(G,E)}(V_2^{\vee},V_1^{\vee})$, where $D(G,E)$ is the locally $\bQ_p$-analytic distribution algebra (so that the continuous strong duals $V_1^{\vee}$ and $V_2^{\vee}$ are coadmissible modules over $D(G,E)$),\;and this latter extension group is defined in the abelian category of abstract $D(G,E)$-modules.\;By \cite[Lemma 2.1.1]{breuil2019ext1},\;$\ext^1_{G}(V_1,V_2)$ is equal to the extension group of admissible locally $\bQ_p$-analytic representations.\;Suppose that $V_1,V_2$ have the same central character $\chi$.\;Let $\ext^1_{G,Z}(V_1,V_2)$ be the subspace of $\ext^1_{G}(V_1,V_2)$ of locally $\bQ_p$-analytic extensions with central character $\chi$.\;For $\sigma\in \Sigma_L$ and any $U(\fg_{\Sigma_L\backslash\{\sigma\}})$-finite representations $V_1,V_2$,\;we denote by $\ext^1_{G,\sigma}(V_1,V_2)\subseteq \ext^1_{G}(V_1,V_2)$ the subspace of extensions which are 
$U(\fg_{\Sigma_L\backslash\{\sigma\}})$-finite.\;Moreover,\;we put $\ext^1_{G,\sigma,Z}(V_1,V_2):=\ext^1_{G,\sigma}(V_1,V_2)\cap \ext^1_{G,Z}(V_1,V_2)$.\;

Denote by $\sW_{\Sigma_L}\cong \prod_{\sigma\in \Sigma_L}\sW_{n,\sigma}$  the Weyl group of $\res_{L/\bQ_p}\GLN_n$,\;where $\sW_{n,\sigma}\cong \sW_n$.\;For $i\in \Delta$ and $\sigma\in \Sigma_L$,\;let $s_{i,\sigma}\in \sW_{n,\sigma}$ be the simple reflection corresponding to $i\in \Delta$.\;

For any $I\subseteq \Delta$,\;write $\bL_I=\GLN_{h_1}\times\cdots\times \GLN_{h_r}$.\;Then
\[\delta_{\op_I}:=\prod_{i=1}^r|\mathrm{det}_{\GLN_{h_i}}|^{-\sum_{j=i+1}^rh_j+\sum_{j=1}^{i-1}h_j}=\prod_{i=1}^r|\mathrm{det}_{\GLN_{h_i}}|^{-n-h_i+2\sum_{j=1}^{i}h_j}\]
is the  modulus character of $\op_I$.\;Let $\eta_I=\delta_{\op_I}^{1/2}$.\;

For $u\in \sW_s$,\;recall the $E$-point $\underline{x}^u=(x_{u^{-1}(i)})_{1\leq i\leq s}\in \sbanpiku$.\;Note that $\pi_{x_i}$ is the smooth representation $\pi_{\mathrm{sm}}(E_{i})$ of $\GLN_{r_i}(L)$ over $E$ associated with $E_{i}$.\;For $u\in \sW_{s}$,\;put 
\[\pi_{\sm}(\underline{x}^u):=\boxtimes_{j=1}^s\pi_{\mathrm{sm}}(E_{u^{-1}(j)})=\boxtimes_{j=1}^s\pi_{x_{u^{-1}(j)}}.\]
For $S_0^u\subseteq I\subseteq \Delta$,\;consider the smooth principal series of $\bL_{I}(L)$ (if $I=\Delta$,\;we drop the symbol $I=\Delta$):
\begin{equation}
	\begin{aligned}
		\mathrm{PS}^{\infty}_{I,u}(\underline{x})&:=\Big(\ind^{\bL_I(L)}_{\op_{S^u_0}(L)\cap \bL_I(L) }\pi_{\sm}(\underline{x}^u)\eta_{S^u_0}\Big)^{\infty}
	\end{aligned}
\end{equation}
which is irreducible by the generic assumption and is independent on the choice of $u\in \sW_s$ if $I=\Delta$.\;

For $i\in \Delta$,\;put $\widehat{i}=\Delta\backslash\{i\}$.\;For $u\in \sW_s$,\;$i\in \Delta\backslash S_0^u$ and $\sigma\in \Sigma_L$,\;put (see \cite[The main theorem]{orlik2015jordan})
\[C_{i,u,\sigma}:=\cF^{G}_{\op_{S^u_0}}\Big(\overline{L}(-s_{i,\sigma}\cdot\underline{\lambda}),\pi_{\sm}(\underline{x}^u)\eta_{S^u_0}\Big)\cong \cF^{G}_{\op_{\widehat{i}}}\Big(\overline{L}(-s_{i,\sigma}\cdot\underline{\lambda}),\mathrm{PS}^{\infty}_{\widehat{i},u}(\underline{x})\Big),\]
which is an irreducible admissible locally $\sigma$-analytic representation of $G$.\;


\begin{lem}For $u,u'\in \sW_s$,\;$1\leq j,j'\leq s$,\;and $\sigma,\sigma'\in \Sigma_L$,\;we have $C_{t^u_j,u,\sigma}\cong C_{t^{u'}_{j'},u',\sigma'}$ if and only if $t^u_j=t^{u'}_{j'}$,\;$\sigma=\sigma'$,\;$j=j'$,\;and $\{u^{-1}(1),\cdots,u^{-1}(j)\}=\{(u')^{-1}(1),\cdots,(u')^{-1}(j')\}$.\;
\end{lem}


\subsection{Preliminary on extension groups}


The following discussion generalizes the results in  \cite[Section 2.1]{BDcritical25} to potentially crystalline case (or cuspidal case).\;In the sequel,\;write
\[\pi_{\lalg}(\underline{x},\bh)=\mathrm{PS}^{\infty}_{1}(\underline{x})\otimes_E L(\lambda)\cong \mathrm{PS}^{\infty}_{u}(\underline{x})\otimes_E L(\lambda).\]
for all $u\in \sW_s$.\;Similar to \cite[Lemma 2.1.6]{BDcritical25},\;we have
\begin{lem}\label{parameterline} Let $\sigma\in \Sigma_L$.\;We have isomorphisms of finite-dimensional $E$-vector spaces:
	\begin{equation}
		\begin{aligned}
			&\homo_{\sm}(\bZ_{S^{u}_0}(L),E)\oplus_{\homo_{\sm}(L^{\times},E)}\homo(L^{\times},E)\xrightarrow{\sim} \ext^1_{G}\big(\pi_{\lalg}(\underline{x},\bh),\pi_{\lalg}(\underline{x},\bh)\big),\\
\text{resp.,\;}	&\homo_{\sm}(\bZ_{S^{u}_0}(L),E)\oplus_{\homo_{\sm}(L^{\times},E)}\homo_{\sigma}(L^{\times},E)\xrightarrow{\sim} \ext^1_{G,\sigma}\big(\pi_{\lalg}(\underline{x},\bh),\pi_{\lalg}(\underline{x},\bh)\big).\;
		\end{aligned}
	\end{equation}
which have $E$-dimension $s+d_L$ (resp.,\;$s+1$).\;
\end{lem}
\begin{proof}Write $\pi_{\lalg}:=\pi_{\lalg}(\underline{x},\bh)$ for simplicity.\;We prove that $\dim_E\ext^1_{G}\left(\pi_{\lalg},\pi_{\lalg}\right)$ is $s+d_L$ and define the maps.\;We define a canonical injection:
	\begin{equation}\label{lalginj1}
		\begin{aligned}
			\homo_{\sm}(\bZ_{S^u_0}(L),E)&\rightarrow\ext^1_{G}\big(\pi_{\lalg},\pi_{\lalg}\big),\;\\
			\psi&\mapsto \Big(\ind^G_{\op_{S_0^u}(L)}\pi_{\sm}(\underline{x}^u)\eta_{S_0^u}\otimes_E\big(1+\big(\psi\circ\mathrm{det}_{\bL_{S_0^u}}\big)\epsilon\big)\Big)^{\mathrm{sm}}\otimes_EL(\lambda).
		\end{aligned}
	\end{equation}
On the other hand,\;we define a canonical injection
\[\homo(L^{\times},E)\hookrightarrow \ext^1_{G}\left(\pi_{\lalg},\pi_{\lalg}\right),\;\psi\mapsto \pi_{\lalg}\otimes_E\big(1+\big(\psi\circ\mathrm{det}_{\GLN_n}\big)\epsilon\big).\]
These two injections are coincides on $\homo_{\mathrm{sm}}(L^{\times},E)$.\;By the same argument as in \cite[Lemma 3.16]{HigherLinvariantsGL3(Qp)},\;we have
\[\dim_E\ext^1_{G}\left(\pi_{\lalg},\pi_{\lalg}\right)=s+d_L,\;\dim_E\ext^1_{G,\sigma}\left(\pi_{\lalg},\pi_{\lalg}\right)=s+1.\]
Then we get the first isomorphism in the lemma,\;and its restriction on $\homo_{\sigma}(L^{\times},E)$ gives the second isomorphism.\;
\end{proof}

\begin{lem}\label{parameterlineextgp}
	Let $j\in \Delta_s$,\;$u\in \sW_{s}^{\widehat{j},\emptyset}$ and $\sigma\in \Sigma_L$.\;We have 
	\begin{equation}
		\begin{aligned}
			\dim_E\ext^1_G\left(C_{t^u_j,u,\sigma},\pi_{\lalg}(\underline{x},\bh)\right)=1,\;
			\dim_E\ext^1_G\left(\pi_{\lalg}(\underline{x},\bh),C_{t^u_j,u,\sigma}\right)=1.
		\end{aligned}
	\end{equation}
\end{lem}
\begin{proof}The assertion follows from \cite[Proposition 5.1.14]{BQ24}.\;
\end{proof}

Let $\sigma\in \Sigma_L$,\;$u\in \sW_s$ and  $i=t^u_j\in \Delta\backslash S_0^u$ for some $j\in \Delta_s$,\;fix $V_{i,u,\sigma}$ a locally $\sigma$-analytic representation of $G$,\;which is isomorphic to a non-split extension of $C_{i,u,\sigma}$ by $\pi_{\lalg}(\underline{x},\bh)$ (see Lemma \ref{parameterlineextgp},\;$V_{i,u,\sigma}$ is unique up to a scalar),\;i.e.,\;
\[V_{i,u,\sigma}:=\big[\pi_{\lalg}(\underline{x},\bh)-C_{i,u,\sigma}\big],\]
and fix an injection $\iota_{i,u,\sigma}:\pi_{\lalg}(\underline{x},\bh)\hookrightarrow V_{i,u,\sigma}$.\;We have   $C_{i,u,\sigma}\cong V_{i,u,\sigma}/\pi_{\lalg}(\underline{x},\bh)$ and  $\bL_{S^u_0}\subseteq \bL_{\widehat{i}}\cong \GLN_i\times\GLN_{n-i}$.\;

\begin{pro}\label{reinterforextgp}There is a canonical isomorphism of $(s+2)$-dimensional $E$-vector spaces associated to $(V_{i,u,\sigma},\iota_{i,u,\sigma})$:
	\begin{equation}\label{isoforfullextgp}
		\begin{aligned}
			\homo_{\sm}(\bZ_{S^u_0}(L),E)\oplus_{\homo_{\sm}(\bL_{\widehat{i}}(L),E)}\homo_{\sigma}(\bL_{\widehat{i}}(L),E)\xrightarrow{\sim}\ext^1_{G,\sigma}\left(\pi_{\lalg}(\underline{x},\bh),V_{i,u,\sigma}\right).
					\end{aligned}
	\end{equation}
and the restriction of (\ref{isoforfullextgp}) on $\homo_{\sigma}(\bL_{\widehat{i}}(L),E)$ only depends on $(V_{i,u,\sigma},\iota_{i,u,\sigma})$.\;Moreover,\;there is a short exact sequence:
\begin{equation}
	\begin{aligned}
		0\rightarrow \ext^1_{G,\sigma}\Big(\pi_{\lalg}(\underline{x},\bh),\pi_{\lalg}(\underline{x},\bh)\Big) &\xrightarrow{\iota_{i,u,\sigma}} \ext^1_{G,\sigma}\Big(\pi_{\lalg}(\underline{x},\bh),V_{i,u,\sigma}\Big) \\&\rightarrow \ext^1_{G,\sigma}\Big(\pi_{\lalg}(\underline{x},\bh),V_{i,u,\sigma}/\pi_{\lalg}(\underline{x},\bh)\Big)\rightarrow 0.
	\end{aligned}
\end{equation}
\end{pro}
\begin{proof}It follows from Lemma \ref{parameterline} and Lemma \ref{parameterlineextgp} and 
	the  exact sequence $0\rightarrow \ext^1_{G,\sigma}\Big(\pi_{\lalg},\pi_{\lalg}\Big) \rightarrow \ext^1_{G,\sigma}\Big(\pi_{\lalg},V_{i,u,\sigma}\Big) \rightarrow \ext^1_{G,\sigma}\Big(\pi_{\lalg},V_{i,u,\sigma}/\pi_{\lalg}\Big)$ that $\dim_E\ext^1_{G,\sigma}\Big(\pi_{\lalg},V_{i,u,\sigma}\Big)\leq s+2$.\;We construct the desired map and show that the last map in such exact sequence is surjective.\;We first define the injection (by Lemma \ref{parameterline}):
\begin{equation}
		\homo_{\sm}(\bZ_{S^u_0}(L),E)\hookrightarrow \ext^1_{G,\sigma}\Big(\pi_{\lalg},\pi_{\lalg}\Big)\hookrightarrow \ext^1_{G,\sigma}\Big(\pi_{\lalg},V_{i,u,\sigma}\Big).\;
\end{equation}
We next construct a canonical injection $\homo_{\sigma}(\bL_{\widehat{i}}(L),E)\hookrightarrow\ext^1_{G,\sigma}\left(\pi_{\lalg}(\underline{x},\bh),V_{i,u,\sigma}\right)$.\;Put
	\begin{equation}
		R_{u}:=\left(\ind^G_{\op_{\widehat{i}}(L)}\mathrm{PS}^{\infty}_{\widehat{i},u}(\underline{x})\otimes L_{\widehat{i}}(\lambda_{\sigma})\right)^{\sigma-\ana}.\;
	\end{equation}
Consider the following injection:
	\begin{equation}
		\begin{aligned}
			\homo_{\sigma}(\bL_{\widehat{i}}(L),E)&\hookrightarrow \ext^1_{G,\sigma}\left(R_{u},R_{u}\right)\\
			\psi&\mapsto \left(\ind^G_{\op_{\widehat{i}}(L)}\mathrm{PS}^{\infty}_{\widehat{i},u}(\underline{x})\otimes L_{\widehat{i}}(\lambda_{\sigma})\otimes(1+(\psi\circ\det\nolimits_{\bL_{\widehat{i}}})\epsilon)\right)^{\sigma-\ana}.
		\end{aligned}
	\end{equation}
	Composed with the pull-back map for the natural map $\pi_{\lalg}\hookrightarrow R_{u}$,\;we get a map $f_u:\homo_{\sigma}(\bL_{\widehat{i}}(L),E)\hookrightarrow \ext^1_{G,\sigma}\big(\pi_{\lalg},R_{u}\big)\rightarrow  \ext^1_{G,\sigma}\big(\pi_{\lalg},V_{i,u,\sigma}\big)
	$ (it is still an injection and we can prove that it only depends on $(V_{i,u,\sigma},\iota_{i,u,\sigma})$.\;The map $f_u$ induces an injection
	\[\homo_{\sm}(\bL_{\widehat{i}}(L),E)\rightarrow \ext^{1}_{G,\sigma}\big(\pi_{\lalg},\pi_{\lalg}\big),\]
	which is compatible with the injection $\homo_{\sm}(\bZ_{S^u_0}(L),E)\rightarrow\ext^1_{G,\sigma}\big(\pi_{\lalg},\pi_{\lalg}\big)$.\;Thus,\;we get the desired map.\;On the other hand,\;by the second isomorphism in Lemma \ref{reinterforextgp},\;we see that the image of $\homo_{\sigma}(\bL_{\widehat{i}}(L),E)$ in $\ext^1_{G,\sigma}\big(\pi_{\lalg},V_{i,u,\sigma}\big)$ is not contained in $\ext^1_{G,\sigma}\left(\pi_{\lalg},\pi_{\lalg}\right)$,\;so that the second map in (\ref{isoforfullextgp}) is an surjection and thus an isomorphism by comparing the dimensions.\;
\end{proof}

\begin{pro}
	We have a perfect paring of $1$-dimensional $E$-vector spaces:
	\begin{equation}\label{dualpairing}
		\ext^1_{G,\sigma}\left(\pi_{\lalg}(\underline{x},\bh),C_{i,u,\sigma}\right)\times\ext^1_{G,\sigma}\left(C_{i,u,\sigma},\pi_{\lalg}(\underline{x},\bh)\right)\xrightarrow{\sim}\homo_{\sigma}(\GLN_{n-i}(\cO_L),E).\;
	\end{equation}
	and a canonical isomorphism of $1$-dimensional $E$-vector space $E\xrightarrow{\sim }\homo_{\sigma}(\GLN_{n-i}(\cO_L),E)$,\;$1\mapsto \log_{p,\sigma}\circ \det$.\;
\end{pro}
\begin{proof}
	Similar to the proof of \cite[Proposition 2.1.10]{BDcritical25},\;we have the following (formal) isomorphism of $1$-dimensional $E$-vector spaces:
	\[\ext^1_{G,\sigma}\left(C_{i,u,\sigma},\pi_{\lalg}\right)\otimes_E\ext^1_{G,\sigma}\left(\pi_{\lalg},C_{i,u,\sigma}\right)^{\vee}\xrightarrow{\sim}\ext^1_{G,\sigma}\left(C_{i,u,\sigma}\otimes\ext^1_{G,\sigma}\left(\pi_{\lalg},C_{i,u,\sigma}\right),\pi_{\lalg}\right),\]
	where the $G$-representation $C_{i,u,\sigma}\otimes\ext^1_{G,\sigma}\left(\pi_{\lalg},C_{i,u,\sigma}\right)$ has the trivial action of $G$ on $\ext^1_{G,\sigma}\left(\pi_{\lalg},C_{i,u,\sigma}\right)$. Any 
	representative $V'_{i,u,\sigma}$ of the image of the canonical vector of the left hand side is isomorphic to a non-split extension of $C_{i,u,\sigma}$ by $\pi_{\lalg}$.\;Then $V'_{i,u,\sigma}$ corresponds to an injection $\iota:\pi_{\lalg}\hookrightarrow V'_{i,u,\sigma}$ and an isomorphism $\kappa:V'_{i,u,\sigma}/\pi_{\lalg}\xrightarrow{\sim}C_{i,u,\sigma}\otimes_E\ext^1_{G,\sigma}\left(\pi_{\lalg},C_{i,u,\sigma}\right)$.\;By the same strategy as in the \cite[(37)]{BDcritical25},\;we have a canonical isomorphism of $1$-dimensional $E$-vector spaces (where the embedding $\homo_{\sigma}(\cO_L^{\times},E)\hookrightarrow \homo_{\sigma}(\bL_{\widehat{i}}(\cO_L^{\times}),E)$ is induced by $\det:\bL_{\widehat{i}}(\cO_L^{\times})\rightarrow \cO_L^{\times}$):
	\[\homo_{\sigma}(\GLN_{n-i}(\cO_L),E)\xrightarrow{\sim}\homo_{\sigma}(\bL_{\widehat{i}}(\cO_L^{\times}),E)/\homo_{\sigma}(\cO_L^{\times},E).\]
	By the proof of Proposition \ref{reinterforextgp} (similar to \cite[(34)]{BDcritical25}),\;the latter space is isomorphic to $\ext^1_{G,\sigma}\left(\pi_{\lalg},V/\pi_{\lalg}\right)$ and thus isomorphic to (via $\kappa$) $\ext^1_{G,\sigma}\left(\pi_{\lalg},C_{i,u,\sigma}\otimes_E\ext^1_{G,\sigma}\left(\pi_{\lalg},C_{i,u,\sigma}\right)\right)\xleftarrow{\sim} \ext^1_{G,\sigma}\left(\pi_{\lalg},C_{i,u,\sigma}\right)\otimes_E\ext^1_{G,\sigma}\left(\pi_{\lalg},C_{i,u,\sigma}\right)$,\;the result follows.\;
\end{proof}
\begin{rmk}\label{rmkforextgroupsigma}Similar to the Step $1$ in the proof of \cite[Proposition 2.2.4]{BDcritical25},\;the surjection $\bL_{\widehat{i}}(L)\twoheadrightarrow \GLN_{n-i}(L)$ gives a canonical injection $\homo_{\sigma}(\GLN_{n-i}(L),E)\hookrightarrow \homo_{\sigma}(\bL_{\widehat{i}}(L),E)$.\;A choice of $\log_p(p)$ define a section $\homo_{\sigma}(\cO_L^{\times},E)\hookrightarrow \homo_{\sigma}(L^{\times},E)$ to the restriction map $\homo_{\sigma}(L^{\times},E)\hookrightarrow \homo_{\sigma}(\cO_L^{\times},E)$,\;and thus a section $\homo_{\sigma}(\GLN_{n-i}(\cO_L^{\times}),E)\hookrightarrow \homo_{\sigma}(\GLN_{n-i}(L),E)$ to the restriction map $\homo_{\sigma}(\GLN_{n-i}(L),E)\hookrightarrow \homo_{\sigma}(\GLN_{n-i}(\cO_L^{\times}),E)$,\;which induces an isomophism:
	\begin{equation*}
		\begin{aligned}
			\Big(\homo_{\sm}(\bL_{\widehat{i}}(L),E)\oplus_{\homo_{\sm}(L^{\times},E)}\homo(L^{\times},E)\Big)\oplus \homo_{\sigma}(\GLN_{n-i}(\cO_L^{\times}),E)\xrightarrow{\sim} \homo_{\sigma}(\bL_{\widehat{i}}(L),E).
		\end{aligned}
	\end{equation*}
	We thus get (similar to \cite[(59)]{BDcritical25}) isomorphsims $\homo_{\sigma}(\GLN_{n-i}(\cO_L^{\times}),E)\xrightarrow{\sim}\ext^1_{G,\sigma}\left(\pi_{\lalg},V_{i,u,\sigma}/\pi_{\lalg}\right)$ and 
	\begin{equation*}
		\begin{aligned}
			\Big(\homo_{\sm}(\bZ_{S^u_0}(L),E)\oplus_{\homo_{\sm}(L^{\times},E)}\homo(L^{\times},E)\Big)\oplus \homo_{\sigma}(\GLN_{n-i}(\cO_L^{\times}),E)\xrightarrow{\sim} \ext^1_{G,\sigma}\left(\pi_{\lalg},V_{i,u,\sigma}\right).
		\end{aligned}
	\end{equation*}
\end{rmk}

\subsubsection{Constructions of locally analytic representations}

For $u\in\sW_s$,\;consider the locally  $\bQ_p$-analytic  principal series:
\begin{equation}\label{locaparabolicind}
	\begin{aligned}
	&\mathrm{PS}_{u}(\underline{x},\bh):=\Big(\ind^G_{\op_{S^u_0}(L)}\pi_{\sm}(\underline{x}^u)\eta_{S^u_0}\otimes_E L_{S_0^u}(\lambda)\Big)^{\bQ_p-\ana}.\\
	\end{aligned}
\end{equation}
Note that the locally $\bQ_p$-algebraic vectors  $\mathrm{PS}^{\lalg}_{u}(\underline{x},\bh)$ in $\mathrm{PS}_{u}(\underline{x},\bh)$ is isomorphic to  $\soc_G\mathrm{PS}_{u}(\underline{x},\bh)=\pi_{\lalg}(\underline{x},\bh)$.\;Let $\mathrm{PS}_{u,1}(\underline{x},\bh)$ be the first two layers of the socle filtrartion of $\mathrm{PS}_{u}(\underline{x},\bh)$.\;By \cite[The main theorem]{orlik2015jordan}),\;
\[\mathrm{PS}_{u,1}(\underline{x},\bh)=\bigoplus_{\pi_{\lalg}(\underline{x},\bh)}^{{j\in \Delta_s,\sigma\in \Sigma_L}}V_{t^u_j,[u]_j,\sigma}.\;\]
Consider the amalgamated sum:
\[\pi^{\flat}_{1}(\underline{x},\bh):=\bigoplus_{\pi_{\lalg}(\underline{x},\bh)}^{\substack{j\in \Delta_s,u\in\sW_{s}^{\widehat{j},\emptyset}\\\sigma\in \Sigma_L}}V_{t^u_j,u,\sigma},\;\]
so that we have a short exact sequence:
\begin{equation}\label{firstwhole}
	0\rightarrow \pi_{\lalg}(\underline{x},\bh) \rightarrow \pi^{\flat}_{1}(\underline{x},\bh)\rightarrow \oplus_{\substack{j\in \Delta_s,u\in\sW_{s}^{\widehat{j},\emptyset}\\\sigma\in \Sigma_L}}C_{t^u_j,u,\sigma}\rightarrow 0.\;
\end{equation}
Moreover,\;$\pi^{\flat}_{1}(\underline{x},\bh)$ is the unique quotient of the amalgamated sum 
\[\bigoplus^{u\in \sW_{s}}_{\pi_{\lalg}(\underline{x},\bh)}\mathrm{PS}_{u,1}(\underline{x},\bh)\]
of socle $\pi_{\lalg}(\underline{x},\bh)$.\;We have injection $\mathrm{PS}_{u,1}(\underline{x},\bh)\hookrightarrow \pi^{\flat}_{1}(\underline{x},\bh)$ of $G$-representations for any $u\in \sW_{s}$.\;

For $u\in \sW_s$,\;we consider the natural map
\begin{equation}
	\begin{aligned}
		\homo\big(\bZ_{S^u_0}(L),E\big)&\rightarrow\ext^1_{G}\big(\mathrm{PS}_{u}(\underline{x},\bh),\mathrm{PS}_{u}(\underline{x},\bh)\big),\;\\
		\psi&\mapsto \Big(\ind^G_{\op_{S^u_0}(L)}\pi_{\sm}(\underline{x}^u)\eta_{S^u_0}\otimes_EL_{S^u_0}(\lambda)\otimes_E(1+(\psi\circ\mathrm{det}_{\bL_{S_0^u}})\epsilon)\Big)^{\bQ_p-\ana}.
	\end{aligned}
\end{equation}
In particular,\;by Schraen's spectral sequence \cite[(4.37)]{schraen2011GL3},\;we have an isomorphism
\begin{equation}
	\begin{aligned}
		\homo\big(\bZ_{S^u_0}(L),E\big)&\xrightarrow{\sim}\ext^1_{G}\big(\pi_{\lalg}(\underline{x},\bh),\mathrm{PS}_{u}(\underline{x},\bh)\big).
	\end{aligned}
\end{equation}
By \cite[Lemma 2.26]{2019DINGSimple},\;this map factors through (induced by the injection $\mathrm{PS}_{u,1}(\underline{x},\bh)\hookrightarrow \mathrm{PS}_{u}(\underline{x},\bh)$)
\[\ext^1_{G}\big(\pi_{\lalg}(\underline{x},\bh),\mathrm{PS}_{u,1}(\underline{x},\bh)\big)\rightarrow \ext^1_{G}\big(\pi_{\lalg}(\underline{x},\bh),\mathrm{PS}_{u}(\underline{x},\bh)\big),\]
then we obtain a map $\homo\big(\bZ_{S^u_0}(L),E\big)\rightarrow\ext^1_{G}\big( \pi_{\lalg}(\underline{x},\bh),\mathrm{PS}_{u,1}(\underline{x},\bh)\big)$.\;Composing with the push-forward map via the injection $\mathrm{PS}_{u,1}(\underline{x},\bh)\hookrightarrow \pi^{\flat}_1(\underline{x},\bh)$,\;we actually obtain a map:
\[\zeta_{u}:\homo\big(\bZ_{S^u_0}(L),E\big)\rightarrow\ext^1_{G}\big(\pi_{\lalg}(\underline{x},\bh),\pi^{\flat}_1(\underline{x},\bh)\big).\;\]

\begin{pro}\label{extpi1repn}
	\begin{itemize}
		\item[(1)] We have  $\dim_E\ext^1_{G}\big(\pi_{\lalg}(\underline{x},\bh),\mathrm{PS}_{u,1}(\underline{x},\bh)\big)=s(1+d_L)$ and an exact sequence:
		\begin{equation}\label{extenforPSw}
			\begin{aligned}
				0\rightarrow \ext^1_{G}\big(\pi_{\lalg}(\underline{x},\bh),\pi_{\lalg}(\underline{x},\bh)\big)\rightarrow &\;\ext^1_{G}\big(\pi_{\lalg}(\underline{x},\bh),\mathrm{PS}_{u,1}(\underline{x},\bh)\big)\\
				&\rightarrow \oplus_{\substack{j\in \Delta_s\\\sigma\in \Sigma_L}}  \ext^1_{G}\Big(\pi_{\lalg}(\underline{x},\bh),C_{t_j^u,[u]_j,\sigma}\Big)\rightarrow 0 .
			\end{aligned}
		\end{equation}
		Thus,\;we get an ismorphism $\zeta_u:\homo\big(\bZ_{S^u_0}(L),E\big)\xrightarrow{\sim}\ext^1_{u}\big(\pi_{\lalg}(\underline{x},\bh),\mathrm{PS}_{u,1}(\underline{x},\bh)\big)$.\;
		\item[(2)]We have $\ext^1_{G}\left(\pi_{\lalg}(\underline{x},\bh),\pi^{\flat}_1(\underline{x},\bh)\right)=s+(2^s-1)d_L$ and an short exact sequence:
		\begin{equation}
			\begin{aligned}
				0\rightarrow \ext^1_{G}\big(\pi_{\lalg}(\underline{x},\bh),\pi_{\lalg}(\underline{x},\bh)\big)\rightarrow &\;\ext^1_{G}\big(\pi_{\lalg}(\underline{x},\bh),\pi^{\flat}_1(\underline{x},\bh)\big)\\
				&\rightarrow \oplus_{\substack{j\in \Delta_s,u\in \sW_{s}^{\widehat{j},\emptyset}\\\sigma\in \Sigma_L}}  \ext^1_{G}\left(\pi_{\lalg}(\underline{x},\bh),C_{t^u_j,u,\sigma}\right)\rightarrow 0 .
			\end{aligned}
		\end{equation}
	\end{itemize}
\end{pro}
\begin{proof}
	For $(1)$,\;by Proposition \ref{reinterforextgp},\;the second last map in (\ref{extenforPSw}) is surjective.\;On the other hand,\;
	the natural morphism (by Remark \ref{rmkforextgroupsigma})
	\[\homo\big(\bZ_{S^u_0}(L),E\big)\rightarrow \Big(\homo_{\sm}(\bZ_{S^u_0}(L),E)\oplus_{\homo_{\sm}(L^{\times},E)}\homo(L^{\times},E)\Big)\oplus \Big(\oplus_{\sigma,i\in \Delta\backslash S_0^u}\homo_{\sigma}(\GLN_{n-i}(\cO_L^{\times}),E)\Big)\]
	is an isomorphism,\;therefore $\zeta_u$ is an ismorphism.\;When $u$ varying,\;we prove that the second last map in $(2)$ is also surjective since it is covered by the last term in (\ref{extenforPSw}) with $u$ varying.\;The dimension of $\ext^1_{G}\left(\pi_{\lalg}(\underline{x},\bh),\pi^{\flat}_1(\underline{x},\bh)\right)$ is equal to $s+d_L+d_L\sum_{j\in \Delta_s}\#\sW_{s}^{\widehat{j},\emptyset}=s+(2^s-1)d_L$.\;
\end{proof}


\begin{rmk}\label{rmkforsigmaextANA}
Fix $\sigma\in\Sigma_L$,\;replacing the locally $\bQ_p$-analytic parabolic induction in (\ref{locaparabolicind}) by locally $\sigma$-analytic parabolic induction,\;we obtain a $\sigma$-analytic version of all above discussion.\;We get a locally $\sigma$-analytic representation  $\pi^{\flat}_{1,\sigma}(\underline{x},\bh)$ which lies in the following exact sequence:
\begin{equation}\label{firstwholesigma}
	0\rightarrow \pi_{\lalg}(\underline{x},\bh) \rightarrow \pi^{\flat}_{1,\sigma}(\underline{x},\bh)\rightarrow \oplus_{\substack{j\in \Delta_s\\u\in \sW_{s}^{\widehat{j},\emptyset}}}C_{t_j^u,u,\sigma}\rightarrow 0.\;
\end{equation}
Note that $\pi^{\flat}_{1,\sigma}(\underline{x},\bh)=\bigoplus^{\pi_{\lalg}(\underline{x},\bh)}_{j\in \Delta_s,u\in\sW_{s}^{\widehat{j},\emptyset}}V_{t^u_j,u,\sigma}$ and $\pi^{\flat}_{1}(\underline{x},\bh)=\bigoplus^{\pi_{\lalg}(\underline{x},\bh)}_{\sigma\in\Sigma_L}\pi^{\flat}_{1,\sigma}(\underline{x},\bh)$.\;We have  a short exact sequence:
\begin{equation}
	\begin{aligned}
		0\rightarrow \ext^1_{G,\sigma}\big(\pi_{\lalg}(\underline{x},\bh),\pi_{\lalg}(\underline{x},\bh)\big)\rightarrow &\;\ext^1_{G,\sigma}\big(\pi_{\lalg}(\underline{x},\bh),\pi^{\flat}_{1,\sigma}(\underline{x},\bh)\big)\\
		&\rightarrow \oplus_{\substack{j\in \Delta_s\\u\in \sW_{s}^{\widehat{j},\emptyset}}}  \ext^1_{G,\sigma}\left(\pi_{\lalg}(\underline{x},\bh),C_{t^u_j,u,\sigma}\right)\rightarrow 0 .
	\end{aligned}
\end{equation}
Moreover,\;$\dim_E\ext^1_{G,\sigma}\left(\pi_{\lalg}(\underline{x},\bh),\pi^{\flat}_{1,\sigma}(\underline{x},\bh)\right)=s+(2^s-1)$.\;In a similar way,\;for $u\in\sW_s$,\;we get a map:
\[\zeta_{u,\sigma}:\homo_{\sigma}(\bZ_{S^u_0}(L),E)\rightarrow\ext^1_{G,\sigma}\big(\pi_{\lalg}(\underline{x},\bh),\pi^{\flat}_{1,\sigma}(\underline{x},\bh)\big).\;\]
\end{rmk}

\subsection{Main results}\label{reconstD}

This section follows the route in the recent work \cite[Section 2.4]{BDcritical25} of Breuil-Ding.\;Fix $\sigma\in \Sigma_L$,\;we will construct a map which only depends  on a choice of $\log_p(p)\in E$  (recall the locally $\sigma$-analytic representation  $\pi^{\flat}_{1,\sigma}(\underline{x},\bh)$  in Remark \ref{rmkforsigmaextANA}):
	\begin{equation}\label{constructionvarphi}
	\begin{aligned}
		t^{\flat}_{D_{\sigma}}:\ext^1_{G,\sigma}\left(\pi_{\lalg}(\underline{x},\bh),\pi^{\flat}_{1,\sigma}(\underline{x},\bh)\right)\rightarrow \ext^{1,\flat}_{\gr}({\Dpik},{\Dpik})\oplus\homo_{\fil}(D_{\sigma},D_{\sigma}).\;
	\end{aligned}
\end{equation}
(see Section \ref{reinterfor33} for the spaces $\ext^{1,\flat}_{\gr}({\Dpik},{\Dpik})$ and $\homo_{\fil}(D_{\sigma},D_{\sigma})$) and 
\begin{itemize}
	\item[(a)] describe the image of $t^{\flat}_{D_{\sigma}}$.\;
	\item[(b)] describe what information of  Hodge filtration  is determined by the kernel of  $t^{\flat}_{D_{\sigma}}$.\;
\end{itemize}
In the sequel,\;put
\[\Delta'=\{t_j^u:u\in \sW_s,j\in \Delta_s\}=\cup_{u\in \sW_s}\Delta\backslash S_0^u\subseteq\Delta.\]
For $i\in\Delta'$,\;put $\cI_{i}^{\flat}:=\{u\in \sW_s:\;i=t_{l_i(u)}^u \text{\;for some\;} 1\leq l_i(u)\leq s-1\}=\{u\in \sW_s:i\in \Delta\backslash S_0^u\}$.\;For $u\in \cI_{i}^{\flat}$,\;let 
\[I_i(u)=\{u^{-1}(1),\cdots,u^{-1}(l_i(u))\},\]
and $I_i(u)^c=\{1,\cdots,s\}\backslash I_i(u)$.\;In particular,\;if $S_0=\emptyset$,\;then $\Delta'=\Delta$ and $\cI_{i}^{\flat}\cong \{I\subseteq\{1,\cdots,n\}:\;|I|=i\}$.\;


Recall that we have fixed a basis $\{e_{1,j,\sigma}\}_{1\leq j\leq r_1},\cdots,\{e_{s,j,\sigma}\}_{1\leq j\leq r_s}$ of $D_{\sigma}^1,\cdots,D_{\sigma}^s$ respectively,\;and we relabel them with $e_{1,\sigma},\cdots,e_{n,\sigma}$  by using lexicographical ordering (a basis of $D_{\sigma}=\oplus_{1\leq i\leq s}D_{\sigma}^i$).\;For $1\leq i\leq s$,\;let 
\[I_i=\{q:\;1\leq q\leq n,\;e_{q,\sigma}=e_{i,j,\sigma} \text{\;for some\;} 1\leq j\leq r_i\}=\{t_{i-1}^1+1,\cdots,t_i^1\}.\]
For $J\subseteq \{1,\cdots,n\}$,\;put
\[e_J:=\wedge_{j\in J}e_{j,\sigma}\in \bigwedge_E\nolimits^{\!|J|}\!D_{\sigma}.\]
For $1\leq i\leq s$,\;put 
\[\be_i:=\wedge_{j=1}^{r_i}e_{i,j,\sigma}=e_{I_i}\in\bigwedge\nolimits_E^{\!r_i}\!D^i_{\sigma}.\;\]
For $I\subseteq \{1,\cdots,s\}$,\;put $D_{\sigma}[I]:=\oplus_{i\in I}D_{\sigma}^i$ and 
\[\be_I:=\wedge_{i\in I}\be_i\in \bigwedge_E\nolimits^{\!\sum_{i\in I}r_i}\!D_{\sigma}[I]\subset \bigwedge_E\nolimits^{\!\sum_{i\in I}r_i}\!D_{\sigma}.\;\]

\subsubsection{Preliminaries}

%

For each $u\in \cI_{i}^{\flat}$,\;we fix an
isomorphism of $1$-dimensional $E$-vector spaces (as in \cite[(43)]{BDcritical25}):
\begin{equation}
	\epsilon_{u,i}:\ext^1_{G,\sigma}\big(C_{i,u,\sigma},\pi_{\lalg}(\underline{x},\bh)\big)\xrightarrow{\sim}E\be_{I_i(u)^c}\in \bigwedge\nolimits_E^{\!n-i}\!D_{\sigma}.
\end{equation}
Then we get for each $i\in\Delta'$ a map:
\begin{equation}\label{dfnforeplisonistar}
	\epsilon_{i}:=\bigoplus_{u\in \cI_{i}^{\flat}}\epsilon_{u,i}:\ext^1_{G,\sigma}\Big(\bigoplus_{u\in \cI_{i}^{\flat}}C_{i,u,\sigma},\pi_{\lalg}(\underline{x},\bh)\Big)\rightarrow\bigwedge\nolimits_E^{\!n-i}\!D_{\sigma}.\;
\end{equation}
Let $\big(\bigwedge_{E}^{\!n-i}\!D_{\sigma}\big)^{\flat}$ be the image of $\epsilon_{i}$.\;Let
\[\big(\bigwedge\nolimits_{E}^{\!n-i}\!D_{\sigma}\big)^{\flat,c}=E\big\langle e_J:\;J\subseteq \{1,\cdots,n\},|J|=n-i,e_J\neq \be_{I_i(u)^c},\;\forall\;u\in \cI_{i}^{\flat}\big\rangle.\]
Then $\bigwedge_E^{n-i}D_{\sigma}=\big(\bigwedge\nolimits_{E}^{\!n-i}\!D_{\sigma}\big)^{\flat}\oplus\big(\bigwedge\nolimits_{E}^{\!n-i}\!D_{\sigma}\big)^{\flat,c}$ and  $\dim_E\big(\bigwedge_{E}^{\!n-i}\!D_{\sigma}\big)^{\flat}=|\cI_{i}^{\flat}|$.\;In the sequel,\;we identity $\big(\bigwedge_{E}^{\!n-i}\!D\big)^{\flat}$ with the quotient $\bigwedge_{E}^{\!n-i}\!D/\big(\bigwedge_{E}^{\!n-i}\!D\big)^{\flat,c}$.\;

For $1\leq i\leq n-1$,\;we define the following $E$-line of $\bigwedge_E^{n-i}D_{\sigma}$:
\begin{equation}
	\begin{aligned}
		\fil_i^{\max}(D_{\sigma})&:= \fil_H^{-\bh_{n,\sigma}}(D_{\sigma})\wedge \fil_H^{-\bh_{n-1,\sigma}}(D_{\sigma})\wedge\cdots\wedge \fil_H^{-\bh_{i+1,\sigma}}(D_{\sigma})\\
		&\xrightarrow{\sim}\bigwedge\nolimits_E^{\!n-i}\!\fil_H^{-\bh_{i+1,\sigma}}(D_{\sigma})\subseteq\bigwedge\nolimits_E^{\!n-i}\!D_{\sigma} .\;
	\end{aligned}
\end{equation}
For each $u\in \cI_{i}^{\flat} $,\;the non-critical assumption implies that the coefficient of $\be_{I_i(u)^c}$  in 	$\fil_i^{\max}(D_{\sigma})$ is non-zero.\;

The composition of the surjections in \cite[(68) \& (69)]{BDcritical25} and the inverse of the isomorphism \cite[(69)]{BDcritical25} gives:
\begin{equation}\label{mapforgisigma}
	\begin{aligned}
		g_{i,\sigma}:\homo_E\Big(\bigwedge\nolimits_E^{\!n-i}\!&D_{\sigma},\fil_i^{\max}(D_{\sigma})\Big)=\homo_E\Big(\bigwedge\nolimits_E^{\!n-i}\!D_{\sigma},\bigwedge\nolimits_E^{\!n-i}\!\fil_H^{-\bh_{i+1,\sigma}}(D_{\sigma})\Big)\\&\twoheadrightarrow \homo_E\Big(\bigwedge\nolimits_E^{\!n-i-1}\!\fil_H^{-\bh_{i+1,\sigma}}(D_{\sigma})\wedge D_{\sigma} ,\bigwedge\nolimits_E^{\!n-i}\!\fil_H^{-\bh_{i+1,\sigma}}(D_{\sigma})\Big)\\
		&\xrightarrow{\sim}\left\{f\in\homo_E\big(D_{\sigma},\fil_H^{-\bh_{i+1,\sigma}}(D_{\sigma})\big):f|_{\fil_H^{-\bh_{i+1,\sigma}}(D_{\sigma})} \text{scalar}\right\}\hookrightarrow \homo_{\fil}\big(D_{\sigma},D_{\sigma}\big),
	\end{aligned}
\end{equation}
the third inclusion is given by \cite[Lemma 2.2.5 (ii)]{BDcritical25}.\;By \cite[Lemma 2.2.5 (ii)]{BDcritical25},\;we have a surjection
\[\bigoplus_{i=1}^n\left\{f\in\homo_E\big(D_{\sigma},\fil_H^{-\bh_{i+1,\sigma}}(D_{\sigma})\big):f|_{\fil_H^{-\bh_{i+1,\sigma}}(D_{\sigma})} \text{scalar}\right\}\twoheadrightarrow \homo_{\fil}(D_{\sigma},D_{\sigma}).\]
Let $\Delta':=\{i_1<i_2<\cdots<i_{|\Delta'|}\}$ and $S:=\{i_{l_1}<i_{l_2}<\cdots<i_{l_{|S|}}\}\subseteq \Delta'$ for $1\leq l_1<\cdots<l_{|S|}\leq |\Delta'|$.\;The partial filtration $(\fil_H^{\bh_{j+1},\sigma}(D_{\sigma}),j\in S)$  on $D$ induces a natural decreasing filtration on $\bigwedge\nolimits_{E}^{\!n-i}\!D$.\;In  \cite[(97)-(99)]{BDcritical25},\;the author consider the \textit{one but last filtration} of this induced filtration:
\[\fil_{S,i}^{2^{\mathrm{nd}}-\max}(D_{\sigma})\]
Note that $\fil_i^{\max}(D_{\sigma})$ is just the last step of this induced filtration by \cite[Remark 2.3.6,\;(ii)]{BDcritical25}.\;See \cite[(97)-(99)]{BDcritical25} for the following definitions of $\fil_{S,i}^{2^{\mathrm{nd}}-\max}(D_{\sigma})$:
\begin{equation}
	\begin{aligned}
		&\fil_{S,i_{l_{|S|}}}^{2^{\mathrm{nd}}-\max}(D_{\sigma})=\Big(\bigwedge\nolimits_E^{\!n-i_{l_{|S|}}-1}\!\fil_H^{-\bh_{i_{l_{|S|}}+1,\sigma}}(D_{\sigma})\Big)\wedge \fil_H^{-\bh_{i_{l_{|S|-1}}+1,\sigma}}(D_{\sigma})\\
		&\fil_{S,i_{l_{j}}}^{2^{\mathrm{nd}}-\max}(D_{\sigma})=\Big(\bigwedge\nolimits_E^{\!n-i_{l_{j+1}}}\!\fil_H^{-\bh_{i_{l_{j+1}}+1,\sigma}}(D_{\sigma})\Big)\wedge\Big(\bigwedge\nolimits_E^{\!i_{l_{j+1}}-i_{l_{j}}-1}\!\fil_H^{-\bh_{i_{l_{j}}+1,\sigma}}(D_{\sigma})\Big)\wedge \fil_H^{-\bh_{i_{l_{j-1}}+1,\sigma}}(D_{\sigma})
	\end{aligned}
\end{equation}
 for $j\in\{1,\cdots,|S|-1\}$,\;In particular,\;if $S=\{i\}$,\;then $\fil_{\{i\},i}^{2^{\mathrm{nd}}-\max}(D_{\sigma})=\bigwedge\nolimits_E^{\!n-i-1}\!\fil_H^{-\bh_{i+1,\sigma}}(D_{\sigma})\wedge D_{\sigma}$ (see the second row of (\ref{mapforgisigma})).\;

For $i\in \Delta'$,\;define 
\begin{equation}
	\begin{aligned}
		D_{\sigma,(i)}^{\flat}:=\oplus_{u\in \cI_{i}^{\flat}}\oplus_{l_i(u)+1\leq q\leq s}D_{\sigma}^{u^{-1}(q)}=\oplus_{u\in \cI_{i}^{\flat}}\oplus_{j\in I_i(u)^c}D_{\sigma}^{j}=D_{\sigma}[\cS_i]
	\end{aligned}
\end{equation}
with $\cS_i:=\cup_{u\in \cI_{i}^{\flat}}I_i(u)^c\subseteq\{1,\cdots,s\}$.\;Put 
\[D_{\sigma}^{\flat,(i)}:=D_{\sigma}[\{1,\cdots,s\}\backslash\cS_i]=\oplus_{q\in \{1,\cdots,s\}\backslash\cS_i}D_{\sigma}^q.\]
Thus $D_{\sigma}=D_{\sigma}^{\flat,(i)}\oplus D^{\flat}_{\sigma,(i)}$.\;By definition,\;we have inclusions
\[\big(\bigwedge\nolimits_{E}^{\!n-i}\!D_{\sigma}\big)^{\flat}\hookrightarrow\bigwedge\nolimits_E^{\!n-i}\!D_{\sigma,(i)}^{\flat}\hookrightarrow \bigwedge\nolimits_E^{\!n-i}\!D_{\sigma}.\]


In the sequel,\;let $\fil_{S,i}^{2^{\mathrm{nd}}-\max}(D)^{\flat}$  (resp.,\;$\fil_i^{\max}(D_{\sigma})^{\flat}$)
be the image of $\fil_{S,i}^{2^{\mathrm{nd}}-\max}(D)$ (resp.,\;$\fil_i^{\max}(D_{\sigma})$) in $ \bigwedge_E^{n-i}D_{\sigma}\twoheadrightarrow \bigwedge_E^{n-i}D_{\sigma}/\big(\bigwedge_{E}^{\!n-i}\!D_{\sigma}\big)^{\flat,c}\cong \big(\bigwedge_{E}^{\!n-i}\!D_{\sigma}\big)^{\flat}$.\;

\begin{lem}\label{lemforwedge} 
	The image under $g_{i,\sigma}$ of $\homo_E\Big(\big(\bigwedge_{E}^{\!n-i}\!D_{\sigma}\big)^{\flat},\bigwedge\nolimits_E^{\!n-i}\!\fil_H^{-\bh_{i+1,\sigma}}(D_{\sigma})\Big)$ is 
	\begin{equation}\label{imageofgflat}
		\begin{aligned}
			&\;\left\{f\in\homo_E\big(D_{\sigma}/D_{\sigma}^{\flat,(i)},\fil_H^{-\bh_{i+1}}(D_{\sigma})\big):f|_{\fil_H^{-\bh_{i+1}}(D_{\sigma})} \text{\;scalar\;}\right\} \Big(\hookrightarrow \homo_{\fil}\big(D_{\sigma,(i)}^{\flat},D_{\sigma}\big)\Big).
		\end{aligned}
	\end{equation}
Moreover,\;the kernel under $g_{i,\sigma}$ of $\homo_E\Big(\big(\bigwedge_{E}^{\!n-i}\!D_{\sigma}\big)^{\flat},\bigwedge\nolimits_E^{\!n-i}\!\fil_H^{-\bh_{i+1,\sigma}}(D_{\sigma})\Big)$ is 
\begin{equation}
	\begin{aligned}
		&\;\homo_E\Big(\big(\bigwedge\nolimits_{E}^{\!n-i}\!D_{\sigma}\big)^{\flat}/\fil_{\{i\},i}^{2^{\mathrm{nd}}-\max}(D_{\sigma})^{\flat},\fil_i^{\max}(D_{\sigma})^{\flat}\Big)
	\end{aligned}
\end{equation}
\end{lem}
\begin{proof}For each $e_t\in D_{\sigma}^{\flat,(i)}$,\;$x\wedge e_t\notin \big(\bigwedge\nolimits_{E}^{\!n-i}\!D_{\sigma}\big)^{\flat}$ for any $x\in \bigwedge\nolimits_E^{\!n-i-1}\!\fil_H^{-\bh_{i+1}}(D_{\sigma})$,\;then $f$ vanishes on $D_{\sigma}^{\flat,(i)}$ by definition.\;When $f$ belongs to the space in (\ref{imageofgflat}),\;there exists $g\in \homo_E\big(\bigwedge\nolimits_E^{\!n-i}\!D_{\sigma},\bigwedge_E^{n-i}\fil_H^{-\bh_{i+1}}(D_{\sigma})\big)$ such that $g_{i,\sigma}(g)=f$,\;then   $g_{i,\sigma}\big(g|_{\big(\bigwedge\nolimits_{E}^{\!n-i}\!D_{\sigma}\big)^{\flat}}\big)=f$.\;This proves the surjectivity.\;By \cite[Remark 2.3.6,\;(ii)]{BDcritical25},\;we have
\[\ker(g_{i,\sigma})= \homo_E\Big(\big(\bigwedge\nolimits_{E}^{\!n-i}\!D_{\sigma}\big)/\fil_{\{i\},i}^{2^{\mathrm{nd}}-\max}(D_{\sigma}),\fil_i^{\max}(D_{\sigma})\Big).\]
This deduces the second assertion.\;
\end{proof}	

\begin{lem}\label{hodgegapforquofil}
	Consider the induced subspace filtration $\fil_H^{\bullet}(D_{\sigma}^{\flat,(i)})$ of $\fil_H^{\bullet}(D_{\sigma})$ on $D_{\sigma}^{\flat,(i)}$ and the quotient filtration $\fil_{\mathrm{quot}}^{\bullet}(D^{\flat}_{\sigma,(i)})$ on $D^{\flat}_{\sigma,(i)}\cong D_{\sigma}/D_{\sigma}^{\flat,(i)}$.\;Let  $\bd_i:=\dim_ED_{\sigma}^{\flat,(i)}\leq i$.\;Then
	\begin{equation}
	\begin{aligned}
		\gr_H^{-\bh_{j,\sigma}}(D_{\sigma}^{\flat,(i)})\neq 0,\quad \big(\text{resp.,\;}\gr_{\mathrm{quot}}^{-\bh_{j,\sigma}}(D^{\flat}_{\sigma,(i)})\big)
	\end{aligned}
	\end{equation}
	if and only if $1\leq j\leq \bd_i$ (resp.,\;$\bd_i+1\leq j\leq n$).\;Moreover,\;we have
		\begin{equation}\label{comparefilwithflat}
		\begin{aligned}
			\fil_H^{-\bh_{j,\sigma}}(D_{\sigma}^{\flat,(i)})/\cong \fil_H^{-\bh_{j,\sigma}}(D_{\sigma})/\fil_H^{-\bh_{\bd_i+1,\sigma}}(D_{\sigma})\quad (\text{resp.,\;}\fil_{\mathrm{quot}}^{-\bh_{j,\sigma}}(D^{\flat}_{\sigma,(i)})\cong\fil_H^{-\bh_{j,\sigma}}(D_{\sigma})).
		\end{aligned}
	\end{equation}
\end{lem}
\begin{proof}
	These follow from the non-critical assumption in an adapted basis of $D_{\sigma}=D^{\flat}_{\sigma,(i)}\oplus D_{\sigma}^{\flat,(i)} $.\;
\end{proof}

\begin{lem}\label{filforhomospace}
For $i\in \Delta'$,\;we have 
	\begin{equation*}
	\begin{aligned}
&\;\left\{f\in\homo_E\big(D_{\sigma}/D_{\sigma}^{\flat,(i)},\fil_H^{-\bh_{i+1,\sigma}}(D_{\sigma})\big):f|_{\fil_H^{-\bh_{i+1,\sigma}}(D_{\sigma})} \text{\;scalar\;}\right\}\\\cong&\; E\oplus \homo_E\Big(\fil_H^{-\bh_{\bd_i+1,\sigma}}(D_{\sigma})/\fil_H^{-\bh_{i+1,\sigma}}(D_{\sigma}),\fil_H^{-\bh_{i+1,\sigma}}(D_{\sigma})\Big)\\\cong&\;E\oplus\bigoplus_{\substack{\bd_i+1\leq a\leq i\\i+1\leq b\leq n}}\homo_E(\gr_H^{-\bh_{a,\sigma}}(D_{\sigma}),\gr_H^{-\bh_{b,\sigma}}(D_{\sigma})).\;
	\end{aligned}
\end{equation*}
\end{lem}
\begin{proof}We have $\fil_H^{-\bh_{i+1,\sigma}}(D_{\sigma})=\oplus_{i+1\leq b\leq n}\gr_H^{-\bh_{b,\sigma}}(D_{\sigma})$ and 
\begin{equation*}
\begin{aligned}
	&\;D_{\sigma}/(D_{\sigma}^{\flat,(i)}+\fil_H^{-\bh_{i+1,\sigma}}(D_{\sigma}))\\
	=&\;\bigoplus_{1\leq a\leq i}\frac{\fil_H^{-\bh_{a,\sigma}}(D_{\sigma})+D_{\sigma}^{\flat,(i)}}{\fil_H^{-\bh_{a+1,\sigma}}(D_{\sigma})+D_{\sigma}^{\flat,(i)}}=\bigoplus_{1\leq a\leq i}\frac{\fil_H^{-\bh_{a,\sigma}}(D_{\sigma})}{\fil_H^{-\bh_{a+1,\sigma}}(D_{\sigma})+D_{\sigma}^{\flat,(i)}\cap \fil_H^{-\bh_{a,\sigma}}(D_{\sigma})},\\
	=&\;\bigoplus_{\bd_i+1\leq a\leq i}\gr_H^{-\bh_{a,\sigma}}(D_{\sigma})=\fil_H^{-\bh_{\bd_i+1,\sigma}}(D_{\sigma})/\fil_H^{-\bh_{i+1,\sigma}}(D_{\sigma}).\;
\end{aligned}
\end{equation*}
where the third identity follows from ${\fil_H^{-\bh_{a,\sigma}}(D_{\sigma})=\fil_H^{-\bh_{a+1,\sigma}}(D_{\sigma})+D_{\sigma}^{\flat,(i)}\cap \fil_H^{-\bh_{a,\sigma}}}(D_{\sigma})$ when $1\leq a\leq \bd_i$ (by Lemma \ref{hodgegapforquofil} and the non-critical assumption).\;
\end{proof}
 For any $1\leq a,b\leq n$,\;put
 \begin{equation}
 	\begin{aligned}
 		&\cS_{a,b}(S):=\{i\in S:\bd_i+1\leq a\leq i\leq b-1\},\\
 	\end{aligned}
 \end{equation}
 Let
 \begin{equation}
	\begin{aligned}
		\cJ_i:=&\;\{(a,b):\;1\leq a,b\leq n,\;\cS_{a,b}(S)=\{i\}\},\\
		\cJ^1_i:=&\;\{(a,b):\;1\leq a,b\leq n,\;\cS_{a,b}(S)=\{i\}=S\cap \{a,a+1,\cdots,b-1\}\}.\;
	\end{aligned}
\end{equation}

\begin{lem}\label{lemforcardone} Let 
\begin{equation*}
	\cJ^{2}_i:=\cJ_i\backslash \cJ^1_i=\{(a,b):\;1\leq a,b\leq n,\;\cS_{a,b}(S)=\{i\},\;|S\cap \{a,a+1,\cdots,b-1\}|>1\}.\;
\end{equation*}
Then $\cJ^{2}_i=\emptyset$ and hence
	\[\cJ_{l_j}=\cJ^1_{l_j}=\{(a,b):\;\max\{\bd_i,l_{j-1}\}+1\leq a\leq l_j, l_j+1\leq b\leq l_{j+1}\}\]
	for $1\leq j\leq |S|$.\;
\end{lem}
\begin{proof}
	Suppose that $S\cap \{a,a+1,\cdots,b-1\}=\{i_{l_q},\cdots,i_{l_r}\}$ such that $q\leq j<r$.\;Therefore,\;there exists $j'$ such that  $q\leq j'\leq r$ and $j'\neq j$,\;and then $a\leq i_{l_{j'}}\leq \bd_i<\bd_i+1\leq i_{l_j}$,\;this is impossible.\;
\end{proof}

\subsubsection{Construction of map $t^{\flat}_{D_{\sigma}}$}


Similar to \cite[(46)]{BDcritical25},\;we consider the morphisms of $E$-vector spaces:
\begin{equation}\label{compositiondfnpisi}
	\begin{aligned}
		&\ext^1_{G,\sigma}\Big(\bigoplus_{u\in \cI_{i}^{\flat}}C_{i,u,\sigma},\pi_{\lalg}(\underline{x},\bh)\Big)\otimes \ext^1_{G,\sigma}\Big(\bigoplus_{u\in \cI_{i}^{\flat}}C_{i,u,\sigma},\pi_{\lalg}(\underline{x},\bh)\Big)^{\vee}\\
		\xrightarrow{\sim}&\;\ext^1_{G,\sigma}\Big(\Big(\bigoplus_{u\in \cI_{i}^{\flat}}C_{i,u,\sigma}\Big)\otimes\ext^1_{G,\sigma}\Big(\bigoplus_{u\in \cI_{i}^{\flat}}C_{i,u,\sigma},\pi_{\lalg}(\underline{x},\bh)\Big) ,\pi_{\lalg}(\underline{x},\bh)\Big)\\
		\xrightarrow{\sim}&\;\ext^1_{G,\sigma}\Big(\Big(\bigoplus_{u\in \cI_{i}^{\flat}}C_{i,u,\sigma}\Big)\otimes_E\fil_i^{\max}(D_{\sigma})^{\flat},\pi_{\lalg}(\underline{x},\bh)\Big),
	\end{aligned}
\end{equation}
where the second morphism is the push-forward induced by the composition
\[\fil_i^{\max}(D_{\sigma})^{\flat}\hookrightarrow \big(\bigwedge\nolimits_{E}^{\!n-i}\!D_{\sigma}\big)^{\flat}\xrightarrow{\epsilon_i^{-1}}\ext^1_{G,\sigma}\Big(\bigoplus_{u\in \cI_{i}^{\flat}}C_{i,u,\sigma},\pi_{\lalg}(\underline{x},\bh)\Big).\] 
Let $\pi^{\flat}_{s_i,\sigma}(\underline{x},\bh)$ (resp.,\;$\pi^{\flat}_{s_i,u,\sigma}(\underline{x},\bh)$) be a representative of the image for the canonical vector of the first term of (\ref{compositiondfnpisi}) by the composition (\ref{compositiondfnpisi}),\;which lies in the extension group 
\[\ext^1_{G,\sigma}\Big(\Big(\bigoplus_{u\in \cI_{i}^{\flat}}C_{i,u,\sigma}\Big)\otimes_E\fil_i^{\max}(D_{\sigma})^{\flat},\pi_{\lalg}(\underline{x},\bh)\Big),\;\text{resp.,\;}\ext^1_{G,\sigma}\Big(C_{i,u,\sigma}\otimes_E\fil_i^{\max}(D_{\sigma})^{\flat},\pi_{\lalg}(\underline{x},\bh)\Big).\]
We have injections $\pi_{\lalg}(\underline{x},\bh)\hookrightarrow \pi^{\flat}_{s_i,u,\sigma}(\underline{x},\bh) \hookrightarrow \pi^{\flat}_{s_i,\sigma}(\underline{x},\bh)$ of $G$-representations and a canonical isomorphism
\[\bigoplus^{u\in \cI_{i}^{\flat}}_{\pi_{\lalg}(\underline{x},\bh)}\pi^{\flat}_{s_i,u,\sigma}(\underline{x},\bh)\xrightarrow{\sim }\pi^{\flat}_{s_i,\sigma}(\underline{x},\bh).\;\]

\begin{pro}\label{construext1}Denote
	\begin{equation}
		\homo^{\flat,\dagger}_{\fil}(D_{\sigma},D_{\sigma}):=\sum_{i\in \Delta'}g_{i,\sigma}\Big(\homo_E\Big(\big(\bigwedge\nolimits_{E}^{\!n-i}\!D_{\sigma}\big)^{\flat},\fil_i^{\max}(D_{\sigma})^{\flat}\Big)\Big)\hookrightarrow \homo_{\fil}(D_{\sigma},D_{\sigma}).
	\end{equation}
There is a surjection of finite dimensional $E$-vector spaces which only depends on the $(\epsilon_{u,i})_{i\in \Delta',u\in \cI_{i}^{\flat}}$ and  on a choice of $\log_p(p)\in E$:
\[t^{\flat}_{D_{\sigma}}:\ext^1_{G,\sigma}\big(\pi_{\lalg}(\underline{x},\bh),\pi^{\flat}_{1,\sigma}(\underline{x},\bh)\big)\rightarrow\ext^{1,\flat}_{\gr}({\Dpik},{\Dpik})\oplus\homo^{\flat,\dagger}_{\fil}(D_{\sigma},D_{\sigma}).\]
Taking sum over $\sigma\in \Sigma_L$ leads to a surjection:
\[t^{\flat}_{D}:\ext^1_{G}\big(\pi_{\lalg}(\underline{x},\bh),\pi^{\flat}_{1}(\underline{x},\bh)\big)\rightarrow\ext^{1,\flat}_{\gr}({\Dpik},{\Dpik})\oplus\homo^{\flat,\dagger}_{\fil}(D,D),\]
where $\homo^{\flat,\dagger}_{\fil}(D,D):=\oplus_{\sigma\in \Sigma_L}\homo^{\flat,\dagger}_{\fil}(D_{\sigma},D_{\sigma})$.\;
\end{pro}
\begin{proof}We define $t^{\flat}_{D_{\sigma}}$ in restriction to $\ext^1_{\sigma}\big(\pi_{\lalg}(\underline{x},\bh),\pi_{\lalg}(\underline{x},\bh)\big)$ firstly.\;This is essentially the Step 2 in the proof of \cite[Proposition 2.2.4]{BDcritical25}.\;It suffices to use \cite[(63)]{BDcritical25} and replace \cite[(64) \& (65)]{BDcritical25} with 
\begin{equation}
	\begin{aligned}
		t^{\flat}_{D_{\sigma}}|_{\homo_{\mathrm{sm}}(\bZ_{S_0}(L),E)}&\xrightarrow{\sim }\ext^{1,\flat}_{\gr}({\Dpik},{\Dpik})= \prod_{i=1}^s\ext^{1}_g(E_i[1/t],E_i[1/t])\\
		(\psi_1,\cdots,\psi_s)&\mapsto \Big(E_{i}[1/t]\otimes_{\cR_{E,L}}\cR_{E[\epsilon]/\epsilon^2,L}(1+\psi_{i}\epsilon)\Big)_{1\leq i\leq s}.
	\end{aligned}
\end{equation}
The Step 3 in the proof of \cite[Proposition 2.2.4]{BDcritical25} need more modification.\;Choose $i\in \Delta'$,\;wed have isomorphisms 
	\begin{equation}
		\begin{aligned}
			\ext^1_{G,\sigma}\Big(\pi_{\lalg}(\underline{x},\bh),\pi^{\flat}_{s_i,\sigma}(\Dpik)/\pi_{\lalg}(\underline{x},\bh)\Big)&\xrightarrow{\sim}\ext^1_{G,\sigma}\Big(\pi_{\lalg}(\underline{x},\bh),\Big(\bigoplus_{u\in \cI_{i}^{\flat}}C_{i,u,\sigma}\Big)\otimes_E\fil_i^{\max}(D_{\sigma})^{\flat}\Big)\\
			&\xleftarrow{\sim}\ext^1_{G,\sigma}\Big(\pi_{\lalg}(\underline{x},\bh),\bigoplus_{u\in\cI_{i}^{\flat}}C_{i,u,\sigma}\Big)\otimes_E\fil_i^{\max}(D_{\sigma})^{\flat}\\
			&\xrightarrow[(\ref{dualpairing})]{\sim}\ext^1_{G,\sigma}\Big(\bigoplus_{u\in \cI_{i}^{\flat}}C_{i,u,\sigma},\pi_{\lalg}(\underline{x},\bh)\Big)^{\vee}\otimes_E\fil_i^{\max}(D_{\sigma})^{\flat}\\
			&\xleftarrow{\sim}\homo_E\Big(\big(\bigwedge\nolimits_{E}^{\!n-i}\!D_{\sigma}\big)^{\flat},\fil_i^{\max}(D_{\sigma})^{\flat}\Big).
		\end{aligned}
	\end{equation}
We get the desired surjection $t^{\flat}_{D_{\sigma}}$ by taking sum over  $i\in \Delta'$.\;
\end{proof}
\begin{rmk}
Similar to \cite[Proposition 2.2.7]{BDcritical25},\;we can show that $t^{\flat}_{D_{\sigma}}$ does not depends on the choices of $(\epsilon_{u,i})_{i\in \Delta',u\in \cI_{i}^{\flat}}$.\;
\end{rmk}

\subsubsection{Image of $t^{\flat}_{D_{\sigma}}$}

For  $u\in \sW_s$,\;let $u^{\sharp}\in \sW_n$ be the unique element sending $e_{i,j,\sigma}$ for $1\leq j\leq r_i$ to $e_{u^{-1}(i),j,\sigma}$ for $1\leq j\leq r_i$,\;which thus induces a permutation on $e_{1,\sigma},\cdots,e_{n,\sigma}$.\;

For $i\in \Delta'$ and $u\in \cI_{i}^{\flat}$,\;we  define $\homo_{\fil,\bF_u}^i(D_{\sigma},D_{\sigma})\subseteq \homo_{\fil,\bF_u}(D_{\sigma},D_{\sigma})$ by
\begin{equation}\label{dfnforhomofilmaxilevi}
	\begin{aligned}
		\homo_{\fil,\bF_u}^i(D_{\sigma},D_{\sigma}):=\left\{f\in \homo_{\fil,\bF_u}(D_{\sigma},D_{\sigma}):\exists\;a,b\in E\; \text{such that\;} f(e_{(u^{\sharp})^{-1}(j),\sigma})=ae_{(u^{\sharp})^{-1}(j),\sigma},\right.\\ \left. \forall\;1\leq j\leq i,f(e_{(u^{\sharp})^{-1}(j),\sigma})-be_{(u^{\sharp})^{-1}(j),\sigma}\in \oplus_{l=1}^{i}Ee_{(u^{\sharp})^{-1}(l),\sigma},\forall\;i+1\leq j\leq n\right\},
	\end{aligned}
\end{equation}
Keep the notation in Section \ref{reinterfor33}.\;Then we have (note that $\tau_{\widehat{i},\sigma}\subseteq \tau_{S_0^u,\sigma}\subseteq \fp_{S^u_0,\sigma}$,\;see Remark \ref{explainHomasLiealg})
\begin{equation}
	\begin{aligned}
		&\homo^i_{\fil,\bF_u}(D_{\sigma},D_{\sigma})\cong \mathrm{Ad}_{(u^{\sharp})^{-1}}(\tau_{\widehat{i},\sigma})\cap \mathrm{Ad}_{g_{\sigma}}(\fb_{\sigma}).\;
	\end{aligned}
\end{equation}
As in \cite[(134) \& (137)]{BDcritical25},\;sending $f$ to $(a\log_{p,\sigma},b\log_{p,\sigma})$ ($a,b$ are given in the definition (\ref{dfnforhomofilmaxilevi})) induces a canonical map 
\[f_{i,\sigma}:\homo^i_{\fil,\bF_u}(D_{\sigma},D_{\sigma})\rightarrow \homo_{\sigma}(\bL_{\widehat{i}}(\cO_L),E),\]
which is an isomorphism by comparing dimensions.\;Indeed,\;we have (see Remark \ref{explainHomasLiealg})
\begin{equation}
	\begin{aligned}
		\mathrm{Ad}_{(u^{\sharp})^{-1}}(\tau_{\widehat{i},\sigma})\cap \mathrm{Ad}_{g_{\sigma}}(\fb_{\sigma})=\mathrm{Ad}_{(u^{\sharp})^{-1}b_{\sigma}(u)}(\fz_{\widehat{i},\sigma}),\;
	\end{aligned}
\end{equation}
and thus $\dim_E\homo^i_{\fil,\bF_u}(D_{\sigma},D_{\sigma})=2$.\;Put
\begin{equation}\label{dfnforhomoflat}
	\homo^{\flat}_{\fil}(D_{\sigma},D_{\sigma}):=\sum_{i\in\Delta',u\in \cI_{i}^{\flat}}\homo^i_{\fil,\bF_u}(D_{\sigma},D_{\sigma})\subseteq\homo_{\fil}(D_{\sigma},D_{\sigma}).\;
\end{equation}

\begin{pro}\label{dfnisoforflatfilhomo}
	$\homo^{\flat}_{\fil}(D_{\sigma},D_{\sigma})\cong \mathrm{Ad}_{g_{\sigma}}(\fb_{\sigma})_{S_0}^{\flat}=\sum_{u\in \sW_s}\mathrm{Ad}_{(u^{\sharp})^{-1}}(\tau_{S^u_0,\sigma})\cap \mathrm{Ad}_{g_{\sigma}}(\fb_{\sigma})$.\;
\end{pro}
\begin{proof}
	It suffices to study the envelope $\sum_{i\in\Delta',u\in \cI_{i}^{\flat}}\mathrm{Ad}_{(u^{\sharp})^{-1}}(\tau_{\widehat{i},\sigma})\cap \mathrm{Ad}_{g_{\sigma}}(\fb_{\sigma})$ in $\mathrm{Ad}_{g_{\sigma}}(\fb_{\sigma})$.\;The left hand side equals to 
	$\sum_{u\in \sW_s}\sum_{1\leq j\leq s}\mathrm{Ad}_{(u^{\sharp})^{-1}}(\tau_{\widehat{t^u_j},\sigma})\cap \mathrm{Ad}_{g_{\sigma}}(\fb_{\sigma})$.\;By (\ref{noncriticalintersection}),\;we have $\mathrm{Ad}_{(u^{\sharp})^{-1}}(\tau_{\widehat{{t^u_j}},\sigma})\cap \mathrm{Ad}_{g_{\sigma}}(\fb_{\sigma})=\mathrm{Ad}_{(u^{\sharp})^{-1}b_{\sigma}(u)}(\fz_{\widehat{t^u_j},\sigma})$.\;Therefore,\;we deduce that 
	\begin{equation}
		\begin{aligned}
			\sum_{1\leq j\leq s}\mathrm{Ad}_{(u^{\sharp})^{-1}}(\tau_{\widehat{t^u_j},\sigma})\cap \mathrm{Ad}_{g_{\sigma}}(\fb_{\sigma})=\sum_{1\leq j\leq s}\mathrm{Ad}_{(u^{\sharp})^{-1}b_{\sigma}(u)}(\fz_{\widehat{t^u_j},\sigma})\\
			=\mathrm{Ad}_{(u^{\sharp})^{-1}b_{\sigma}(u)}(\fz_{S^u_0,\sigma})=\mathrm{Ad}_{(u^{\sharp})^{-1}}(\tau_{S_0,\sigma})\cap \mathrm{Ad}_{g_{\sigma}}(\fb_{\sigma}),
		\end{aligned}
	\end{equation}
	so that $\homo^{\flat}_{\fil}(D_{\sigma},D_{\sigma})=\mathrm{Ad}_{g_{\sigma}}(\fb_{\sigma})_{S_0}^{\flat}$.\;
\end{proof}
Consider the following composition of isomorphisms (similar to \cite[(142)-(144)]{BDcritical25}):
\begin{equation}\label{tDmapforiusigma}
	\begin{aligned}
		\ext^1_{G,\sigma}\left(\pi_{\lalg}(\underline{x},\bh),V
		_{i,u,\sigma}\right)&\xleftarrow[\text{Prop.\;}\ref{reinterforextgp}]{\sim}\homo_{\sm}(\bZ_{S^u_0}(L),E)\oplus_{\homo_{\sm}(\bL_{\widehat{i}}(L),E)}\homo_{\sigma}(\bL_{\widehat{i}}(L),E)\\&\xleftarrow{\sim}\homo_{\sm}(\bZ_{S^u_0}(L),E)\oplus\homo_{\sigma}(\bL_{\widehat{i}}(\cO_L),E)\\
		&\xrightarrow[\oplus_{\sigma}f_{i,\sigma}^{-1}]{\sim}\ext^{1,\flat}_{\gr}({\Dpik},{\Dpik})\oplus\homo^i_{\fil,\bF_u}(D_{\sigma},D_{\sigma})\\
		&\hookrightarrow \ext^{1,\flat}_{\gr}({\Dpik},{\Dpik})\oplus\homo_{\fil}(D_{\sigma},D_{\sigma}).\;
	\end{aligned}
\end{equation}
Taking sum of  (\ref{tDmapforiusigma}) over all the $i\in\Delta'$ and $u\in \cI_{i}^{\flat}$ also leads to a surjection:
\[\ext^1_{G,\sigma}\left(\pi_{\lalg}(\underline{x},\bh),\pi^{\flat}_{1,\sigma}(\underline{x},\bh)\right)\twoheadrightarrow\ext^{1,\flat}_{\gr}({\Dpik},{\Dpik})\oplus\homo^{\flat}_{\fil}(D_{\sigma},D_{\sigma}).\]
The remainder of this section will compare such surjection with $t^{\flat}_{D_{\sigma}}$ and   show that 
\[\homo^{\flat,\dagger}_{\fil}(D_{\sigma},D_{\sigma})=\homo^{\flat}_{\fil}(D_{\sigma},D_{\sigma}).\]


\begin{pro}\label{comparesharpflat}
Let $i\in \Delta'$ and $u\in \cI_{i}^{\flat}$.\;The composition (given in (\ref{tDmapforiusigma}))
\begin{equation}
	\begin{aligned}
		\homo_{\sm}(\bZ_{S^u_0}(L),E)\oplus\homo_{\sigma}(\bL_{\widehat{i}}(\cO_L),E)\xrightarrow[f_{i,\sigma}^{-1}]{\sim}\ext^{1,\flat}_{\gr}({\Dpik},{\Dpik})\oplus\homo^i_{\fil,\bF_u}(D_{\sigma},D_{\sigma})
	\end{aligned}
\end{equation}
coincides with the composition
\begin{equation}
	\begin{aligned}
		\homo_{\sm}(\bZ_{S^u_0}(L),E)\oplus\homo_{\sigma}(\bL_{\widehat{i}}(\cO_L),E)&\xrightarrow{\sim}\ext^1_{G,\sigma}\left(\pi_{\lalg}(\underline{x},\bh),V_{i,u,\sigma}\right)\\&\xrightarrow{t^{\flat}_{D_{\sigma}}}\ext^{1,\flat}_{\gr}({\Dpik},{\Dpik})\oplus\homo^{\flat,\dagger}_{\fil}(D_{\sigma},D_{\sigma}).\;
	\end{aligned}
\end{equation}In particular,\;we have $\homo^{\flat,\dagger}_{\fil}(D_{\sigma},D_{\sigma})\cong \homo^{\flat}_{\fil}(D_{\sigma},D_{\sigma})$.\;
\end{pro}
\begin{proof}
This is essentially \cite[Proposition 2.5.4]{BDcritical25}.\;Only need to show that the key statement of  \cite[Lemma 2.2.9 (iii)]{BDcritical25} in the last paragraph of the proof of \cite[Proposition 2.5.4]{BDcritical25} is also suitable for our setting.\;Keep the notation in the proof of \cite[Proposition 2.5.4]{BDcritical25}.\;For $c_u\in \ext^1_{G,\sigma}\left(\pi_{\lalg}(\underline{x},\bh),C_{i,u,\sigma}\right)$,\;we claim that: 
\begin{itemize}
	\item[(1)] For $t\in I_j$ and $j\in I_i(u)$,\;$t^{\flat}_{D_{\sigma}}(c_u)(e_t)=0$.\;
	\item[(2)] For $t\in I_j$ and $j\not\in I_i(u)$,\;there exists a unique $\lambda_u\in E$ such that $c_{u,i}=\lambda_{u,i}\log$ where $\log_{p,\sigma}\in \ext^1_{G,\sigma}\left(\pi_{\lalg}(\underline{x},\bh),V_{i,u,\sigma}\right)$ is the image of $\log_{p,\sigma}\in \homo_{\sigma}(\GLN_{n-i}(\cO_L),E)$ (see Remark \ref{rmkforextgroupsigma}) and 
	\[t^{\flat}_{D_{\sigma}}(c_u)(e_t)-\lambda_{u,i}e_t\in\bigoplus_{t'\in I_{q},q\in I_i(u)}Ee_{t'}\]
\end{itemize}
We drop $\sigma$ in the subscript of $e_{i,\sigma}$ for simplicity.\;The argument in \cite[Page 33]{BDcritical25} and the diagram \cite[(78)]{BDcritical25} are still true for our setting,\;when we replace $\bigwedge\nolimits_E^{\!n-i}\!D_{\sigma}$ (resp.,\;$e_{I^c}$,\;resp,\;$\fil_i^{\max}(D_{\sigma})$, resp.,\;$C(I)$) with $\big(\bigwedge\nolimits_{E}^{\!n-i}\!D_{\sigma}\big)^{\flat}$ (resp.,\;$\be_{I_i(u)^c}$,\;resp.,\;$\fil_i^{\max}(D_{\sigma})^{\flat}$,\;resp.,\;$C_{i,u,\sigma}$).\;For any $t\in I_j$ with $j\in I_i(u)^c$,\;define 
\[\fil_H^{-\bh_{i+1,\sigma}}(D_{\sigma})^{(t)}=\fil_H^{-\bh_{i+1,\sigma}}(D_{\sigma})\cap\big(\oplus_{j\neq t}e_{j}\big).\]
We also have $\fil_H^{-\bh_{i+1,\sigma}}(D_{\sigma})^{(t)}=n-i-1$ and for any non-zero $x\in \bigwedge\nolimits_E^{\!n-i-1}\!\fil_H^{-\bh_{i+1,\sigma}}(D_{\sigma})^{(t)}$,\;we have 
\[x\wedge e_t\notin \Big(\bigoplus_{J\neq I_i(u)^c}E\be_J\Big)\oplus \big(\bigwedge\nolimits_{E}^{\!n-i}\!D_{\sigma}\big)^{\flat,c},\]
which implies that $x\wedge e_t\notin \bigoplus_{J\neq I_i(u)^c}E\be_J$.\;Note that 
\[e^{\star}_{I_i(u)^c}(x\wedge e_t)-x\wedge e_t\in \Big(\bigoplus_{J\neq I_i(u)^c}E\be_J\Big)\oplus \big(\bigwedge\nolimits_{E}^{\!n-i}\!D_{\sigma}\big)^{\flat,c}.\]
By definition of $t^{\flat}_{D_{\sigma}}(\log_{p,\sigma})$,\;this deduces that $x\wedge (t^{\flat}_{D_{\sigma}}(\log_{p,\sigma})(e_t)-e_t)\in\bigoplus_{J\neq I_i(u)^c}E\be_J$,\;and thus 
\begin{equation}\label{proofforzerocoeff}
	t^{\flat}_{D_{\sigma}}(c_u)(e_t)-e_t\in\bigoplus_{t'\in I_{q},q\neq j}Ee_{t'}
\end{equation}
It remains to show that the coefficient of $e_{j'}$ in the right hand side is zero when $t'\in I_{q}$ and $q\not\in I_i(u)$.\;Indeed,\;choose the same elements $\{f_{t''}:{t''}\in I_j,j\in I_i(u)^c\}$ as in  \cite[(84)]{BDcritical25},\;which form a basis of $\fil_H^{-\bh_{i+1,\sigma}}(D_{\sigma})$.\;Suppose that the coefficient of $\be_{t'}$ in the right hand side is non-zero.\;For any $t'\neq t$,\;we also see that the coefficient of $\be_{I_{i}(u)^c}$ is zero in 
\[\big(\wedge_{t''\neq t'}f_{t''}\big)\wedge e_t\in \bigwedge\nolimits_E^{\!n-i-1}\!\fil_H^{-\bh_{i+1,\sigma}}(D_{\sigma})\wedge D_{\sigma}.\]
But $\big(\wedge_{t''\neq t'}f_{t''}\big)\wedge t^{\flat}_{D_{\sigma}}(\log_{p,\sigma})(e_t)=0$ in $\bigwedge\nolimits_E^{\!n-i}\!D_{\sigma}$ and $e_{t'}$ cannot appears in $f_{t''}$ if $t''\neq t'$,\;then by (\ref{proofforzerocoeff}),\;the coefficient of $\be_{I_{i}(u)^c}$ in $\big(\wedge_{t''\neq t'}f_{t''}\big)\wedge e_t$ must be non-zero,\;this is impossible.\;This completes the proof.\;
\end{proof} 

For $\star\in\{\emptyset,\sigma\}$,\;we  have the following composition 
\begin{equation}\label{dfnforgammaD}
	\begin{aligned}
		\gamma^{\flat}_{\Dpik,\star}:\bigoplus_{u\in\sW_s}\overline{\ext}^{1,\flat}_{\star,u}(\Dpik,\Dpik)
		&\xrightarrow[\sim]{\kappa^{\flat}_{\Dpik,\star}} \bigoplus_{u\in\sW_s}\homo_{\star}\big(\bZ_{S^u_0}(L),E\big)
		\xrightarrow{\zeta_{\Dpik,\star}} \ext^1_{G}\left(\pi_{\lalg}(\underline{x},\bh),\pi^{\flat}_{1,\star}(\underline{x},\bh)\right).\;
	\end{aligned}
\end{equation}
where $\kappa^{\flat}_{\Dpik,\star}:=\oplus_{u\in\sW_s}\kappa_{u}$ is the first isomorphism and $\zeta_{\Dpik,\star}:=\oplus_{u\in\sW_s}\zeta_{u,\star}$ (thus $\gamma^{\flat}_{\Dpik,\star}=\zeta_{\Dpik,\star}\circ \kappa^{\flat}_{\Dpik,\star}$).\;Consider the natural map:
\[\overline{g}^{\flat}_{\Dpik,\star}:\bigoplus_{u\in\sW_s}\overline{\ext}^{1,\flat}_{\star,u}(\Dpik,\Dpik)\rightarrow \overline{\ext}^{1}(\Dpik,\Dpik).\]
The same strategy as in   \cite[Proposition 2.5.5,\;Corollary 2.5.6]{BDcritical25} and Proposition \ref{comparesharpflat} gives:
\begin{pro}\label{mainthmforfullrefine} $f^{\flat}_{\Dpik,\sigma}\circ\overline{g}^{\flat}_{\Dpik,\sigma}=t^{\flat}_{D_{\sigma}}\circ\gamma^{\flat}_{\Dpik,\sigma}$.\;After taking the sum over $\sigma\in \Sigma_L$,\;$f^{\flat}_{\Dpik}\circ\overline{g}^{\flat}_{\Dpik}=t^{\flat}_{D}\circ\gamma^{\flat}_{\Dpik}$.\;
\end{pro}

Let $\pi^{\flat}_{\min}(\Dpik)_{\star}$ be  the tautological extension of $\ker(t^{\flat}_{D_{\star}})\otimes_E\pi_{\lalg}(\underline{x},\bh)$ by $\pi^{\flat}_{1,\star}(\underline{x},\bh)$.\;Then 
\[\ker(t^{\flat}_{D})=\oplus_{\sigma\in \Sigma_L}\ker(t^{\flat}_{D_{\sigma}}),\quad\pi^{\flat}_{\min}(\Dpik)=\bigoplus_{\pi_{\lalg}(\underline{x},\bh)}^{\sigma\in\Sigma_L}\pi^{\flat}_{\min}(\Dpik)_{\sigma}.\]
Similar to \cite[Corollary 2.2.13]{BDcritical25},\;we can show that  $\pi^{\flat}_{\min}(\Dpik)_{\star}$ does not depends on any choices of $(\epsilon_{u,i})_{i\in \Delta',u\in \cI_{i}^{\flat}}$ and $\log_p(p)$.\;

\subsubsection{The $p$-adic Hodge parameters determined by the kernel of  $t^{\flat}_{D_{\sigma}}$}

This section describes what information of $p$-adic Hodge parameters is determined by the kernel of  $t^{\flat}_{D_{\sigma}}$ or   $\pi^{\flat}_{\min}(\Dpik)_{\sigma}$.\;For $S\subseteq \Delta'$,\;put
\[\pi^{\flat}_{S,\sigma}(\underline{x},\bh):=\bigoplus^{i\in S}_{\pi_{\lalg}(\underline{x},\bh)}\pi^{\flat}_{s_i,\sigma}(\underline{x},\bh).\]
Let $\ext^1_{G,\mathrm{inf}}\big(\pi_{\lalg}(\underline{x},\bh),\pi^{\flat}_{S,\sigma}(\underline{x},\bh)\big)$ be the subspace of $\ext^1_{G}\big(\pi_{\lalg}(\underline{x},\bh),\pi^{\flat}_{S,\sigma}(\underline{x},\bh)\big)$ of locally $\bQ_p$-analytic extensions with an infinitesimal characters.\;Put
\[\ker(t^{\flat}_{D_{\sigma}})_S:=\ker(t^{\flat}_{D_{\sigma}}|_{\ext^1_{G,\mathrm{inf}}(\pi_{\lalg}(\underline{x},\bh),\pi^{\flat}_{S,\sigma}(\underline{x},\bh))}),\quad \ker(t^{\flat}_{D_{\sigma}}):=\ker(t^{\flat}_{D_{\sigma}})_{\Delta'}\]
for simplicity.\;Note that $\ker(t^{\flat}_{D_{\sigma}})_{S_1}\subseteq \ker(t^{\flat}_{D_{\sigma}})_{S_2}\subseteq \ker(t^{\flat}_{D_{\sigma}})$ if $S_1\subseteq S_2\subseteq \Delta'$.

Let $\Delta':=\{i_1<i_2<\cdots<i_{|\Delta'|}\}$ and $S:=\{i_{l_1}<i_{l_2}<\cdots<i_{l_{|S|}}\}$ for $1\leq l_1<\cdots<l_{|S|}\leq |\Delta'|$.\;Consider the following compositions:
\begin{equation}\label{compositionfinalSver}
	\begin{aligned}
	\beta_S:	\ext^1_{G,\sigma}\Big(\pi_{\lalg}(\underline{x},\bh),\pi^{\flat}_{S,\sigma}(\underline{x},\bh)\Big)&\twoheadrightarrow \ext^1_{G,\sigma}\Big(\pi_{\lalg}(\underline{x},\bh),\pi^{\flat}_{S,\sigma}(\underline{x},\bh)/\pi_{\lalg}(\underline{x},\bh)\Big)\\
		&\xrightarrow{\sim}\bigoplus_{i\in S}\ext^1_{G,\sigma}\Big(\pi_{\lalg}(\underline{x},\bh),\pi_{s_i}(D_{\sigma})/\pi_{\lalg}(\underline{x},\bh)\Big)\\
		&\xrightarrow{\sim}\bigoplus_{i\in S}\homo_E\Big(\big(\bigwedge\nolimits_{E}^{\!n-i}\!D_{\sigma}\big)^{\flat},\fil_i^{\max}(D_{\sigma})^{\flat}\Big).\;
	\end{aligned}
\end{equation}
By Lemma \ref{lemforwedge},\;we obtain the following  surjection
\begin{equation}
	\begin{aligned}
		g_{\sigma}:\bigoplus_{i\in S}\homo_E\Big(\big(\bigwedge\nolimits_{E}^{\!n-i}\!D_{\sigma}\big)^{\flat},&\fil_i^{\max}(D_{\sigma})^{\flat}\Big)\\&\twoheadrightarrow \sum_{i\in S}\left\{f\in\homo_E\big(D_{\sigma,(i)}^{\flat},\fil_H^{-\bh_{i+1,\sigma}}(D_{\sigma})\big):f|_{\fil_H^{-\bh_{i+1,\sigma}}(D_{\sigma})} \text{scalar}\right\}.
	\end{aligned}
\end{equation}

\begin{pro}\label{propfor2ndfil}
	\begin{itemize}
		\item[(1)] The  image under (\ref{compositionfinalSver}) of the subspace $\ext^1_{G,\mathrm{inf}}\Big(\pi_{\lalg}(\underline{x},\bh),\pi^{\flat}_{S,\sigma}(\underline{x},\bh)\Big)$  is 
		\[\homo_E\Big(\big(\bigwedge\nolimits_{E}^{\!n-i}\!D_{\sigma}\big)^{\flat}/\fil_i^{\max}(D_{\sigma})^{\flat},\fil_i^{\max}(D_{\sigma})^{\flat}\Big).\]
		\item[(2)] Let $i=i_{l_j}\in S$.\;Then the  image under (\ref{compositionfinalSver}) of the subspace $\ker(t^{\flat}_{D_{\sigma}})_S$ is	the  inverse image of 
		\[\Big\{f\in\homo_E\Big(\fil_H^{-\bh_{\bd_i+1,\sigma}}/\fil_H^{-\bh_{i+1,\sigma}},\fil_H^{-\bh_{i+1,\sigma}}\Big):f\Big(\fil_H^{-\bh_{{\max\{i_{l_{j-1}},\bd_i\}}+1,\sigma}}/\fil_H^{-\bh_{i+1,\sigma}}\Big)\subseteq\fil_H^{-\bh_{i_{l_{j+1}}+1,\sigma}}\Big\}\]
		via $g_{i,\sigma}$ (we omit $(D_{\sigma})$ in $\fil_H^{-\bh_{\bullet,\sigma}}(D_{\sigma})$ for simplicity),\;which is a subspace of
		\[\homo_E\Big(\big(\bigwedge\nolimits_{E}^{\!n-i}\!D_{\sigma}\big)^{\flat}/\fil_{S,i}^{2^{\mathrm{nd}}-\max}(D_{\sigma})^{\flat},\fil_i^{\max}(D_{\sigma})^{\flat}\Big)\]
		(compare with \cite[(99) \& (102)]{BDcritical25}).\;
	%
	
	\end{itemize}
\end{pro}
\begin{proof}
We examine the proof of \cite[Propositions 2.3.4 \& 2.3.7]{BDcritical25}.\;We only need to take the sum for $\{I_i(u)\}_{u\in \cI_i^{\flat}}$ in  \cite[(96)]{BDcritical25},\;so that we need to replace $\fil_i^{\max}(D_{\sigma})$ with $\fil_i^{\max}(D_{\sigma})^{\flat}$,\;$(1)$ follows.\;It remains to study the image of $\ker(g_{\sigma})$ via $g_{i,\sigma}$ when $i\in S$.\;When $S_0=\emptyset$,\;then $\Delta'=\Delta$ and $D_{\sigma,(i)}^{\flat}=D_{\sigma}$ for all $i$,\;\cite[Proposition 2.3.7]{BDcritical25} prove that the subspace $\homo_E\Big(\big(\bigwedge\nolimits_{E}^{\!n-i}\!D_{\sigma}\big)^{\flat}/\fil_{S,i}^{2^{\mathrm{nd}}-\max}(D_{\sigma})^{\flat},\fil_i^{\max}(D_{\sigma})^{\flat}\Big)$ in $(2)$ gives such image.\;We draw the matrix of each $f_j\in \homo_E(D_{\sigma},D_{\sigma})$ (note that $f_j\in \homo_E(D_{\sigma,(j)}^{\flat},D_{\sigma})$ in our situation) in an adapted basis of $D_{\sigma}$ for the filtration $\fil^{-\bh_{i,\sigma}}_H(D_{\sigma})$,\;then we get the matrix form of $f_{i_{l_1}}+\cdots+f_{i_{l_{|S|}}}$.\;If $f_{i_{l_1}}+\cdots+f_{i_{l_{|S|}}}=0$,\;then $f_{j}$ also satisfies that  (see \cite[(102)]{BDcritical25})
\begin{equation*}
	\begin{aligned}
		&\Big\{f_{|S|}\in \homo_E\Big(D_{\sigma}/\big(D_{\sigma}^{\flat,(i_{l_{|S|}})}+\fil_H^{-\bh_{i_{l_{|S|}}+1,\sigma}}\big),\fil_H^{-\bh_{i_{l_{|S|}}+1,\sigma}}\Big):\;f_{|S|}\Big(\fil_H^{-\bh_{i_{l_{|S|-1}}+1,\sigma}}/\fil_H^{-\bh_{i_{l_{|S|}}+1,\sigma}}\Big)=0\Big\},\\
		&\Big\{f_j\in\homo_E\Big(D_{\sigma}/\big(D_{\sigma}^{\flat,(i_{l_{j}})}+\fil_H^{-\bh_{i_{l_{j}}+1,\sigma}}\big),\fil_H^{-\bh_{i_{l_{j}}+1,\sigma}}\Big):f_j\Big(\fil_H^{-\bh_{i_{l_{j-1}}+1,\sigma}}/\fil_H^{-\bh_{i_{l_{j}}+1,\sigma}}\Big)\subseteq\fil_H^{-\bh_{i_{l_{j+1}}+1,\sigma}}\Big\}
	\end{aligned}
\end{equation*}
when $j<|S|$.\;Then the same proof of \cite[(102)]{BDcritical25} shows that the image of $\ker(g_{\sigma})$ must lands in the space $\homo_E\Big(\big(\bigwedge\nolimits_{E}^{\!n-i}\!D_{\sigma}\big)^{\flat}/\fil_{S,i}^{2^{\mathrm{nd}}-\max}(D_{\sigma})^{\flat},\fil_i^{\max}(D_{\sigma})^{\flat}\Big)$.\;In general,\;the irredularity of the shape $\bL_{S_0}$ makes $\ker(g_{\sigma})$ smaller.\;Note that
\[\homo_E(D_{\sigma},D_{\sigma})=\bigoplus_{1\leq a,b\leq n}\homo_E\big(\gr_H^{-\bh_{a,\sigma}},\gr_H^{-\bh_{b,\sigma}}\big)\]
We use Lemma \ref{filforhomospace} and study the kernel of 
\[ \bigoplus_{i\in S}\bigoplus_{\substack{\bd_i+1\leq a\leq i\\i+1\leq b\leq n}}\homo_E\big(\gr_H^{-\bh_{a,\sigma}},\gr_H^{-\bh_{b,\sigma}}\big)\rightarrow \bigoplus_{1\leq a,b\leq n}\homo_E\big(\gr_H^{-\bh_{a,\sigma}},\gr_H^{-\bh_{b,\sigma}}\big).\]
We see that the kernel of the map 
\[g_{a,b}:\bigoplus_{i\in \cS_{a,b}(S)}\homo_E\big(\gr_H^{-\bh_{a,\sigma}},\gr_H^{-\bh_{b,\sigma}}\big)\rightarrow \bigoplus_{1\leq a,b\leq n}\homo_E\big(\gr_H^{-\bh_{a,\sigma}},\gr_H^{-\bh_{b,\sigma}}\big).\]
has $E$-dimension $|\cS_{a,b}(S)|-1$.\;Moreover,\;for any $i\in \cS_{a,b}(S)$,\;the projection of  $\ker(g_{a,b})$ to the  $i$-component $\homo_E(\gr_H^{-\bh_{a,\sigma}},\gr_H^{-\bh_{b,\sigma}})$ is surjective if only if $|\cS_{a,b}(S)|>1$.\;Therefore,\;for $i=i_{l_j}$,\;we see that $g_{i,\sigma}(\ker (g_{\sigma}))$ is equals to
\begin{equation*}
	\begin{aligned}
		&\;\bigoplus_{\substack{1\leq a<b\leq n\\i\in \cS_{a,b}(S),|\cS_{a,b}(S)|>1}}\homo_E\big(\gr_H^{-\bh_{a,\sigma}},\gr_H^{-\bh_{b,\sigma}}\big)=\bigoplus_{b=i+1}^n\bigoplus_{\substack{\bd_i+1\leq a\leq i\\ |\cS_{a,b}(S)|>1}}\homo_E\big(\gr_H^{-\bh_{a,\sigma}},\gr_H^{-\bh_{b,\sigma}}\big)\\
       		=&\;\homo_E\Big(\fil_H^{-\bh_{\bd_i+1,\sigma}}/\fil_H^{-\bh_{i+1,\sigma}},\fil_H^{-\bh_{i+1,\sigma}}\Big)\big/\bigoplus_{(a,b)\in\cJ^1_i}\homo_E\big(\gr_H^{-\bh_{a,\sigma}},\gr_H^{-\bh_{b,\sigma}}\big)\\
       	=&\;\Big\{f_i\in\homo_E\Big(\fil_H^{-\bh_{\bd_i+1,\sigma}}/\fil_H^{-\bh_{i+1,\sigma}},\fil_H^{-\bh_{i+1,\sigma}}\Big):f_i\Big(\fil_H^{-\bh_{{\max\{i_{l_{j-1}},\bd_i\}}+1,\sigma}}/\fil_H^{-\bh_{i+1,\sigma}}\Big)\subseteq\fil_H^{-\bh_{i_{l_{j+1}}+1,\sigma}}\Big\}.
	\end{aligned}
\end{equation*}
by Lemma \ref{lemforcardone}.\;We completes the proof.\;
\end{proof}

\begin{thm}\label{Hodgeparadeter}$\ker(t^{\flat}_{D_{\sigma}})_S$ determines (recall that $S:=\{i_{l_1}<i_{l_2}<\cdots<i_{l_{|S|}}\}$)
	\begin{equation}
	\begin{aligned}
		\Big\{\fil_H^{-\bh_{{\max\{i_{l_{j}},\bd_{i_{l_t}}\}}+1,\sigma}}(D_{\sigma})\oplus D_{\sigma}/\fil_H^{-\bh_{\bd_{i_{l_t}},\sigma}}(D_{\sigma})\Big\}_{1\leq j\leq |S|}
	\end{aligned}
\end{equation}	
In particular,\;if $\bd_{i_{l_t}}=0$,\;then $\ker(t^{\flat}_{D_{\sigma}})_S$ determines $\{\fil_H^{-\bh_{i_{l_{j}}+1,\sigma}}(D_{\sigma})\}_{1\leq j\leq |S|}$.\;
\end{thm}
\begin{proof}
	Let $t=|S|$.\;For $0\leq j\leq t-1$,\;put $S_j:=\{i_{l_1},i_{l_2},\cdots,i_{l_j},i_{l_t}\}$.\;Note that $S_0=\{i_{l_t}\}$ and $S_{t-1}=S$.\;Similar to the Step 2 of the proof of \cite[2.3.10]{BDcritical25} and Proposition \ref{propfor2ndfil},\;the following data is determined by  $\ker(t^{\flat}_{D_{\sigma}})_S$:
	\begin{equation}
		\begin{aligned}
			&\Big\{\homo_E\Big(\fil_H^{-\bh_{\bd_{i_{l_t}}+1,\sigma}}/\fil_H^{-\bh_{{\max\{i_{l_{j}},\bd_{i_{l_t}}\}}+1,\sigma}},\fil_H^{-\bh_{i_{l_t}+1,\sigma}}\Big)\Big\}_{0\leq j\leq t-1}
		\end{aligned}
	\end{equation}
(note that $i_0=0$) and 
	\begin{equation}
	\begin{aligned}
		&\homo_E\Big(\fil_H^{-\bh_{\bd_{i_{l_t}}+1,\sigma}}/\fil_H^{-\bh_{i_{l_{t}}+1,\sigma}},\fil_H^{-\bh_{i_{l_t}+1,\sigma}}\Big).
	\end{aligned}
\end{equation}
By the Lemma \ref{lem:recover-subspace-relative},\;$\ker(t^{\flat}_{D_{\sigma}})_S$ determines 
	\begin{equation}
	\begin{aligned}
		\Big\{\fil_H^{-\bh_{{\max\{i_{l_{j}},\bd_{i_{l_t}}\}}+1,\sigma}}\oplus D_{\sigma}/\fil_H^{-\bh_{\bd_{i_{l_t}},\sigma}}\Big\}_{1\leq j\leq t}.
	\end{aligned}
\end{equation}
We complete the proof.\;
\end{proof}

\begin{cor}\label{Hodgeparadetercor} For $1\leq t\leq |\Delta'|$,\;$\ker(t^{\flat}_{D})$ determines  (recall that $\Delta':=\{i_1<i_2<\cdots<i_{|\Delta'|}\}$)
	\begin{equation}
		\begin{aligned}
			\Big\{\fil_H^{-\bh_{{\max\{i_{{j}},\bd_{i_t}\}}+1,\sigma}}(D_{\sigma})\oplus D_{\sigma}/\fil_H^{-\bh_{\bd_{i_t},\sigma}}(D_{\sigma})\Big\}_{1\leq j\leq t}
		\end{aligned}
	\end{equation}	
	In particular,\;if $\bd_{i_{t}}=0$,\;then $\ker(t^{\flat}_{D_{\sigma}})$ determines $\{\fil_H^{-\bh_{i_{j}+1,\sigma}}(D_{\sigma})\}_{1\leq j\leq t}$.\;
\end{cor}

\begin{lem}\label{lem:recover-subspace-relative}
	Let $H'\subseteq H\subseteq D$ be subspaces of a finite-dimensional $E$-vector space $D$, and let $F$ be a
	nonzero $E$-vector space.\;Fix a quotient map $D\twoheadrightarrow H$ and consider the following composition 
	\[q:D\twoheadrightarrow H\twoheadrightarrow H/H'.\]
	Let $W:=\{\,f\circ q:f\in\Hom_E(H/H',F)\,\}
	\subseteq\Hom_E(D,F)$.\;Then
	$	\bigcap_{\varphi\in W}\ker(\varphi)=q^{-1}(0)=H'+\ker(D\twoheadrightarrow H)$ 
	and $H'=H\cap\bigcap_{\varphi\in W}\ker(\varphi).$

\end{lem}

\begin{proof}
	Since every element of $W$ factors through $q$, we have $q^{-1}(0)\subseteq\ker(\varphi)$ for every $\varphi\in W$.\;Hence $q^{-1}(0)\subseteq
	\bigcap_{\varphi\in W}\ker(\varphi).$\;Conversely,\;let $x\notin q^{-1}(0)$.\;Then $q(x)\neq0$ in $H/H'$.\;Since $F\neq0$,\; choose a nonzero vector $y\in F$ and a linear functional $	\lambda\in(H/H')^\vee$	such that $	\lambda(q(x))\neq0$.\;Define
	\[
	f:H/H'\longrightarrow F,\qquad
	\bar h\longmapsto\lambda(\bar h)y.
	\]
	Then $\varphi:=f\circ q\in W$ and $	\varphi(x)=\lambda(q(x))y\neq0$.\;Hence $x\notin
	\bigcap_{\varphi\in W}\ker(\varphi).$.\;We completes the proof.\;
\end{proof}



\begin{exmp}
A typical example is $\bL_{S_0}\cong \GL_1^{\oplus (n-m)}\times \GLN_m$.\;As $E$-vector spaces,\;$D_{\sigma}=D_{\sigma}^1\oplus\cdots D_{\sigma}^{n-m}\oplus D_{\sigma}^{n-m+1}$ with $\dim_ED_{\sigma}^j=1$ for $1\leq j\leq n-m$ and $\dim_ED
_{\sigma}^{n-m+1}=m$.\;For $i\in \Delta'$,\;we have
\begin{equation*}
	D_{\sigma,(i)}=\left\{
	\begin{array}{ll}
		D_{\sigma}/D_{\sigma}^{n-m+1},\;&n-i<m\\
		D_{\sigma} ,\;&n-i\geq m
	\end{array}
	\right.,\;\bd_i=\left\{
	\begin{array}{ll}
		m,\;&n-i<m\\
		0,\;&n-i\geq m
	\end{array}
	\right.,\;
\end{equation*}
We divide into two cases.\;
\begin{itemize}
	\item[(1)] $m\leq n-m$.\;In this case,\;$\Delta'=\Delta$.\;For $n-i\geq m$,\;$\ker(t^{\flat}_{D_{\sigma}})$ determines $\{\fil_H^{-\bh_{j+1,\sigma}}(D_{\sigma})\}_{1\leq j\leq i}$.\;If $n-i<m$,\;$\ker(t^{\flat}_{D_{\sigma}})$ determines 
	\begin{equation*}
		\begin{aligned}
			\Big\{\fil_H^{-\bh_{j+1,\sigma}}(D_{\sigma})\oplus D_{\sigma}/\fil_H^{-\bh_{m,\sigma}}(D_{\sigma})\Big\}_{m \leq j\leq i}\cup \Big\{\fil_H^{-\bh_{m+1,\sigma}}(D_{\sigma})\oplus D_{\sigma}/\fil_H^{-\bh_{m,\sigma}}(D_{\sigma})\Big\}.
		\end{aligned}
	\end{equation*}	
	\item[(2)] $m>n-m$.\;In this case,\;$\Delta'=\Delta\backslash \{n-m+1,\cdots,m-1\}$.\;For $n-i\geq m$ (so that $i\leq n-m$ and $i\in \Delta'$),\;$\ker(t^{\flat}_{D_{\sigma}})$ determines $\{\fil_H^{-\bh_{j+1,\sigma}}(D_{\sigma})\}_{1\leq j\leq i}$.\;If $i>m-1$,\;$\ker(t^{\flat}_{D_{\sigma}})$ determines 
	\begin{equation*}
		\begin{aligned}
			\Big\{\fil_H^{-\bh_{j+1,\sigma}}(D_{\sigma})\oplus D_{\sigma}/\fil_H^{-\bh_{m,\sigma}}(D_{\sigma})\Big\}_{m \leq j\leq i}\cup \Big\{\fil_H^{-\bh_{m+1,\sigma}}(D_{\sigma})\oplus D_{\sigma}/\fil_H^{-\bh_{m,\sigma}}(D_{\sigma})\Big\}.
		\end{aligned}
	\end{equation*}	
\end{itemize}
\end{exmp}

\subsection{Universal extensions}\label{sectionforunverext}

Let $R_{\Dpik}$ be the universal deformation ring of deformations of $\Dpik$ over $\Art_E$.\;For $?\in\{\flat,0\}$,\;let $R^{?}_{\Dpik,u}$ be the universal deformation ring that pro-represents the functor $\cF_{\Dpik,\cF_u}^{?}$ over $\Art_E$,\;and let $R_{\Dpik,g}$ be the universal deformation ring of de Rham deformations.\;Let $I^{?}_u$ (resp.,\;$I_{g}$) be the kernel of $R_{\Dpik}\rightarrow R^{?}_{\Dpik,u}$ (resp.,\;$R_{\Dpik}\rightarrow R_{\Dpik,g}$).\;Put $R^{?}_{\Dpik,u,g}=R_{\Dpik}/(I_g+I^{?}_u)$.\;We have natural surjections $R_{\Dpik}\twoheadrightarrow R^{?}_{\Dpik,u}\rightarrow  R^{?}_{\Dpik,u,g}$.\;For a continuous character $\delta$ of $\bZ_{S^u_0}(L)$,\;denote by the $R_{S_0^u,\delta}$ (resp.\;$R_{S_0^u,\delta,g}$) the universal deformations ring (resp.,\;the universal deformations ring of locally algebraic deformations) of $\delta$ over $\Art_E$.\;We have  a natural surjection $R_{S_0^u,\delta}\twoheadrightarrow R_{S_0^u,\delta,g}$.\;

For a complete local Noetherian $E$-algebra $R$,\;we use $\fm_R$ (or use $\fm$ for simplicity when it does not cause confusion) to denote its maximal ideal.\;Then for $\ast\in\{\emptyset,g,u\}$ and $?\in\{\flat,0,\emptyset\}$,\;we have  $(\fm_{R^{?}_{\Dpik,\ast}}/\fm^2_{R^{?}_{\Dpik,\ast}})^{\vee}\cong \ext^{1,?}_{\ast}(\Dpik,\Dpik)$.\;Moreover,\;we have
\[(\fm_{R_{S_0^u,\delta}}/\fm^2_{R_{S_0^u,\delta}})^{\vee}\cong \homo({\bZ_{S^u_0}(L)},E),(\fm_{R_{S_0^u,\delta,g}}/\fm^2_{R_{S_0^u,\delta,g}})^{\vee}\cong \homo_{\mathrm{sm}}({\bZ_{S^u_0}(L)},E).\]
For $u\in \sW_s$,\;we have the following commutative Cartesian diagram over $\Art_E$ (for $?\in\{\flat,0\}$):
\begin{equation}
	\xymatrix{
		R_{S_0^u,1}/\fm^2 \ar@{->>}[r] \ar[d]^{\kappa^{\flat}_u} & R_{S_0^u,1,g}/\fm^2 \ar[d]^{\kappa^{\flat}_u} 
		\\
		R^{\flat}_{\Dpik,u}/\fm^2 \ar@{->>}[r]	&  R^{\flat}_{\Dpik,u,g}/\fm^2},\quad \xymatrix{
		R_{S_0^u,1}/\fm^2 \ar@{->>}[r] \ar[d]^{\kappa^{0}_u} & R_{S_0^u,1,g}/\fm^2 \ar[d]^{\kappa^{0}_u} 
		\\
		R^{0}_{\Dpik,u}/\fm^2 \ar@{->>}[r]	&  R_{\Dpik,g}/\fm^2}.
\end{equation}
Let $\FZ_{\Omega}$ be the Bernstein centre over $E$ associated to $\mathrm{PS}^{\infty}_{1}(\underline{x})$ of $G$.\;Recall that $\FZ_{\Omega_{S^u_0}}$ is the Bernstein centre over $E$ associated to $\pi_{\sm}(\underline{x}^u)$ of $\bL_{S_0^u}(L)$.\;Let $\widehat{\FZ}_{\Omega,\underline{x}}$ (resp.,\;$\widehat{\FZ}_{\Omega_{S_0^u},\underline{x}^u}$) be the completion of $\FZ_{\Omega}$ (resp.,\;$\FZ_{\Omega_{S^u_0}}$) at $\mathrm{PS}^{\infty}_{u}(\underline{x})$ (resp.,\;$\pi_{\sm}(\underline{x}^u)$).\;We have natural isomorphisms $\widehat{\FZ}_{\Omega,\underline{x}}\xrightarrow{\sim} \widehat{\FZ}_{\Omega_{S_0^u},\underline{x}^u}\xrightarrow{\sim} R_{S_0^u,1,g}$ (note that $\widehat{\FZ}_{\Omega,\underline{x}}\cong R_{S_0^u,1,g}\cong R_{S_0,1,g}$ by sending a smooth deformation $\chi$ of $1$ to $\big(\ind^G_{\op_{S^u_0}(L)}\pi_{\sm}(\underline{x}^u)\eta_{S^u_0}\chi\big)^{\infty}$,\;see Lemma \ref{smhomoindepent}).\;Moreover,\;by Lemma \ref{smhomoindepent},\;we see that the composition 
\[A_0:=\widehat{\FZ}_{\Omega,\underline{x}}/\fm^2\xrightarrow{\sim} R_{S_0^u,1,g}\hookrightarrow R^{\flat}_{\Dpik,u,g}/\fm^2\]
is independent of the choice of $u$.\;For $?\in\{\emptyset,\flat,0\}$,\;let 
\[A^{?}_{\Dpik}:=R^{?}_{\Dpik}/\fm^2\times_{R_{\Dpik,g}} A_0,\text{\;and\;}A^{?}_{\Dpik,u}:=R^{?}_{\Dpik,u}/\fm^2\times_{R^{\flat}_{\Dpik,u,g}} A_{0}.\;\]
Then the tangent space of $A^{?}_{\Dpik}$ (resp.,\;$A^{?}_{\Dpik,u}$) is naturally isomorphism to $\overline{\ext}^{1,?}(\Dpik,\Dpik)$ (resp.,\;$\overline{\ext}^{1,?}_u(\Dpik,\Dpik)$), and the tangent space of $A^{?}_{\Dpik,u}$ is naturally isomorphic to $\homo\big(\bZ_{S^u_0}(L),E\big)$.\;Let $\cI^{?}_{u}$ be the kernel of the natural morphism $A_{\Dpik}\rightarrow A^{?}_{\Dpik,u}$.\;Let $\cI^{?}_{\Dpik}$ be the kernel of the natural morphism $A_{\Dpik}\rightarrow \prod_{u\in \sW_s}A^{?}_{\Dpik,u}$,\;then we get an injection $A^{?}_{\Dpik}:=A_{\Dpik}/\cI^{?}_{\Dpik}\hookrightarrow\prod_{u\in \sW_s}A^{?}_{\Dpik,u}$ (In particular,\;the infinite fern in Theorem \ref{infinitePoten}   deduces that $I_{\Dpik}=0$).\;

Let $\pi^{\flat}_{1}(\underline{x},\bh)^{\univ}$ (resp.,\;$\pi^{\flat}_{1}(\underline{x},\bh)_u^{\univ}$) be the universal tautological extension of
\[\ext^1_{G}\big(\pi_{\lalg}(\underline{x},\bh),\pi^{\flat}_1(\underline{x},\bh)\big)\otimes_E \pi_{\lalg}(\underline{x},\bh),\;\quad(\text{resp.,\;}\ext^1_{G}\left(\pi_{\lalg}(\underline{x},\bh),\mathrm{PS}_{u,1}(\underline{x},\bh)\right)\otimes_E \pi_{\lalg}(\underline{x},\bh))\]
 by $\pi^{\flat}_{1}(\underline{x},\bh)$.\;Let $\widetilde{\delta}_u^{\univ}$ be the universal extension of $\homo(\bZ_{S^u_0}(L),E)\otimes_E\delta_u$ by $\delta_u$.\;Similar to the proof of \cite[Lemma 3.35]{ParaDing2024},\;the induced representation $I^G_{\op_{S^u_0}(L)}\pi_{\sm}(\underline{x}^u)\eta_{S^u_0}\widetilde{\delta}_u^{\univ}$  is the universal extension of $\homo(\bZ_{S^u_0}(L),E)\otimes \pi_{\lalg}(\underline{x},\bh)$ by $\mathrm{PS}_{u,1}(\underline{x},\bh)$.\;We have hence  isomorphisms of $G$-representations
\begin{equation}
	\begin{aligned}
		&\left(I^G_{\op_{S_0}(L)}\pi_{\sm}(\underline{x}^u)\eta_{S^u_0}\widetilde{\delta}_u^{\univ}\right)\oplus_{\mathrm{PS}_{u,1}(\underline{x},\bh)}\pi^{\flat}_{1}(\underline{x},\bh)\xrightarrow{\sim }\pi^{\flat}_{1}(\underline{x},\bh)_u^{\univ}.\;
	\end{aligned}
\end{equation}
For $?\in\{\flat,0\}$,\;we have a natural action of $A^{?}_{\Dpik,u}\hookrightarrow R_{S_0^u,1}/\fm^2$ on $\widetilde{\delta}_u^{\univ}$ where $y\in \fm_{A^{?}_{\Dpik,u}}/\fm^2_{A^{?}_{\Dpik,u}}\cong \homo\big(\bZ_{S^u_0}(L),E\big)^{\vee}$ acts via
$y:\widetilde{\delta}_u^{\univ}\twoheadrightarrow \homo\big(\bZ_{S^u_0}(L),E\big)\otimes_E\delta_u\xrightarrow{y}\delta_u\hookrightarrow \widetilde{\delta}_u^{\univ}$.\;Similarity,\;$\pi^{\flat}_{1}(\underline{x},\bh)_u^{\univ}$ admits an $A^{?}_{\Dpik,u}$-action,\;which is given by 
\[y:\pi^{\flat}_{1}(\underline{x},\bh)_u^{\univ}\rightarrow \ext^1_{u}\left(\pi_{\lalg}(\underline{x},\bh),\pi^{\flat}_{1}(\underline{x},\bh)\right)\otimes \pi_{\lalg}(\underline{x},\bh)\xrightarrow{y}\pi_{\lalg}(\underline{x},\bh)\hookrightarrow \pi^{\flat}_{1}(\underline{x},\bh)_u^{\univ},\]
where $y\in \fm_{A^{?}_{\Dpik,u}}/\fm^2_{A^{?}_{\Dpik,u}}\cong \homo\big(\bZ_{S^u_0}(L),E\big)^{\vee}\xrightarrow{\zeta_u}\big(\zeta_u\big(\homo\big(\bZ_{S^u_0}(L),E\big)\big)^{\vee}$.\;The injection
\[I^G_{\op_{S^u_0}(L)}\pi_{\sm}(\underline{x}^u)\eta_{S^u_0}\widetilde{\delta}_u^{\univ}\hookrightarrow \pi^{\flat}_{1}(\underline{x},\bh)_u^{\univ}\] is $A^{?}_{\Dpik,u}$-equivalent.\;

\begin{thm}\label{actionADonuniversal}
For $?\in\{\emptyset,\flat,0\}$,\;there is a unique $A^{?}_{\Dpik}$-action (the $A_{\Dpik}$-action factors through the natural surjections $A_{\Dpik}\twoheadrightarrow  A^0_{\Dpik}\twoheadrightarrow  A^{\flat}_{\Dpik}$) on $\pi^{\flat}_{1}(\underline{x},\bh)^{\univ}$ such that for all $u\in \sW_s$,\;we have an $A^{?}_{\Dpik,u}\times G$-equivariant injection
	\[\pi^{\flat}_{1}(\underline{x},\bh)_u^{\univ}\hookrightarrow \pi^{\flat}_{1}(\underline{x},\bh)^{\univ}[\cI^{?
	}_u].\]
\end{thm}
\begin{proof}
	By Theorem \ref{mainthmforfullrefine},\;we define an $A_{\Dpik}$-action on $\pi^{\flat}_{1}(\underline{x},\bh)^{\univ}$ by letting $y\in \fm_{A_{\Dpik}}/\fm^2_{A_{\Dpik}}\cong \overline{\ext}^1(\Dpik,\Dpik)\rightarrow \ext^1_{G}\left(\pi_{\lalg}(\underline{x},\bh),\pi^{\flat}_1(\underline{x},\bh)\right)^{\vee}$ act via
	\[y:\pi^{\flat}_{1}(\underline{x},\bh)^{\univ}\rightarrow \ext^1_{G}\left(\pi_{\lalg}(\underline{x},\bh),\pi^{\flat}_1(\underline{x},\bh)\right)\otimes_E\pi_{\lalg}(\underline{x},\bh)\xrightarrow{y} \pi_{\lalg}(\underline{x},\bh)\hookrightarrow \pi^{\flat}_{1}(\underline{x},\bh)^{\univ}.\;\]
By definition,\;such $A_{\Dpik}$-action factors through quotients $A_{\Dpik}\twoheadrightarrow  A^0_{\Dpik}\twoheadrightarrow  A^{\flat}_{\Dpik}$.\;This action satisfies the property in the theorem.\;The uniqueness follows from the fact that $\pi^{\flat}_{1}(\underline{x},\bh)^{\univ}$ is generated by $\pi^{\flat}_{1}(\underline{x},\bh)_u^{\univ}$.\;
\end{proof}

\section{Local-global compatibility}

\subsection{Patched Bernstein eigenvarieties and Bernstein parabolic varieties}\label{BENVARPARAVAR}

In this section,\;we recall briefly  the patched Bernstein eigenvariety  and Bernstein paraboline variety  of Breuil-Ding (see \cite[Section 3.3,\;Section 4.2]{Ding2021}).\;

We use the following patched setting and follow the notation of \cite[Section 4.1.1]{2019DINGSimple} and \cite[Section 2]{PATCHING2016}.\;We have a patched Galois deformation ring $R_{\infty}=R_{\infty}^{\fp}\widehat{\otimes} \defvarring$,\;where $\fp$ is a $p$-adic place over a totally real field $F^+$,\;$L:=F^+_{\fp}$,\;and $\overline{r}: \gal_L \rightarrow \GLN_n(k_E)$ is a continuous representation and  $\defvarring$ is  the maximal reduced and $p$-torsion free quotient of the universal $\co_E$-lifting ring of $\overline{r}$.\;We have a $R_{\infty}$-admissible unitary representation $\Pi_{\infty}$ (i.e.,\;the so-called patched Banach representation) of $G$ over $E$.\;We refer to \textit{loc.cit} for detail.\;Let $\Pi_\infty^{R_\infty-\ana}$ be the subrepresentation of $G$ of locally $R_\infty$-analytic vectors of $\Pi_\infty$ (see \cite[Section 3.1]{breuil2017interpretation}).\;

Put $\FX_{\infty}:=(\Spf R_{\infty})^{\rig}$,\;$\FX_{\infty}^{{\fp}}:=(\Spf R^{\fp}_{\infty})^{\rig}$ and  $\mathfrak{X}_{\overline{r}}=(\spf\;R_{\overline{r}}^{\square})^{\rig}$.\;We have thus
$\FX_{\infty}:=\FX_{\infty}^{{\fp}}\times \mathfrak{X}_{\overline{r}}$.\;As in Section \ref{Omegafil},\;for $u\in\sW_s$,\;we fix a cuspidal Bernstein component $\Omega_{S^u_0}$ of the Levi subgroup $\bL_{S^u_0}$.\;Let $\bh:=(\hpi_{1,\sigma},\hpi_{2,\sigma},\cdots,\hpi_{n,\sigma})_{\sigma\in \Sigma_L}\in X_{\Delta}^+$.\;We put $\hpi_{i}=(\hpi_{i,\sigma})_{\sigma\in \Sigma_L}$ for $1\leq i\leq n$ and put ${\bm\lambda}_\bh=(\hpi_{i,\sigma}+i-1)_{\sigma\in \Sigma_L,1\leq i\leq n}=\bh-\theta$.\;Let 
\[\mathcal{E}_{\omepik,\fp,{\bm\lambda}_\bh}^\infty(\overline{\rho})\hookrightarrow\FX_{\infty}\times\sbanpiku\times \rigchu\] 
be the patched Bernstein eigenvariety associated to $\Pi_{\infty}$ and  $\Omega_{S_0^u}$ (see \cite[Section 4.1]{He20222},\;this is an easy variation of the patched eigenvariety introduced in  \cite{Ding2021}).\;There exists a natural coherent sheaf $\cM_{\omepik,{\bm\lambda}_\bh}^{\infty}$ over $\FX_{\infty}\times\sbanpiku\times \rigchu$ such that we have a $\bZ_{S_0^u}(L)\times R_{\infty}$-equivalent isomorphism
\[\Gamma\Big(\FX_{\infty}\times\sbanpiku\times \rigchu,\cM_{\omepik,{\bm\lambda}_\bh}^{\infty}\Big)\cong B_{\omepik,{\bm\lambda}_\bh}(\Pi_\infty^{R_\infty-\ana})^\vee,\]
and $\mathcal{E}_{\omepik,\fp,{\bm\lambda}_\bh}^\infty(\overline{\rho})$ is
defined as the Zariski-closed support of $\cM_{\omepik,{\bm\lambda}_\bh}^{\infty}$ in $\FX_{\infty}\times\sbanpiku\times \rigchu$.\;In particular,\;for $y=(\rho_y, \pi_{y},\chi_y)=(\rho_{y}^{\fp},\rho_{y,\fp},\pi_{y},\chi_y)\in \FX_{\infty}^{{\fp}}\times \mathfrak{X}_{\overline{r}}\times\sbanpiku\times \rigchu$,\;$y\in \mathcal{E}_{\omepik,\fp,{\bm\lambda}_\bh}^\infty(\overline{\rho})$ if and only if
\[\homo_{\bL_{S^u_0}(L)}\Big(\pi_{y}\otimes_E\big((\chi_y)_{\varpi_{L}}\circ\mathrm{ det}_{\bL_{S^u_0}(L)}\big)\otimes_E L_{S^u_0}({\bm\lambda}_\bh),J_{\bP_{S^u_0}(L)}\big(\Pi_\infty^{R_\infty-\ana}[\fm_y] \big)\Big)\neq0.\]
where $\fm_y=(\fm_y^{\fp},\fm_{y,\fp})$ is the corresponding maximal ideal of $R_{\infty}[1/p]$.\;

Let $\defvaru\hookrightarrow\mathfrak{X}_{\overline{r}}^\Box\times \sbanpiku\times\rigchu$ be the  Bernstein paraboline varieties \cite[Section 4.2]{Ding2021} of type {$(\omepik,\bh)$},\;which parametrizes local Galois representations that admit $\Omega_{S_0^u}$-filtrations.\;By \cite[Theorem 4.2.5,\;Corollary 4.2.5]{Ding2021} or \cite[Proposition 4.2]{He20222},\;the rigid space $\defvaru$ is equidimensional of dimension $n^2+\left(\frac{n(n-1)}{2}+s\right)d_L$.\;Note that ``$n^2$" comes from the framing.\;Let $x=(\rho_x,\underline{x},\undelram)\in \defvaru$,\;then $D_{\rig}(\rho_x)$ admits an $\omepik$-filtration $\cF=\{\fil_i^\cF D_{\rig}(\rho_x)\}$ such that,\;for all $1=1,\cdots,s$,\;
\[\gr_{i}^{\cF}D_{\rig}(\rho_x)\otimes_{\cR_{k(x),L}}\cR_{k(x),L}((\delta_i^0)^{-1}_{\varpi_L})\big[1/t\big]=\Delta_{x_i}\big[1/t\big].\]
The Bernstein paraboline variety $\defvaru$ can be viewed as a local analogue of the patched Bernstein eigenvariety $\mathcal{E}_{\omepik,\fp,{\bm\lambda}_\bh}^\infty(\overline{\rho})$.\;Consider the composition
\begin{equation}\label{composition}
	\begin{aligned}
		\mathcal{E}_{\omepik,\fp,{\bm\lambda}_\bh}^\infty(\overline{\rho})\hooklongrightarrow \FX_{\overline{\rho}^{\fp}}^\Box&\times \mathfrak{X}_{\overline{r}}^\Box \times \sbanpiku\times \rigchu\\
		&\xrightarrow{\iota_{\omepik}} \FX_{\overline{\rho}^{\fp}}^\Box\times \mathfrak{X}_{\overline{r}}^\Box \times \sbanpiku\times \rigchu.
	\end{aligned}
\end{equation}
where $\iota_{\omepik}:\sbanpiku\xrightarrow{\sim}\sbanpiku$ is an isomorphism given in the argument before \cite[(3.32)]{Ding2021}.\;By an easy variation of the proof of \cite[Theorem 3.3.9]{Ding2021},\;the composition (\ref{composition}) factors through $\FX_{\overline{\rho}^{\fp}}^\Box\times \defvarrho$,\;i.e.,\;
\begin{equation}\label{mapbersteineigenvarpikdefvarrho}
	\Lambda:\mathcal{E}_{\omepik,\fp,{\bm\lambda}_\bh}^\infty(\overline{\rho})\hooklongrightarrow \FX_{\overline{\rho}^{\fp}}^\Box\times \iota_{\omepik}^{-1}(\defvaru).
\end{equation}	

\subsection{Main results on local-global compatibility}

This section follows the route of \cite[Section 4.1]{ParaDing2024}.\;Let $\rho\in \FX_{\infty}$ be a continuous Galois representation such that  the  following Hypothesis holds.\;Let $\fm_{\rho}=(\fm^{\fp},\fm_{\fp})$ be  the corresponding maximal ideal of $R_{\infty}[1/p]$.\;
\begin{hypothesis}\label{hyongaloisrep1}(Keep the notation in  Section \ref{Omegafil} and Definition \ref{noncriticalassumption})
	\begin{itemize}
		\item[(a)] $\rho_L:=\rho_{\fp}:=\rho|_{\gal_{F^+_{\fp}}}$ is a potentially crystalline  $p$-adic Galois representation with distinct Hodge-Tate weights $\bh:=(\hpi_{1}>\hpi_{2}>\cdots>\hpi_{n} )$.\;
		\item[(b)] Let $\Dpik:=D_{\rig}(\rho_L)$ be the associated $(\varphi,\Gamma)$-module over $\cR_{E,L}$ of rank $n$.\;For  $u\in \sW_s$,\;$\Dpik$ admits a non-critical generic $\omepik$-filtration $\cF_u$ with parameters  $(\underline{x}^u_+,\underline{\delta}^{0,u})\in \sbanpiku\times \rigchu$.\;
		\item[(c)] For $u\in \sW_s$,\;put $y_{u}:=(\rho,\iota^{-1}_{\omepik}(\underline{x}^u_+,\underline{\delta}^{0,u}))\in \FX_{\infty}\times\sbanpiku\times \rigchu$.\;Then $y_u\in \mathcal{E}_{\omepik,\fp,{\bm\lambda}_\bh}^\infty(\overline{\rho})$ (and thus $x_u:=(\rho_L,\underline{x}^u_+,\underline{\delta}^{0,u})\in \defvaru$).\;
	\end{itemize}
\end{hypothesis}

By \cite[Propositions 4.10 $\&$ 4.11]{He20222},\;we have 
\begin{lem}
$x_u$ (resp.,\;$y_{u}$) is a smooth point of $\defvaru$ (resp.,\;$\mathcal{E}_{\omepik,\fp,{\bm\lambda}_\bh}^\infty(\overline{\rho})$).\;Moreover,\;$y_{u}$ does not admit companion points of non-dominant weight.\;
\end{lem}
\begin{rmk}
By the argument in the proof of \cite[Theorem 4.35]{PATCHING2016},\;the coherent sheaf $\cM_{\omepik,{\bm\lambda}_\bh}^{\infty}$ is locally free of rank $1$ at all $y_{u}$.\;
\end{rmk}

Recall that the completion of $R_{\overline{r}}^{\square}[1/p]$ at $\rho_L$ is natural to $R^{\Box}_{\rho_L}\cong R^{\Box}_{\Dpik}$,\;where $R^{\Box}_{\rho_L}$ is the framed universal deformation ring of $\rho_L$ over $\Art_E$.\;Let $R^{\Box}_{\Dpik}:=R_{\Dpik}\otimes_{R_{\rho_L}}R^{\Box}_{\rho_L}$.\;Put $R_{\Dpik,u}^{\Box,?}:=R_{\Dpik}^{\Box}\widehat{\otimes}_ER^{?}_{\Dpik,u}$ for $u\in \sW_{s}$ and $?\in\{\flat,0\}$.\;Let $(\defvaru)^{\widehat{}}_{x_u}$ be the completion of $\defvaru$ at $x_u$.\;As $x_u$ is non-critical,\;by \cite[Proposition 6.4.5]{Ding2021} and \cite[Proposition 4.9]{He20222},\;$(\defvaru)^{\widehat{}}_{x_u}$ is naturally isomorphic to $R_{\Dpik,u}^{\Box,0}$.\;

Keep the notation in Section \ref{sectionforunverext}.\;Let $\fa\supseteq\fm^2_{R^{\Box}_{\Dpik}}$ be an ideal of $R^{\Box}_{\Dpik}$ satisfies that $\fa/\fm^2_{R^{\Box}_{\Dpik}}\oplus \fm_{A_{\Dpik}}/\fm^2_{A_{\Dpik}}\xrightarrow{\sim } \fm_{R^{\Box}_{\Dpik}}/\fm^2_{R^{\Box}_{\Dpik}}$.\;Then the composition $A_{\Dpik}\hookrightarrow R^{\Box}_{\Dpik}/\fm^2_{R^{\Box}_{\Dpik}}\twoheadrightarrow R^{\Box}_{\Dpik}/\fa$ is an isomorphism (i.e.,\;the ideal $\fa$ deletes the information comes from framing).\;For $u\in \sW_s$ and $?\in\{\emptyset,\flat,0\}$,\;we use $\fa$ to denote its image in $R^{\Box,?}_{\Dpik,u}$.\;Thus we get $\fa/\fm^2_{R^{\Box,?}_{\Dpik,u}}\oplus \fm_{A^{?}_{\Dpik,u}}/\fm^2_{A^{?}_{\Dpik,u}}\xrightarrow{\sim } \fm_{R^{\Box,?}_{\Dpik,u}}/\fm^2_{R^{\Box,?}_{\Dpik,u}}$ and $\fa/\fm^2_{R^{\Box,?}_{\Dpik,u,g}}\oplus \fm_{A_{0,u}}/\fm^2_{A_{0,u}}\xrightarrow{\sim } \fm_{R^{\Box,?}_{\Dpik,g}}/\fm^2_{R^{\Box,?}_{\Dpik,g}}$,\;so that the natural maps
$A^{?}_{\Dpik,u}\hookrightarrow R^{\Box,?}_{\Dpik,u}/\fm^2_{R^{\Box,?}_{\Dpik,u}}\twoheadrightarrow R^{\Box,?}_{\Dpik,u}/\fa$ and $A_{0,u}\hookrightarrow R^{\Box,?}_{\Dpik,u,g}/\fm^2_{R^{\Box,?}_{\Dpik,u,g}}\twoheadrightarrow R^{\Box,?}_{\Dpik,u,g}/\fa$ are isomorphisms.\;We use $\fa\subseteq R_{\overline{r}}^{\square}[1/p]$ denote the preimage of $\fa\subseteq R_{\overline{r}}^{\square}$.\;Let $\fm_{\rho}:=(\fm_{\fp},\fm^{\fp})\subseteq\fa_{\rho}:=(\fa,\fm^{\fp})\subseteq R_{\infty}[1/p]$.\;

Our goal is to investigate $\Pi_\infty^{R_\infty-\ana}[\fa_{\rho}]$,\;which is equipped with a natural $R^{\Box}_{\Dpik}/\fa\cong A_{\Dpik}$-action.\;By \cite[Proposition 4.33]{PATCHING2016},\;$\Pi_\infty^{R_\infty-\ana}[\fm_{\rho}]^{\lalg}\cong \pi_{\lalg}(\underline{x},\bh)$.\;Similar to \cite[Lemma 4.2]{ParaDing2024},\; $\fa+\fm_{A_{\Dpik}}=\fm_{R^{\Box}_{\Dpik}}$ and hence $\fa_{\rho}+\fm_{A_{\Dpik}}=\fm_{\rho}$,\;we get that 
$\Pi_\infty^{R_\infty-\ana}[\fm_{\rho}]=\Pi_\infty^{R_\infty-\ana}[\fa_{\rho}][\fm_{A_{\Dpik}}]$.\;Similar to \cite[Lemma 4.2 (2)]{ParaDing2024},\;we can prove that
\[\homo_G(\Pi_\infty^{R_\infty-\ana}[\fm_{\rho}]^{\lalg},\Pi_\infty^{R_\infty-\ana}[\fa_{\rho}])\cong \homo_G(\Pi_\infty^{R_\infty-\ana}[\fm_{\rho}]^{\lalg},\Pi_\infty^{R_\infty-\ana}[\fm_{\rho}])\cong E.\]
Let $U_u=U_{\fp,u}\times U^{\fp}_u\subseteq  \iota^{-1}_{\omepik}\big(\defvaru\big)\times\FX_{\overline{\rho}^{\fp}}^{\Box}$ be a smooth affinoid neighourhood of $y_u$ such $y_{u'}\not\in U_{\fp,u}$ for $u'\neq u$.\;Let $\fm_{y_{u,\fp}}$ be the maximal ideal of $\cO(U_{\fp,u})$ at $y_{u,\fp}$ and $\fa\supset \fm^2_{y_{u,\fp}}$ be the closed ideal generated by the above ideal $\fa\subseteq R_{\overline{r}}^{\square}[1/p]$.\;In the sequel,\;we put
\[\cM_{y_u}:=\cM_{\omepik,{\bm\lambda}_\bh}^{\infty}(U_u)/(\fm_{y_{u,\fp}}+\fm^{\fp}),\widetilde{\cM}_{y_u}:=\cM_{\omepik,{\bm\lambda}_\bh}^{\infty}(U_u)/(\fa+\fm^{\fp})\]
for simplicity.\;There are natural $\bZ_{S^u_0}(L)\times R_{\infty}$-equivariant injections:
\[\cM_{y_u}\hookrightarrow \widetilde{\cM}_{y_u}\hookrightarrow J_{\bP^u_{S_0}}(\Pi_\infty^{R_\infty-\ana})[\fa_{\rho}].\]
Since $\cM_{\infty}$ is locally free of rank $1$ at $y_{u}$,\;we obtain that $\widetilde{\cM}_{y_u}\cong R^{\Box,0}_{\Dpik,u}/\fa_{\rho}$,\;so that $\dim_E\widetilde{\cM}_{y_u}=1+(s+sd_L)$.\;In this case,\;similar to the argument around \cite[(4.5)]{ParaDing2024},\;we have a $\bZ_{S^u_0}(L)\times A^{0}_{\Dpik,u}$-equivariant isomorphism $\widetilde{\cM}^{\vee}_{y_u}\cong \pi_{\sm}(\underline{x}^u)\eta^{-1}_{S^u_0}\otimes_E\widetilde{\delta}_u^{\univ}$ (note that the $A^{0}_{\Dpik,u}$-action factors through $A^{0}_{\Dpik,u}\twoheadrightarrow A^{\flat}_{\Dpik,u}$).\;Similar to the proof of \cite[Lemma 4.19]{He20222},\;by using the non-critical assumption of $y_u$,\;the maps $\cM_{y_u}\hookrightarrow  \Pi_\infty^{R_\infty-\ana}[\fm_{\rho}]$ and $ \widetilde{\cM}_{y_u}\hookrightarrow \Pi_\infty^{R_\infty-\ana}[\fa_{\rho}]$ are balanced,\;hence induces a $G\times R_{\infty}$-equivariant injection:
\begin{equation}
	\begin{aligned}
		&\iota_u:I^G_{\op_{S^u_0}(L)}\pi_{\sm}(\underline{x}^u)\eta_{S^u_0}\widetilde{\delta}_u^{\univ}\hookrightarrow \Pi_\infty^{R_\infty-\ana}[\fa_{\rho}].\;
	\end{aligned}
\end{equation}
where the $R_{\infty}$-action on $I^G_{\op_{S^u_0}(L)}\pi_{\sm}(\underline{x}^u)\eta_{S^u_0}\widetilde{\delta}_u^{\univ}$ is induced by $R_{\infty}\rightarrow (R^{\Box,0}_{\Dpik,u}/\fa)\otimes_E (R^{\fp}_{\infty}[1/p]/\fm^{\fp})\xrightarrow{\sim}A_{\Dpik,u}^{0}\twoheadrightarrow (R^{\Box,\flat}_{\Dpik,u}/\fa)\otimes_E (R^{\fp}_{\infty}[1/p]/\fm^{\fp})\xrightarrow{\sim}A_{\Dpik,u}^{\flat}$.\;Moreover,\;we have an injection of locally analytic representations $\mathrm{PS}_{u}(\underline{x},\bh)\hookrightarrow I^G_{\op_{S^u_0}(L)}\pi_{\sm}(\underline{x}^u)\eta_{S^u_0}\widetilde{\delta}_u^{\univ}$.\;

Let $\widetilde{\pi}$ be the subrepresentation of $\Pi_\infty^{R_\infty-\ana}[\fa_{\rho}]$ generated by $\mathrm{Im}(\iota_u)$  for all $u\in \sW_s
$.\;Note that $\widetilde{\pi}$ has an $A_{\Dpik}$-action via $A_{\Dpik}\xrightarrow{\sim}R^{\Box}_{\Dpik}/\fa \xrightarrow{\sim}R_{\infty}[1/p]/\fa_{\rho}$.\;Similar to \cite[Lemma 4.4]{ParaDing2024},\;we have
\begin{pro}\label{imiotawmx}
 $\iota_u(\mathrm{PS}_{u,1}(\underline{x},\bh))=(\mathrm{Im}(\iota_u))[\fm_{\rho}]$.\;
 \end{pro}
\begin{proof}
Note that the composition $\mathrm{PS}_{u}(\underline{x},\bh)\hookrightarrow I^G_{\op_{S^u_0}(L)}\pi_{\sm}(\underline{x}^u)\eta_{S^u_0}\widetilde{\delta}_u^{\univ}\hookrightarrow \Pi_\infty^{R_\infty-\ana}[\fa_{\rho}]$ sends $\mathrm{PS}^{\lalg}_{u}(\underline{x},\bh)$ to $\Pi_\infty^{R_\infty-\ana}[\fm_{\rho}]$.\;Since  $\homo_{G}(\mathrm{PS}_{u,1}(\underline{x},\bh),\Pi_\infty^{R_\infty-\ana}[\fm_{\rho}])\xrightarrow{\sim}\homo_{G}(\mathrm{PS}^{\lalg}_{u}(\underline{x},\bh),\Pi_\infty^{R_\infty-\ana}[\fm_{\rho}])$, then the same strategy as in the first paragraph of the proof of \cite[Lemma 4.6]{ParaDing2024} shows that the previous composition has image in $\Pi_\infty^{R_\infty-\ana}[\fm_{\rho}]$.\;Using the same strategy as in \cite[Lemma 4.6]{ParaDing2024},\;the injection $\iota_u(\mathrm{PS}_{u,1}(\underline{x},\bh))\hookrightarrow(\mathrm{Im}(\iota_u))[\fm_{\rho}]$ is surjective too.\;
\end{proof}

\begin{thm}\label{mainthmglobal}We have an $A_{\Dpik}\times G$-equivariant isomorphism $\widetilde{\pi}\cong \pi^{\flat}_{1}(\underline{x},\bh)^{\univ}$.\;In particular,\;we have a $\GLN_n(L)$-equivariant injection
	\[\pi^{\flat}_{\min}(\Dpik)\cong \pi^{\flat}_{1}(\underline{x},\bh)^{\univ}[\fm_{A_{\Dpik}}]\cong \widetilde{\pi}[\fm_{A_{\Dpik}}]\hookrightarrow  \Pi_\infty^{R_\infty-\ana}[\fa_{\rho}][\fm_{A_{\Dpik}}]\cong \Pi_\infty^{R_\infty-\ana}[\fm_{\rho}].\]
\end{thm}
\begin{proof}
	Similar to the argument in the proof \cite[Theorem 4.5,\; Corollary 4.6]{ParaDing2024},\;we can also see that $\widetilde{\pi}$ is a subrepresentation of $\pi^{\flat}_{1}(\underline{x},\bh)^{\univ}$ and contains all $I^G_{\op_{S^u_0}(L)}\pi_{\sm}(\underline{x}^u)\eta_{S^u_0}\widetilde{\delta}_u^{\univ}$ so that $\widetilde{\pi}\cong \pi^{\flat}_{1}(\underline{x},\bh)^{\univ}$.\;
\end{proof}

\section{Appendix:\;envelope of intersections of Borel and parabolic subalgebras  in Lie algebras}\label{reproofforinfinite}

\begin{lem}\label{Inifinitefernpoten} Keep the non-critical assumption in Remark \ref{explainHomasLiealg}.\;Then $\sum_{u\in \sW_s}\mathrm{Ad}_{(u^{\sharp})^{-1}}(\fp_{S^u_0,\sigma})\cap \mathrm{Ad}_{g_{\sigma}}(\fb_{\sigma})=\mathrm{Ad}_{g_{\sigma}}(\fb_{\sigma})$.\;
\end{lem}
\begin{proof}Keep the notation in Remark \ref{explainHomasLiealg}.\;Note that $\mathrm{Ad}_{g_{\sigma}}(\fb_{\sigma})=\mathrm{Ad}_{b_{\sigma}w_{0,\sigma}}(\fb_{\sigma})$ and 
\[\sum_{u\in \sW_s}\mathrm{Ad}_{(u^{\sharp})^{-1}}(\fp_{S^u_0,\sigma})\cap \mathrm{Ad}_{b_{\sigma}w_{0,\sigma}}(\fb_{\sigma})=\sum_{u\in \sW_s}\mathrm{Ad}_{b_{\sigma}}(\mathrm{Ad}_{b_{\sigma}^{-1}(u^{\sharp})^{-1}}(\fp_{S^u_0,\sigma})\cap \mathrm{Ad}_{w_{0,\sigma}}(\fb_{\sigma})).\;\]
Therefore,\;it suffices to show that 
 \[\mathrm{Ad}_{w_{0,\sigma}}(\fb_{\sigma})=\sum_{u\in \sW_s}\mathrm{Ad}_{b_{\sigma}^{-1}(u^{\sharp})^{-1}}(\fp_{S^u_0,\sigma})\cap \mathrm{Ad}_{w_{0,\sigma}}(\fb_{\sigma}).\]
 Let $e^{i,j}$ be the elementary $(n\times n)$-matrix whose $(l,m)$-entry is given by $\delta_{i,l}\delta_{j,m}$.\;For any $i>j$,\;we claim that there is some permutation
$u_{i,j}\in\sW_s$ and $i-j$ scalars $\{x_{i,l}\}_{j+1\leq l\leq i}$ in $E$ such that the matrix $a^{i,j}:=e^{i,j}+\sum_{l=j+1}^ix_{i,l}e^{i,l}$ lies in $\mathrm{Ad}_{b_{\sigma}^{-1}(u^{\sharp})^{-1}}(\fp_{S^u_0,\sigma})$.\;Recall the previous basis 
$\{e_{1,j,\sigma}\}_{1\leq j\leq r_1},\cdots,\{e_{s,j,\sigma}\}_{1\leq j\leq r_s}$ (resp.,\;$e_{1,\sigma},\cdots,e_{n,\sigma}$).\;Recall that this basis gives a $\bP_{S^u_0}$-parabolic filtration 
\[V_{\bullet}=(0\subseteq V_0\subset V_1\subset\cdots\subset V_s=E^n),\]
 where $V_i=\langle \{e_{u^{-1}(1),j,\sigma}\}_{1\leq j\leq r_{u^{-1}(1)}},\cdots,\{e_{u^{-1}(i),j,\sigma}\}_{1\leq j\leq r_{u^{-1}(i)}}\rangle$.\;Let $s_{i,j}=(i,j)\in \sW_s$.\;Suppose that $t^u_{j'}+1\leq j\leq t^u_{j'+1}$ and $t^u_{i'}+1\leq j\leq t^u_{i'+1}$ for $\leq j'\leq i'\leq s$.\;Consider the following basis 
\[\cB=\{b^{-1}_{\sigma}(e_{1,\sigma}),\cdots,b^{-1}_{\sigma}(e_{j-1,\sigma}),e_{j,\sigma},b^{-1}_{\sigma}(e_{j+1,\sigma}),\cdots,b^{-1}_{\sigma}(e_{i,\sigma}),e_{i+1,\sigma},\cdots,e_{n,\sigma}\},\]
of $E^n$ (this is also a basis since $b^{-1}_{\sigma}$ is upper triangular).\;Define $\pi(x)=0$ for $x\in \cB\backslash\{e_{j,\sigma}\}$ and $\pi(e_{j,\sigma})=e_{i,\sigma}$.\;We also put $V'_l=\langle e_{1,\sigma},\cdots,e_{l,\sigma}\rangle$ for $1\leq l\leq n$.\;Note that
\[b_{\sigma}\pi b^{-1}_{\sigma}(s_{i,j}V_{l})=\left\{
\begin{array}{ll}
	0,\;&1\leq l \leq i-1,\\
	E\langle b_{\sigma}e_{i,\sigma}\rangle,\;&i\leq l \leq n.
\end{array}
\right.\]
Since $\langle b_{\sigma}e_{i,\sigma}\rangle\subseteq V'_i\subseteq s_{i,j}V_{i'}$ and $V_i\subseteq V_{i'}$,\;we find a morphism $\pi:E^n\rightarrow E^n$ such that the endomorphism $b_{\sigma}\pi b^{-1}_{\sigma}$ stabilizes the $\bP_{S^u_0}$-parabolic filtration $s_{i,j}V_{\bullet}$.\;Then the desired $a^{i,j}$ is the matrix in the standard basis.\;Similar to the last paragraph in \cite[Lemma 2.1]{DensityHMS},\;such $a^{i,j}$ also form a basis of the $\mathrm{Ad}_{w_{0,\sigma}}(\fb_{\sigma})$ which lie in the envelope $\sum_{u\in \sW_s}\mathrm{Ad}_{b_{\sigma}^{-1}(u^{\sharp})^{-1}}(\fp_{S^u_0,\sigma})\cap \mathrm{Ad}_{w_{0,\sigma}}(\fb_{\sigma})$.\;
\end{proof}


\bibliographystyle{plain}

\printindex
\end{document}